\documentclass[11pt,reqno]{amsart}

\usepackage[latin1]{inputenc}
\usepackage{amsmath,amsthm,amssymb,amscd,mathrsfs,amsbsy,ulem}
\usepackage{mathtools,color,esint,bm,float}
\usepackage{booktabs}
\usepackage{nicematrix}
\usepackage{tikz}
\usepackage[perpage]{footmisc}
\usepackage{subfig,caption}
\usepackage[shortlabels]{enumitem}
\usepackage{cancel}
\usepackage{bbm}

\newcolumntype{?}{!{\vrule width 0.4pt}}

\DeclareRobustCommand{\SkipTocEntry}[5]{}

\allowdisplaybreaks[4]

\definecolor{darkgreen}{rgb}{0.0, 0.6, 0.13}

\usepackage{cite}

\usepackage{hyperref}

\newtheorem{thm}{Theorem}[section]
 \newtheorem{cor}[thm]{Corollary}
 \newtheorem{lem}[thm]{Lemma}
 \newtheorem{prop}[thm]{Proposition}

 \theoremstyle{definition}
 \newtheorem{df}[thm]{Definition}
 \theoremstyle{remark}

 \numberwithin{equation}{section}

 \newtheoremstyle{remarkstyle}  
  {3pt}   
  {3pt}   
  {}      
  {}      
  {\bfseries} 
  {.}     
  {.5em}  
  {}      

\theoremstyle{remarkstyle} 
\newtheorem{remark}{Remark} %

\newcommand{\Bc}{\mathcal B}

\newcommand{\Fc}{\mathcal F}

\newcommand{\Qc}{\mathcal Q}
\newcommand{\Rc}{\mathcal R}

\newcommand{\vt}{\boldsymbol{t}}

\newcommand{\eq}{\begin{equation}}
\newcommand{\eeq}{\end{equation} }

\begin{document}

\title{Local-in-Time Existence of $L^1$ solutions to the Gravity Water Wave Kinetic Equation}

\author{Yulin Pan}
\address{\textsc{Department of Naval Architecture and Marine Engineering, University of Michigan, Ann Arbor, MI, USA}}
\email{\texttt{yulinpan@umich.edu}}

\author{Xiaoxu Wu}
\address{\textsc{Mathematical Sciences Institute, Australia National University}}
\email{\texttt{xiaoxu.wu@anu.edu.au}}

\date{}

 \begin{abstract}
In this paper, we study the Cauchy problem for the four-wave kinetic equation describing the weak turbulence of gravity water waves. The mathematical challenges of this analysis stem primarily from two interrelated aspects: (1) the extreme algebraic complexity of the collision kernel, where controlling its growth in the highly non-local regime constitutes the primary analytical bottleneck, and (2) the construction of strong solutions under the resulting singular integral operators. First, we re-analyze the interaction kernel in this precise regime, where the interacting wave numbers satisfy $|k|, |k_3| \gg |k_1|, |k_2|$. We establish a rigorous upper bound of $\mathcal{O}(|k||k_3|)$, which rigorously verifies the asymptotic smallness of the interaction coefficient anticipated in the physics literature \cite{zakharov2010energy, geogjaev2017numerical, geogjaev2025properties}. Furthermore, this result improves upon the recent $\mathcal{O}\big((|k||k_3|)^{3/2}\big)$ estimate proposed in \cite{waterkernel2024}, demonstrating a strictly milder singularity of wave interactions in this limit. Physically, this regime governs the energy exchange between disparate scales, such as the modulation of short gravity waves by long ocean swells. Second, leveraging this crucial integrability gain alongside a refined structural decomposition of the collision operator, we establish the local-in-time existence of $L^1$ strong solutions to the gravity water kinetic equation for initial data in a suitably weighted $L^2 \cap L^\infty$ space. Specifically, we prove that for any initial data in this class, the resulting $L^1$ strong solution strictly propagates the weighted $L^2 \cap L^\infty$ regularity and conserves the fundamental physical properties of the kinetic model. Our results provide a rigorous mathematical framework, laying a firm foundation for future studies on global dynamics and inter-scale energy cascades.

\end{abstract}

\maketitle

\section{Introduction}
Wave Turbulence Theory (WTT) serves as the foundational statistical framework for describing the interactions of a large number of weakly nonlinear dispersive waves. A cornerstone of this theory is the wave kinetic equation (WKE), which governs the long-time evolution of the wave action spectrum $f(k)$ in momentum space. One of the first WKEs was established by Klaus Hasselmann in 1962~\cite{HM1962} for surface gravity waves, sometimes referred to as Hasselmann's kinetic equation. This equation serves as the foundation of modern wave forecasting, in the sense that it is routinely solved numerically in conjunction with additional terms representing wind input and wave breaking to provide forecasts of global wave states. In addition, it was found by Zakharov~\cite{gkln2022,nazarenko2011,ZZ1982,K2012} that WKEs (such as Hasselmann's equation) yield power-law Kolmogorov-Zakharov solutions that can often be observed in experiments and field~\cite{lenain2017measurements,dyachenko2004weak,denissenko2007gravity,zhang2022numerical,pan2014direct,falcon2007observation}.

The mathematical landscape of water waves and WTT has witnessed spectacular breakthroughs, spanning from the foundational well-posedness theories of the Euler system to the rigorous justification of effective kinetic limits. The local well-posedness of the water wave system is now well-established across various physical scenarios, owing to the foundational contributions of many authors (see, for example, \cite{ABZ2014,CS2007,L2005,W1997,W1999}). Furthermore, for sufficiently small initial data, global existence and long-time dynamics have been rigorously established \cite{dip2017,GMS2012,ip2015,ip2018,it2017,w2011,AD2015}.

In recent years, derivation and justification of WKE has received intense mathematical attention, initially gaining momentum in the context of dispersive models such as the semilinear Schr\"odinger equation. The rigorous derivation of kinetic equations for weakly interacting dispersive waves has seen impressive advances. Following a series of pioneering works on the derivation of the 4-wave WKE from the cubic nonlinear Schr\"odinger equation (NLS) (see, e.g., \cite{BGHS2021,CG2025,DH2021}), a state-of-the-art milestone was achieved by Deng and Hani \cite{DH2023}. They rigorously derived the WKE up to the kinetic timescale and subsequently established its validity for as long as the kinetic equation itself remains well-posed \cite{dhlong2023}. Recently, Deng, Hani, and Ma \cite{dhm2025} successfully extended this robust program to interacting particle systems, providing a long-time derivation of the Boltzmann equation, again conditioned on its underlying well-posedness.

Crucially, these powerful derivation paradigms are now being adapted to the much more complex setting of genuine fluid interfaces. In \cite{DIP2025,DIP2025-2}, Deng, Ionescu, and Pusateri established long-time regularity results for the water wave system configured with large total energy but small local energy. This configuration represents the precise physical and mathematical setup required for applying WTT to water waves. For a broader context on other recent breakthroughs in rigorous derivations and WTT, we refer the reader to \cite{bgms2025,BGHS2018,DH2023,DH2026,dhm2024,dhm2025,DIP2025,ma2022,SBS2026,st2021,V2025,ST2020,EV2015,ampatzoglou2025rigorous,MR3987176}. At a fundamental structural level, the wave kinetic equations derived from WTT share profound mathematical similarities with the classical Boltzmann equation in statistical mechanics, as both govern the macroscopic evolution of complex interacting systems via non-local collision integrals. However, a stark contrast emerges regarding their regularizing effects and long-time asymptotic behaviors. For the Boltzmann equation, the intrinsic structure of the collision nonlinearity provides strong coercivity, often yielding powerful regularizing effects and an exponential decay toward the Maxwellian equilibrium (see, for instance, the seminal works on the exponential H-theorem and spectral gaps \cite{Villani2002_review, DV2005, MR3779780}). See also~\cite{MR2679369,MR2784329,MR4470411,MR2013332}. In contrast, the gravity water wave collision operator exhibits severe algebraic singularities. Furthermore, instead of an exponential Maxwellian, its exact thermodynamic equilibria are characterized by algebraic Rayleigh-Jeans distributions. Even more intricately, the system admits power-law Kolmogorov-Zakharov spectra as non-equilibrium stationary states driven by constant cascades of energy or wave action. It remains highly unobvious whether the water wave kinetic equation can offer analogous regularizing mechanisms or exponential relaxation.

While spectacular progress has recently been made in the rigorous justification of WTT for various dispersive systems, the intrinsic analytical properties of the corresponding WKE for gravity water waves remain largely unexplored. Studying the existence of strong solutions for this equation is therefore of fundamental importance to the broader theory of water wave turbulence. Establishing a rigorous well-posedness framework is essential to validate the physical predictions of the kinetic model, clarify the functional spaces governing the inter-scale energy cascades, and provide indispensable theoretical insights for future mathematical endeavors.

In this paper, we address this critical gap by rigorously constructing local-in-time $L^1$ strong solutions to the WKE for gravity water waves.

\subsection{Mathematical Setup and Main Result} Consider the spatially homogeneous kinetic equation for gravity water waves (setting the gravitational acceleration $g=1$) in dimension $d \geq 2$:
\begin{equation}
\label{4wave}\tag{WKE}
\begin{aligned}
\partial_t f(t,k) &= \mathcal{Q}[f](t,k), \quad \mbox{on } \mathbb{R}_+\times\mathbb{R}^d,\\
f(0,k) &= f_0(k)\in \Bc \quad \mbox{on } \mathbb{R}^d.
\end{aligned}
\end{equation}
We seek solutions within an admissible class $\mathcal{B}$ for non-negative initial data $f_0(k) \geq 0$. As is standard in kinetic theory, the system inherently preserves the real-valuedness and non-negativity of the solution throughout its evolution (see, e.g., \cite{nazarenko2011}). Let $\langle k\rangle := \sqrt{|k|^2+1}$ denote the Japanese bracket. For $1 \leq p \leq \infty$ and $s \geq 0$, we define the weighted Lebesgue space $L^p_s(\mathbb{R}^d)$ as the space of real-valued functions equipped with the norm $\| f \|_{p,s} := \| \langle k \rangle^s f \|_{L^p}$. The admissible class $\mathcal{B}$ is then defined as:
\begin{equation}\label{def: B}
\Bc=\big\{f\in L^2_{22+5d}(\mathbb R^d) \cap L^\infty_{12+4d}(\mathbb R^d) \,:\, f\geq 0 \big\},
\end{equation}
where the polynomial weights $22+5d$ and $12+4d$ are chosen to be sufficiently large to close our subsequent estimates. The collision integral $\mathcal{Q}[f]$, which models the four-wave resonant interactions, is given by:

\begin{align}
\mathcal{Q}[f](k) =  & \iiint_{\mathbb{R}^{3d}} T_{k,k_1,k_2,k_3} \delta(\Sigma)\delta(\Omega) \big[f_2 f_3 (f_1+f) - f f_1 (f_2 + f_3)\big]\,dk_1\,dk_2\, dk_3,\label{Qf}
\end{align}
with the standard shorthand notations for the dispersion relation and distribution functions:
$$
\omega= \omega(k)=\sqrt{|k|}, \qquad \omega_i = \omega(k_i), \qquad f = f(t,k), \qquad f_i = f(t,k_i), \quad \mbox{for } i=1, 2, 3.
$$
The resonant manifold is governed by the momentum and energy conservation laws:
$$
    \Sigma = k+k_1-k_2-k_3 = 0, \qquad \Omega = \omega+\omega_1-\omega_2-\omega_3 = 0.
$$
The collision kernel $T_{k,k_1,k_2,k_3} = |T_{k_1k}^{k_2k_3}|^2$ captures the transition probability of the four-wave interactions. The function $T_{k_1k}^{k_2k_3}$ is defined precisely in Eq.~\eqref{eq: def T}. Crucially, the kernel $T_{k,k_1,k_2,k_3}$ satisfies the standard scattering symmetries; it is invariant under the exchange of the incoming and outgoing pairs, as well as under the permutation of individual variables within each pair:
$$
    T_{k,k_1,k_2,k_3} = T_{k_2,k_3,k,k_1} = T_{k_1,k,k_2,k_3} = T_{k,k_1,k_3,k_2}.
$$

\medskip

Throughout this paper, all functional spaces are assumed to consist of real-valued functions unless stated otherwise. We adopt the following notational conventions for norms: $\|\cdot\|_{p}\equiv \|\cdot\|_{L^p}$, $\|\cdot\|_{p\to p}\equiv \|\cdot\|_{L^p\to L^p},$ $\|\cdot\|_{p,s}\equiv \|\cdot\|_{L^p_s}$, and $\|\cdot\|_{p,s\to p,s}:=\| \cdot\|_{L^p_s\to L^p_s }$.

\begin{remark}
The specific polynomial weights \(22+5d\) and \(12+4d\) in Eq.~\eqref{def: B} are not necessarily optimal. However, since the primary objective of this paper is to rigorously construct strong $L^1$ solutions to the system~\eqref{4wave} in the physical space \(L^1(\mathbb{R}^d)\), these conservative choices for the "good" initial data class $\mathcal{B}$ suffice to control the nonlinear collision operator.
\end{remark}

\medskip

The rigorous notion of a strong solution is specified as follows:
\begin{df}[Strong $L^1$ Solutions] 
We say that the system \eqref{4wave} admits a strong solution in $L^1$ if, for a given initial datum $f(0,k)=f_0$, there exists a lifespan $T=T(f_0)>0$ such that the solution $f(t,k)$ satisfies $f \in C([0,T]; L^1(\mathbb{R}^d))$. 
\end{df}

The main theorem of this paper will establish the local-in-time existence of such strong $L^1$-solutions to the system \eqref{4wave} for any initial data $f_0 \in \Bc$.
\begin{thm}\label{thm: LWP}Assume the dimension is $d \geq 2$. For any initial data $f_0 \in \Bc$, there exists a local existence time $T = T\big(\|f_0\|_{2,{22+5d}},\|f_0\|_{\infty,{12+4d}}\big) > 0$, depending strictly on the weighted norms of the initial data, such that the gravity water wave kinetic equation \eqref{4wave} admits a strong solution $f(t,k)$ in $L^1$ on the time interval $[0, T]$. Moreover, the solution propagates the initial weighted regularity, satisfying
\begin{equation}\label{eq: stay}
f \in L^\infty\big([0, T]; L^2_{22+5d}(\mathbb R^d)\big) \cap L^\infty\big([0, T]; L^\infty_{12+4d}(\mathbb R^d)\big).
\end{equation}
\end{thm}

The rigorous proof of Theorem~\ref{thm: LWP} is deferred to Section~4.

\subsection{Comparison with the NLS Kinetic Model and our approach}
To fully appreciate the severe difficulties inherent to the gravity water WKE, it is instructive to compare it with the kinetic equations derived from the NLS and the related Majda-McLaughlin-Tabak (MMT) models~\cite{GLZ2025,MMT1997}. For the NLS kinetic equation, the dispersion relation is given by $\omega(k) = |k|^2$. In the highly non-local regime ($|k|, |k_3| \gg |k_1|, |k_2|$), integrating out the resonant manifold constraint $\delta(\Sigma)\delta(\Omega)$ naturally yields a decay factor of $\mathcal{O}(|k|^{-1})$. Taken together, the local well-posedness results of Germain, Ionescu, and Tran \cite{GIT2020} in the bounded-kernel case and of Ampatzoglou and L\'eger \cite{AL2025} in the case of kernels with at most linear growth show that this decay rate is sufficient to compensate for the growth of the collision kernel, thereby allowing for the construction of strong solutions via contraction-based arguments.

Surprisingly, however, Ampatzoglou and L\'eger \cite{AL2025} rigorously proved that if the collision kernel grows strictly faster than $\mathcal{O}(|k|)$, the kinetic equation crosses the threshold into ill-posedness in weighted $L^\infty$ spaces. This phenomenon highlights a fundamental barrier: when the collision kernel grows too rapidly, standard analytical techniques completely fail to capture strong solutions in the same weighted $L^\infty$ space.

Turning to the genuine gravity water WKE (where we set the gravitational constant $g=1$ for simplicity), the analytical landscape becomes drastically more formidable. First, the gravity wave dispersion relation $\omega(k) = \sqrt{|k|}$ fails to yield the favorable $\mathcal{O}(|k|^{-1})$ decay upon integrating out the resonant manifold constraints encoded in $\delta(\Sigma)\delta(\Omega)$; indeed, a direct inspection of the resulting integral reveals no decay in $|k|$ whatsoever. Second, the algebraic structure of the water wave collision kernel is notoriously complicated, culminating in a profound singularity precisely in the most analytically challenging non-local regime, where the interacting wave numbers are widely separated in scale. In this extreme configuration, Korotkevich, Nazarenko, Pan, and Shatah \cite{waterkernel2024} suggested that the leading-order term of the kernel exhibits a severe $\mathcal{O}(|k|^3)$ growth. They deduced this by evaluating the interaction coefficient $T_{k,k_1,k_2,k_3}$ on the restricted configuration $k_3=k$ and $|k_1|=|k_2|$. In contrast, through formal asymptotic expansions within the framework of weak turbulence theory, Zakharov and Geogjaev \cite{zakharov2010energy, geogjaev2017numerical, geogjaev2025properties} revealed that the interaction coefficient actually exhibits a sharp leading-order behavior of $\mathcal{O}(|k|^2)$ for highly non-local wave interactions (see, for example, \cite[Eq.~(4.12)]{geogjaev2025properties}). To contextualize the severity of this singularity, we note that recent results by Germain, La, and Zhang \cite[Theorem 1.4 and Eq.~(5.5)]{GLZ2025} on the kinetic Majda--McLaughlin--Tabak (MMT) model establish local well-posedness in Banach spaces for the one-dimensional ($d=1$) case only when the collision kernel exhibits strictly sub-linear growth. Both the $\mathcal{O}(|k|^3)$ and $\mathcal{O}(|k|^2)$ growth rates of the gravity water wave kernel vastly exceed this analytical threshold. Consequently, such staggering singular behavior lies far beyond the reach of existing analytical frameworks, rendering the construction of strong solutions seemingly intractable.

In this paper, we overcome these formidable obstacles through two major breakthroughs: a structural discovery regarding an algebraic cancellation within the collision kernel, and a novel functional decomposition technique for the collision operator.

First, through a delicate re-examination of the water wave collision kernel, we discover a hidden algebraic cancellation. We rigorously demonstrate that a specific component of the collision kernel containing the factor $\omega_1-\omega_2$---which was unaccounted for in~\cite{waterkernel2024}---exactly cancels out the severe $\mathcal{O}(|k|^3)$ leading-order term. Exploiting this exact cancellation, we establish an improved upper bound, proving that the kernel grows at most quadratically (i.e., bounded by $\mathcal{O}(|k|^2)$ or $\mathcal{O}(|k||k_3|)$) in the highly non-local regime.\footnote{We note that the estimate derived in \cite{waterkernel2024} is in fact sharp under the strict kinematic constraint $k_3=k$ and $|k_1|=|k_2|$. The additional cancellation discovered in the present work arises from carefully tracking the small but critical difference between $|k_1|$ and $|k_2|$ on the resonant manifold.} This rigorously validates the asymptotic smallness of the interaction coefficient anticipated in the physics literature \cite{zakharov2010energy, geogjaev2017numerical, geogjaev2025properties}.

Second, recognizing that quadratic growth still precludes the use of standard contraction mapping arguments, we recalibrate our analytical objective. Instead of pursuing classical local well-posedness, we establish the local-in-time \emph{existence} of $L^1$ strong solutions for any initial data belonging to a suitably weighted $L^2 \cap L^\infty$ space, and show that the solution propagates within this space locally in time. Crucially, as the model under our consideration is precisely the celebrated Hasselmann kinetic equation \cite{HM1962,H1963}, our result rigorously establishes the existence of solutions to the foundational water wave model originally formulated by Hasselmann.

To achieve this, we introduce a new methodology specifically tailored to accommodate both the $\omega(k) = \sqrt{|k|}$ dispersion relation and the at most $\mathcal{O}(|k|^2)$ kernel growth. Our core observation is that, under a specific linearization, the collision operator can be structurally decomposed into the sum of a dissipative operator and a bounded operator on $L^p$ ($1 \leq p \leq \infty$). Crucially, we prove that the commutator between this linearized collision operator and the weight $\langle k \rangle^a$ inherits a similarly favorable structure: it can be bounded by a dissipative operator plus a bounded operator on $L^2$, composed with the multiplication operator $\langle k \rangle^a$. This delicate structural decomposition enables us to uniformly control the weighted $L^2 \cap L^\infty$ norm of the solution at each iteration step, ultimately allowing for the rigorous construction of the $L^1$ strong solution. 

It is worth noting that our technique faces a rigid analytical barrier at the quadratic growth threshold for the gravity wave dispersion relation $\omega(k) = \sqrt{|k|}$. Specifically, if the kernel grows faster than $\mathcal{O}(|k|^2)$, the weighted linearized collision operator $\langle k \rangle^a \mathcal{Q}_g(t)$ can no longer be structurally decomposed into a dissipative component and a bounded operator composed with $\langle k \rangle^a$ in any $L^p$ space, rendering the uniform weighted estimates utterly unattainable. In light of this strict analytical barrier, it is a remarkable structural synergy that our refined collision kernel---reduced via the exact cancellation---satisfies this exact $\mathcal{O}(|k|^2)$ threshold, making the implementation of our functional framework rigorously viable.

\subsection{Outline of the proof. } 
Throughout the outline of this proof, let $T > 0$ be a fixed positive constant representing the lifespan of the local solution, whose precise value will be determined later.
\begin{df}\cite[Page 82]{KR} A linear operator $(A, D(A))$ on a Banach space $X$ is called \textit{dissipative} if
\begin{equation}
    \|(\lambda - A)x\| \geq \lambda \|x\| 
\end{equation}
for all $\lambda > 0$ and $x \in D(A)$.
\end{df}

The core of our proof lies in the observation that, for any weight exponent $a>0$, the weighted collision term $\langle k\rangle^a \mathcal{Q}[f]$ can be rewritten as the action of a linear operator on the weighted function $\langle k\rangle^a f$. Specifically, it admits a natural decomposition in $L^2$ into a dissipative part and a bounded part:
\begin{equation}
    \langle k\rangle^a \mathcal{Q}[f] = \big( \mathcal{Q}_{f,D} + \mathcal{Q}_{f,b} \big) (\langle k\rangle^a f), \qquad \text{ in } L^2,
\end{equation}
where $\mathcal{Q}_{f,D}$ is a dissipative operator and $\mathcal{Q}_{f,b}$ is a bounded operator.

To make this operator splitting precise, for a given background flow $g \in L^\infty([0,T]; \Bc)$ and any $p \in [1,\infty)$, we define the linearized operator $\mathcal{Q}_g(t): L^p_2(\mathbb{R}^d) \to L^p(\mathbb{R}^d)$ acting on $h$ as:
\begin{equation}\label{Qfn}
\begin{aligned}
    \mathcal{Q}_g(t) h(k) =  & \iiint_{\mathbb{R}^{3d}} T_{k,k_1,k_2,k_3} \delta(\Sigma)\delta(\Omega) \Big[ 2\chi_{2<3} g_1(t) g_2(t) (h_3-h) \\
    & \qquad\qquad\qquad\qquad + g_2(t) g_3(t)h - 2\chi_{2\geq 3} g_1(t) g_2(t)h \Big]\,dk_1\,dk_2\, dk_3,
\end{aligned}
\end{equation}
where we use the shorthand characteristic functions $\chi_{j<l}:=\chi(|k_j|<|k_l|)$ and $\chi_{j\geq l}:=1-\chi_{j<l}$. Notice that by Lemma~\ref{lem: Qf and Qg}, this linearization satisfies the consistency relation $\Qc_f(t)f(t)=\Qc[f(t)]$. In light of Lemmas~\ref{lem: QgD diss} and~\ref{lem: QgB bd} and Eq.~\eqref{Qfn}, this operator naturally decomposes into a dissipative and a bounded component:
\begin{equation}
    \Qc_{g}(t)=\Qc_{g,D}(t)+\Qc_{g,b}(t).
\end{equation}
Here, $\mathcal{Q}_{g,b}(t)$ (given in Eq.~\eqref{eq: Qbg}) is bounded on $L^p$ for all $1 \leq p \leq \infty$. Meanwhile, by virtue of~\eqref{eq: est kernel} and Lemmas~\ref{lem: bd QgD} and~\ref{lem: Rep1}, the family of dissipative operators $\{\mathcal{Q}_{g,D}(t)\}_{t \in [0,T]}$ is defined on a common, time-independent dense domain $L^p_2(\mathbb{R}^d)$ for $1 \leq p < \infty$. By definition, the limiting operator $\Qc_{g,D}(t)\equiv \Qc_{g,D,0}(t)$ is understood in the sense of the strong limit from $L^2_2$ to $L^2$ of the regularized operators $\Qc_{g,D,\epsilon}(t)$ defined in~\eqref{eq: QDg}. 

To rigorously construct the evolution generated by $\Qc_{g,D}(t)$, we note that since each regularized operator $\Qc_{g,D,\epsilon}(t)$ is bounded on $L^2$ for $\epsilon\in (0,1]$, it generates a unique two-parameter propagator $U_{g,D,\epsilon}(t,s)$ for $0\leq s\leq t\leq T$ satisfying:
\begin{equation}
    \partial_t U_{g,D,\epsilon}(t,s) =\Qc_{g,D,\epsilon}(t)U_{g,D,\epsilon}(t,s),\qquad U_{g,D,\epsilon}(s,s)=I.
\end{equation}
As we will establish in Lemma~\ref{lem: Cauchy}, the family of regularized propagators satisfies the Cauchy criterion in the strong operator topology on $L^2(\mathbb{R}^d)$ as $\epsilon \downarrow 0$. We defer the rigorous proof of this Cauchy property to Appendix~\ref{app: Cauchy}, which crucially relies on the uniform dissipativity of the regularized operators combined with fine commutator estimates involving the weight $\langle k \rangle^a$. Consequently, the strong limit
\begin{equation}
    U_{g,D}(t,s) := s\text{-}\lim_{\epsilon \downarrow 0} U_{g,D,\epsilon}(t,s)
\end{equation}
is well-defined and rigorously serves as the exact propagator generated by the limiting dissipative family $\{\Qc_{g,D}(t)\}$.

\medskip
\noindent\textbf{The iteration scheme.} To construct the solution to the non-linear system \eqref{4wave}, we employ a standard iteration scheme. We initialize the sequence by setting the first iterate to the initial data, $f_1(t,k) = f_0(k)$. For each integer $n \geq 1$, we then define the subsequent iterate $f_{n+1}(t,k)$ as the unique solution to the following linearized initial value problem:
\begin{equation}\label{Qn}
    \begin{cases}
        \partial_t f_{n+1}(t) = \mathcal{Q}_n(t) f_{n+1}(t), \\
        f_{n+1}(0) = f_0 \in \Bc,
    \end{cases} \qquad \text{on } [0,T] \times \mathbb{R}^d,
\end{equation}
where we use the shorthand notation $\mathcal{Q}_n(t) \equiv \mathcal{Q}_{f_n}(t)$ for the linearized collision operator evaluated at the previous iterate. Similar to standard wave kinetic models, the intrinsic structure of the linearized operator ensures that this iterative scheme is non-negativity preserving. Consequently, provided the initial data is non-negative ($f_0 \geq 0$), the iterates strictly satisfy $f_{n+1}(t,k) \geq 0$ for all $n \in \mathbb{N}^+$ and $t \in [0,T]$.

\medskip

\noindent\textbf{The dissipation of the weighted solutions.} Crucially, for any weight exponent $a>0$, the weighted iterate $\langle k\rangle^a f_{n+1}(t)$ obeys the evolution equation:
\begin{equation}\label{eq: DE afn+1}
    \partial_t \big( \langle k\rangle^a f_{n+1}(t) \big) = \big( \mathcal{Q}_{n}(t) + \tilde{\mathcal{Q}}_n(t) + \mathcal{R}_n(t) \big) \big( \langle k\rangle^a f_{n+1}(t) \big),
\end{equation}
where $\tilde{\mathcal{Q}}_n(t): D(\tilde{\mathcal{Q}}_n(t)) \to L^2(\mathbb{R}^d)$ is a supplementary dissipative operator, and $\mathcal{R}_n(t)$ is a bounded operator on $L^p(\mathbb{R}^d)$ for all $1 \leq p \leq \infty$. Writing $g(t) = f_n(t)$ for brevity, the actions of these operators on a test function $h$ are defined as:
\begin{equation}\label{tQfn}
\begin{aligned}
    \tilde{\mathcal{Q}}_n(t) h(k) = 2 \iiint_{\mathbb{R}^{3d}} & T_{k,k_1,k_2,k_3} \frac{\langle k\rangle^a - \langle k_3\rangle^a}{\langle k\rangle^{a/2}\langle k_3\rangle^{a/2}} \delta(\Sigma)\delta(\Omega) \\
    & \times \chi_{2<3}\chi_{1<0} g_1(t) g_2(t) h_3 \,dk_1\,dk_2\, dk_3
\end{aligned}
\end{equation}
and
\begin{equation}\label{eq: def Rnt}
\begin{aligned}
    \mathcal{R}_n(t) h(k) = 2 \iiint_{\mathbb{R}^{3d}} & T_{k,k_1,k_2,k_3} \frac{\langle k\rangle^a - \langle k_3\rangle^a}{\langle k_3\rangle^a} \delta(\Sigma)\delta(\Omega) \\
    & \times \chi_{2<3}\chi_{1\geq 0} g_1(t) g_2(t) h_3 \,dk_1\,dk_2\, dk_3 \\
    + 2 \iiint_{\mathbb{R}^{3d}} & T_{k,k_1,k_2,k_3} \frac{(\langle k\rangle^{a/2} - \langle k_3\rangle^{a/2})^2 (\langle k\rangle^{a/2} + \langle k_3\rangle^{a/2})}{\langle k\rangle^{a/2}\langle k_3\rangle^a} \delta(\Sigma)\delta(\Omega) \\
    & \times \chi_{2<3}\chi_{1<0} g_1(t) g_2(t) h_3 \,dk_1\,dk_2\, dk_3.
\end{aligned}
\end{equation}
The rigorous justification of these operator properties is deferred to Section~\ref{sec: A-Est}: we establish the validity of the weak formulation of Eq.~\eqref{eq: DE afn+1} in Lemma~\ref{lem: we Dyn}, the dissipativity of $\tilde{\mathcal{Q}}_n(t)$ in Lemma~\ref{lem: tQn diss}, and the uniform boundedness of $\mathcal{R}_n(t)$ in Lemma~\ref{lem: eq: Rn bd}. Here, $a > 0$ acts as a flexible weight parameter whose specific value depends on whether we are estimating the weighted $L^2$ or weighted $L^\infty$ norm, as well as whether we are establishing the convergence of the iterative sequence $\{f_n(t)\}$.

It is important to emphasize that the uniqueness of the propagator generated by the combined operator $\mathcal{Q}_{n}(t) + \tilde{\mathcal{Q}}_n(t) + \mathcal{R}_n(t)$ does not follow trivially from abstract bounded perturbation theory. Rather, it is rigorously constructed similarly to the purely dissipative case: by regularizing the dissipative component and taking the strong limit as $\epsilon \downarrow 0$. For this resulting exact propagator, the notions of strong and weak solutions rigorously coincide. In fact, we adopt this as a global convention: throughout the remainder of this paper, whenever we refer to a propagator generated by the sum of a dissipative and a bounded operator, it is implicitly understood to be strictly constructed through this exact $\epsilon$-regularization limit. By leveraging this fact alongside the dissipative nature of the underlying operators and applying Lemmas~\ref{lem: QgB bd} and~\ref{lem: eq: Rn bd} to Eq.~\eqref{eq: DE afn+1}, we can dominate the energy evolution entirely by the bounded components. 

Specifically, setting the weight exponent to $a=22+5d$, we obtain the following differential inequality for the weighted $L^2$ norm:
\begin{equation}
    \frac{d}{dt} \| f_{n+1}(t) \|_{2,22+5d}^2 \leq 2 \left( (C_b + \tilde{C}_a) \sup_{\tau \in [0,T]} \| f_n(\tau) \|_{\infty,10+2d}^2 \right) \| f_{n+1}(t) \|^2_{2,22+5d}, \qquad \forall\, t \in (0,T].
\end{equation}
Applying Gr\"onwall's inequality and taking the square root (which cleanly cancels the factor of 2), this yields the uniform weighted $L^2$ bound:
\begin{equation}\label{eq: La2}
    \| f_{n+1}(t) \|_{2,22+5d} \leq \exp\left( t(C_b + \tilde{C}_a) \sup_{\tau \in [0,T]} \| f_n(\tau) \|_{\infty,10+2d}^2 \right) \|f_0\|_{2,22+5d}, \qquad \forall\, t \in (0,T].
\end{equation}

Turning to the pointwise estimates, let $\tilde{U}_n(t,s)$ denote the evolution family (propagator) generated by the operator $\mathcal{Q}_n(t) + \mathcal{R}_n(t)$. Fixing the weight exponent to $a=12+4d$ specifically for the weighted $L^\infty$ estimate, Duhamel's principle allows us to represent the weighted iterate as:
\begin{equation}
    \langle k\rangle^{12+4d} f_{n+1}(t) = \tilde{U}_{n}(t,0) \big( \langle k\rangle^{12+4d}f_0 \big) + \int_0^t \tilde{U}_{n}(t,s) \big( \tilde{\mathcal{Q}}_{n}(s)\langle k\rangle^{12+4d}f_{n+1}(s) \big) \, ds.
\end{equation}
Since the dissipativity of $\mathcal{Q}_{f_n,D}(t)$ guarantees the exponential bounds for the propagator $\tilde{U}_n(t,s)$ established in Lemma~\ref{lem: tUn h}, taking the $L^\infty$ norm of the Duhamel formulation yields:
\begin{equation}
\begin{aligned}
    \|f_{n+1}(t)&\|_{\infty,12+4d}  \leq \exp\left( t(C_b+\tilde{C}_a) \sup_{\tau \in [0,T]} \| f_n(\tau) \|_{\infty,10+2d}^2 \right) \|f_0\|_{\infty,12+4d} \\
    &+ \int_0^t \exp\left( (t-s)(C_b+\tilde{C}_a) \sup_{\tau \in [0,T]} \| f_n(\tau) \|_{\infty,10+2d}^2 \right) \|\tilde{\mathcal{Q}}_{n}(s)\langle k\rangle^{12+4d}f_{n+1}(s)\|_{\infty} \, ds.
\end{aligned}
\end{equation}
Finally, bounding the integrand by applying Lemma~\ref{lem: infty bd}, we arrive at the following integral inequality for all $t \in (0, T]$:
\begin{equation}\label{eq: Linfty}
\begin{aligned}
    \|f_{n+1}(t)\|_{\infty,12+4d} & \leq \exp\left( t(C_b+\tilde{C}_a) \sup_{\tau \in [0,T]} \| f_n(\tau) \|_{\infty,10+2d}^2 \right) \|f_0\|_{\infty,12+4d} \\
    & \quad + \int_0^t \exp\left( (t-s)(C_b+\tilde{C}_a) \sup_{\tau \in [0,T]} \| f_n(\tau) \|_{\infty,10+2d}^2 \right) \\
    & \qquad \times \tilde{C}_b \left( \sup_{\tau \in [0,T]} \|f_n(\tau)\|_{\infty,10+2d}^2 \right)  \|f_{n+1}(s)\|_{2,22+5d}\, ds.
\end{aligned}
\end{equation}
Combining estimates~\eqref{eq: La2} and~\eqref{eq: Linfty} and taking the supremum over time yields the combined bound for the intersection norm:
\begin{equation}\label{eq: L2infty fbd}
\begin{aligned}
   & \sup_{t\in (0,T]} \|f_{n+1}(t)\|_{L^2_{22+5d}\cap L^\infty_{12+4d}} \\
   & \leq \exp\left( T(C_b+\tilde{C}_a) \sup_{t\in (0,T]} \|f_n(t)\|_{\infty,10+2d}^2 \right) \\
   & \quad \times \left( 1 + T\tilde{C}_b \sup_{t\in (0,T]} \|f_n(t)\|_{\infty,10+2d}^2 \right) \|f_0\|_{L^2_{22+5d}\cap L^\infty_{12+4d}}.
\end{aligned}
\end{equation}
To close the uniform estimates, we impose a smallness condition on the existence time. Specifically, taking $T > 0$ such that 
\begin{equation}\label{eq: def T1}
    T \leq T_1 := \frac{1}{10(C_b+\tilde{C}_a+\tilde{C}_b) \|f_0\|_{L^2_{22+5d}\cap L^\infty_{12+4d}}^2},
\end{equation}
allows us to establish a uniform-in-$n$ bound via mathematical induction. Indeed, assuming the induction hypothesis 
\begin{equation}
    \sup_{t\in (0,T]} \|f_{n}(t)\|_{L^2_{22+5d}\cap L^\infty_{12+4d}} \leq 2\|f_0\|_{L^2_{22+5d}\cap L^\infty_{12+4d}}
\end{equation}
holds for the $n$-th iterate, substituting this into the previous estimate guarantees that 
\begin{equation}
    \sup_{t\in (0,T]} \|f_{n+1}(t)\|_{L^2_{22+5d}\cap L^\infty_{12+4d}} \leq 2\|f_0\|_{L^2_{22+5d}\cap L^\infty_{12+4d}}.
\end{equation}
Since the base case $f_1(t,k) = f_0(k)$ trivially satisfies this bound, we conclude that the entire iterative sequence is uniformly bounded:
\begin{equation}\label{eq: uniform bound fn}
    \sup_{t\in (0,T]} \|f_{n}(t)\|_{L^2_{22+5d}\cap L^\infty_{12+4d}} \leq 2\|f_0\|_{L^2_{22+5d}\cap L^\infty_{12+4d}} \qquad \forall \, n \in \mathbb{N}.
\end{equation}
Combining this uniform bound with an application of~\eqref{eq: DE afn+1} for $a = 2d+2$, together with a standard contraction argument, we establish the existence of a finite lifespan
\[
T = T\bigl(\|f_0\|_{2,22+5d}, \|f_0\|_{\infty,12+4d}\bigr) > 0,
\]
such that the gravity water-wave kinetic equation~\eqref{4wave} admits a local-in-time strong $L^1$ solution $f(t,k)$ on the interval $[0,T]$.

\medskip

  \section{Structure of the Collision Kernel for Gravity Water Waves}

In the context of the kinetic equation for gravity water waves, the four-wave collision kernel takes the form $|T_{k_1 k}^{k_2 k_3}|^2$, where the symmetrized amplitude $T_{k_1 k}^{k_2 k_3}$ is defined as
\begin{equation}\label{eq: def T}
    T_{k_1 k}^{k_2k_3} = \frac{1}{4}\Big( \tilde{T}_{k_1 k}^{k_2 k_3} + \tilde{T}_{k k_1}^{k_2 k_3} + \tilde{T}_{k_2 k_3}^{k_1 k} + \tilde{T}_{k_3 k_2}^{k_1 k} \Big).
\end{equation}
Here, $\tilde{T}_{k_1 k}^{k_2 k_3}$ represents the pre-symmetrized interaction coefficient; its explicit algebraic expression is deferred to~\eqref{eq: tTk1k2k3k4} below. Detailed derivations and representations of this coefficient can be found, for instance, in~\cite[Appendix B]{waterkernel2024} and~\cite[Appendix B]{PRZ2003}. To rigorously construct local-in-time strong solutions, it is crucial to analyze the kernel's singular behavior under extreme scale separations. Therefore, in this section, we isolate and evaluate the kernel on the precise support defined by the resonant manifold and specific frequency filters:
\begin{equation}
    \text{supp} \left\{ \delta(\Sigma) \delta(\Omega) \varphi_{1<3} \varphi_{2<3} \right\}.
\end{equation}
For clarity of exposition, we fix the gravitational constant $g=1$ throughout our derivation. We expect that this analytical framework and its corresponding estimates apply seamlessly to the general case for any arbitrary constant $g > 0$.

\begin{prop}\label{prop: kernel}
On the resonant manifold defined by the support of $\delta(\Sigma) \delta(\Omega)$, the collision kernel $T_{k,k_1,k_2,k_3} \coloneqq |T_{k_1k}^{k_2k_3}|^2$ satisfies the bound, with $k=k_0$,
\begin{equation}\label{eq: est kernel}
    T_{k,k_1,k_2,k_3} \leq C_Q \left( \min_{0\leq j<l\leq 3} \big( |k_j|^2+|k_l|^2 \big) \right)^2 \left( \sum_{j=0}^3|k_j|^2 \right)
\end{equation}
for some constant $C_Q>0$.
\end{prop}

Since each term in the pre-symmetrized amplitude \eqref{eq: tTk1k2k3k4} is homogeneous of degree 3, the squared kernel $T_{k,k_1,k_2,k_3}$ possesses a total homogeneity of degree 6. Exploiting the full symmetry of the kernel, it suffices to establish the bound in the extreme scale-separation regime where $|k_3| \sim |k| \gg |k_1|, |k_2|$. Specifically, the proof reduces to showing that
\begin{equation}\label{eq: est kernel'}
    T_{k,k_1,k_2,k_3} \leq \tilde{C}_Q (|k_1|^2+|k_2|^2)^2|k|^2
\end{equation}
holds for some constant $\tilde{C}_Q > 0$. To systematically achieve this bound, the kernel is decomposed into symmetric constituents; we refer to Figure~\ref{fig:kernel_estimates} for a schematic overview of this decomposition strategy and the corresponding roadmap of the \emph{a priori} estimates.

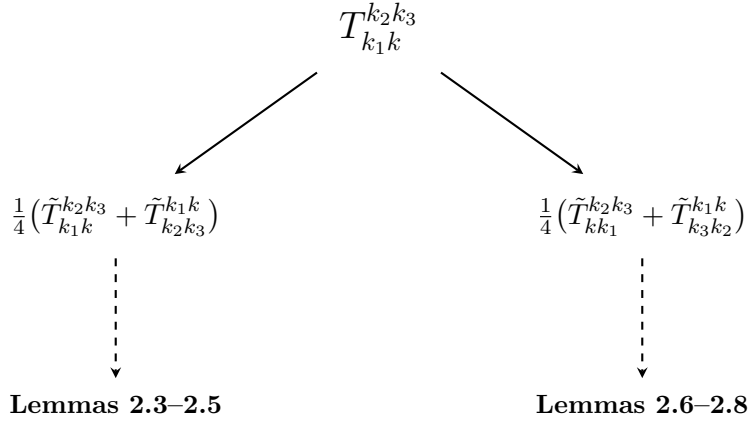
\begin{figure}[htbp]
    \centering
    \begin{tikzpicture}[
        node distance=2.5cm and 4cm,
        >=stealth, 
        thick,
        term_box/.style={align=center, inner sep=8pt},
        lemma_box/.style={align=center, font=\small\bfseries, inner sep=6pt}
    ]
        \node[term_box] (root) at (0,0) {\Large $T_{k_1 k}^{k_2 k_3}$};
        \node[term_box] (left_term) at (-3.5, -2.5) {$\frac{1}{4}\big(\tilde{T}_{k_1 k}^{k_2 k_3} + \tilde{T}_{k_2 k_3}^{k_1 k}\big)$};
        \node[term_box] (right_term) at (3.5, -2.5) {$\frac{1}{4}\big(\tilde{T}_{k k_1}^{k_2 k_3} + \tilde{T}_{k_3 k_2}^{k_1 k}\big)$};
        \node[lemma_box] (left_lemmas) at (-3.5, -5) {Lemmas 2.3--2.5};
        \node[lemma_box] (right_lemmas) at (3.5, -5) {Lemmas 2.6--2.8};
        
        \draw[->] (root) -- (left_term);
        \draw[->] (root) -- (right_term);
        \draw[->, dashed] (left_term) -- (left_lemmas); 
        \draw[->, dashed] (right_term) -- (right_lemmas);
        
    \end{tikzpicture}
    \caption{Decomposition scheme and \emph{a priori} estimates for the collision kernel.}
    \label{fig:kernel_estimates}
\end{figure}

To rigorously localize this asymptotic regime, let $\varphi: [0,\infty) \to [0,1]$ be a smooth cut-off function such that
\begin{equation}\label{eq: def varphi}
    \varphi(\lambda) = \begin{cases}
        1 & \text{if } \lambda \leq \frac{1}{40}, \\
        0 & \text{if } \lambda \geq \frac{1}{20}.
    \end{cases}
\end{equation}
In what follows, we adopt the shorthand notation
\begin{equation}\label{eq: varphi j<l}
    \varphi_{j<l} \coloneqq \varphi\!\left(|k_j|/|k_l|\right), \qquad \bar{\varphi}_{j<l} \coloneqq 1 - \varphi_{j<l}
\end{equation}
for all $j,l \in \{0,1,2,3\}$. On the restricted support of the product $\delta(\Sigma)\delta(\Omega)\varphi_{1<3}\varphi_{2<3}$, the cut-off functions strictly enforce the frequency scale separations
\begin{equation}\label{eq: k12 and k3}
   20 |k_1| \leq |k_3| \quad \text{and} \quad 20|k_2| \leq |k_3|.
\end{equation}
Consequently, the momentum conservation law $k_1 + k = k_2 + k_3$ (enforced by $\delta(\Sigma)$) dictates via the triangle inequality that $|k|$ must lie within the range
\begin{equation}\label{eq: range k}
    |k| \in \big[|k_3|-|k_1|-|k_2|, \, |k_3|+|k_1|+|k_2|\big] \subseteq \left[\frac{9}{10}|k_3|, \, \frac{11}{10}|k_3|\right].
\end{equation}

Let us introduce the auxiliary functions $f_\pm(x,y) \coloneqq x \cdot y \pm |x||y|$ for all $x,y \in \mathbb{R}^d$. The proof of Proposition~\ref{prop: kernel} crucially relies on the preliminary estimates established in Lemma~\ref{lem: basic kernel}. To preserve the continuity of our main exposition, the detailed proof of Lemma~\ref{lem: basic kernel} is deferred to the Appendix.

\begin{lem}\label{lem: basic kernel}
    On the restricted support defined by $\delta(\Sigma)\delta(\Omega)\varphi_{1<3}\varphi_{2<3}$, the interacting wave frequencies and wavenumbers satisfy the following estimates:
    \begin{equation}\label{eq: est omega1-2}
        |\omega_1-\omega_2| = |\omega-\omega_3| \leq \frac{3|k_2-k_1|}{5\omega_3},
    \end{equation}
    \begin{equation}\label{eq: k-k3}
        \big| |k|-|k_3| \big| \leq |k_2-k_1|,
    \end{equation}
    \begin{equation}\label{eq: kk3 1/4}
        \left| \frac{1}{|k|^{1/4}} - \frac{1}{|k_3|^{1/4}} \right| \leq \frac{2|k_2-k_1|}{5|k_3|^{5/4}},
    \end{equation}
    \begin{equation}\label{eq: 1-3-1-0 omega}
        \big| (\omega_1-\omega_3)^2 - (\omega_1-\omega)^2 \big| \leq |k_2-k_1|,
    \end{equation}
    and
    \begin{equation}\label{eq: omega dyz}
        \omega_{y-z}^2 - (\omega_y-\omega_z)^2 \geq 2\min\{\omega_y,\omega_z\} |\omega_y-\omega_z|
    \end{equation}
    for any $y, z \in \mathbb{R}^d$.
\end{lem}

Consider the explicit algebraic expression for the pre-symmetrized interaction coefficient $\tilde T_{k_1k_2}^{k_3k_4}$, which is given by:
\begin{equation}\label{eq: tTk1k2k3k4}
\begin{aligned}
\tilde{T}_{k_1 k_2}^{k_3 k_4} = &-\frac{1}{16\pi^2 (|k_1| |k_2| |k_3| |k_4|)^{1/4}} \Bigg\{ -12 |k_1| |k_2| |k_3| |k_4| \\
& - 2(\omega_1 + \omega_2)^2 \Big[ \omega_3 \omega_4 f_-(k_1,k_2) + \omega_1 \omega_2 f_-(k_3,k_4) \Big] \\
& - 2(\omega_1 - \omega_3)^2 \Big[ \omega_2 \omega_4 f_+(k_1,k_3) + \omega_1 \omega_3 f_+(k_2,k_4) \Big] \\
& - 2(\omega_1 - \omega_4)^2 \Big[ \omega_2 \omega_3 f_+(k_1,k_4) + \omega_1 \omega_4 f_+(k_2,k_3)\Big] \\
& + f_+(k_1,k_2)f_+(k_3,k_4)+f_-(k_1,k_3)f_-(k_2,k_4)+f_-(k_1,k_4)f_-(k_2,k_3) \\
& + \frac{4(\omega_1 + \omega_2)^2 f_-(k_1,k_2)f_-(k_3,k_4)}{\omega_{k_1+k_2}^2 - (\omega_1 + \omega_2)^2} + \frac{4(\omega_1 - \omega_3)^2 f_+(k_1,k_3)f_+(k_2,k_4)}{\omega_{k_1-k_3}^2 - (\omega_1 - \omega_3)^2} \\
& + \frac{4(\omega_1 - \omega_4)^2f_+(k_1,k_4)f_+(k_2,k_3)}{\omega_{k_1-k_4}^2 - (\omega_1 - \omega_4)^2} \Bigg\}.
\end{aligned}
\end{equation}
Furthermore, to streamline the upcoming singularity analysis, we introduce the angular variables:
\begin{equation}\label{eq: ab}
\alpha:=\hat k_1\cdot \hat k\qquad \text{ and }\qquad \beta:=\hat k_2\cdot \hat k.
\end{equation}

\subsection{The \texorpdfstring{$\tilde T_{k_1 k}^{k_2 k_3}+\tilde T_{k_2 k_3}^{k_1 k}$}{T\_{k1,k}\^{k2,k3} + T\_{k2,k3}\^{k1,k}} contribution}
Based on Eq.~\eqref{eq: tTk1k2k3k4}, the expression for $\tilde T_{k_1 k}^{k_2 k_3}$ reads
\begin{equation}\label{eq: tT kk1k2k3}
\begin{aligned}
    \tilde{T}_{k_1 k}^{k_2 k_3} &= -\frac{1}{16\pi^2 (|k_1| |k| |k_2| |k_3|)^{1/4}} \Bigg\{ -12 |k_1| |k| |k_2| |k_3| \\
    &\quad - 2(\omega_1 + \omega)^2 \Big[ \omega_2 \omega_3 f_-(k_1,k) + \omega_1 \omega f_-(k_2,k_3) \Big] \\
    &\quad - 2(\omega_1 - \omega_2)^2 \Big[ \omega \omega_3 f_+(k_1,k_2) + \omega_1 \omega_2 f_+(k,k_3) \Big] \\
    &\quad - 2(\omega_1 - \omega_3)^2 \Big[ \omega \omega_2 f_+(k_1,k_3) + \omega_1 \omega_3 f_+(k,k_2)\Big] \\
    &\quad + f_+(k_1,k)f_+(k_2,k_3)+f_-(k_1,k_2)f_-(k,k_3)+f_-(k_1,k_3)f_-(k,k_2) \\
    &\quad + \frac{4(\omega_1 + \omega)^2 f_-(k_1,k)f_-(k_2,k_3)}{\omega_{k_1+k}^2 - (\omega_1 + \omega)^2} + \frac{4(\omega_1 - \omega_2)^2 f_+(k_1,k_2)f_+(k,k_3)}{\omega_{k_1-k_2}^2 - (\omega_1 - \omega_2)^2} \\
    &\quad + \frac{4(\omega_1 - \omega_3)^2f_+(k_1,k_3)f_+(k,k_2)}{\omega_{k_1-k_3}^2 - (\omega_1 - \omega_3)^2} \Bigg\}.
\end{aligned}
\end{equation}
On the restricted support defined by $\delta(\Sigma)\delta(\Omega)\varphi_{1<3}\varphi_{2<3}$, we decompose $\tilde{T}_{k_1 k}^{k_2 k_3}$ into three distinct scale components:
\begin{equation}
    \tilde{T}_{k_1 k}^{k_2 k_3} = \tilde{T}_{k_1 k,1}^{k_2 k_3} + \tilde{T}_{k_1 k,2}^{k_2 k_3}+\tilde{T}_{k_1 k,3}^{k_2 k_3}.
\end{equation}
Here, the first piece $\tilde{T}_{k_1 k,1}^{k_2 k_3}$ captures the leading-order terms of $O(|k_3|^{2})$, extracted from the second, fourth, eighth, and tenth terms within the braces of Eq.~\eqref{eq: tT kk1k2k3}:

\begin{equation}\label{eqL tT 1}
\begin{aligned}
    \tilde{T}_{k_1 k,1}^{k_2 k_3} \coloneqq &-\frac{1}{16\pi^2 (|k_1| |k| |k_2| |k_3|)^{1/4}} \Bigg\{ - 2\omega^2 \Big[ \omega_2 \omega_3 f_-(k_1,k) + \omega_1 \omega f_-(k_2,k_3) \Big] \\
    &\quad - 2\omega_3^2 \Big[ \omega \omega_2 f_+(k_1,k_3) + \omega_1 \omega_3 f_+(k,k_2)\Big] \\
    &\quad + \frac{4 \omega^2 f_-(k_1,k)f_-(k_2,k_3)}{-2\omega_1\omega} + \frac{4 \omega_3^2f_+(k_1,k_3)f_+(k,k_2)}{2\omega_1\omega_3} \Bigg\}.
\end{aligned}
\end{equation}
The second piece $\tilde T_{k_1k,2}^{k_2k_3}$ comprises the sub-leading terms of $O(|k_3|^{3/2})$ from Eq.~\eqref{eq: tT kk1k2k3}:
\begin{equation}\label{eq: tT2 kk1k2k3}
\begin{aligned}
    \tilde{T}_{k_1 k,2}^{k_2 k_3} \coloneqq &-\frac{1}{16\pi^2 (|k_1| |k| |k_2| |k_3|)^{1/4}} \Bigg\{ -12 |k_1| |k| |k_2| |k_3| \\
    &\quad - 4\omega_1\omega \Big[ \omega_2 \omega_3 f_-(k_1,k) + \omega_1 \omega f_-(k_2,k_3) \Big] + 4\omega_1\omega_3 \Big[ \omega \omega_2 f_+(k_1,k_3) + \omega_1 \omega_3 f_+(k,k_2)\Big] \\
    &\quad + f_+(k_1,k)f_+(k_2,k_3) + f_-(k_1,k_3)f_-(k,k_2) \\
    &\quad + \frac{8\omega_1\omega f_-(k_1,k)f_-(k_2,k_3)}{-2\omega_1\omega} + \frac{4\omega^2 f_-(k_1,k)f_-(k_2,k_3)(|k_1|-k_1\cdot \hat k)}{(-2\omega_1\omega)^2} \\
    &\quad +\frac{-8\omega_1\omega_3 f_+(k_1,k_3)f_+(k,k_2)}{2\omega_1\omega_3} + \frac{4\omega^2_3 f_+(k_1,k_3)f_+(k,k_2)(|k_1|+k_1\cdot \hat k_3)}{(2\omega_1\omega_3)^2} \Bigg\}.
\end{aligned}
\end{equation}
Finally, the third piece $\tilde T_{k_1 k,3}^{k_2 k_3}$ constitutes the remainder term:
\begin{equation}\label{eq: tT3 kk1k2k3}
\begin{aligned}
    \tilde{T}_{k_1 k,3}^{k_2 k_3} \coloneqq &-\frac{1}{16\pi^2 (|k_1| |k| |k_2| |k_3|)^{1/4}} \Bigg\{ - 2\omega_1^2 \Big[ \omega_2 \omega_3 f_-(k_1,k) + \omega_1 \omega f_-(k_2,k_3) \Big] \\
    &\quad - 2(\omega_1 - \omega_2)^2 \Big[ \omega \omega_3 f_+(k_1,k_2) + \omega_1 \omega_2 f_+(k,k_3) \Big] \\
    &\quad - 2\omega_1^2 \Big[ \omega \omega_2 f_+(k_1,k_3) + \omega_1 \omega_3 f_+(k,k_2)\Big] + f_-(k_1,k_2)f_-(k,k_3) \\
    &\quad + f_-(k_1,k)f_-(k_2,k_3) \bigg[ \frac{4(\omega_1 + \omega)^2}{\omega_{k_1+k}^2 - (\omega_1 + \omega)^2} - \frac{4\omega^2}{-2\omega_1\omega} - \frac{8\omega_1\omega}{-2\omega_1\omega} - \frac{4\omega^2(|k_1|-k_1\cdot\hat k)}{(-2\omega_1\omega)^2} \bigg] \\
    &\quad + \frac{4(\omega_1 - \omega_2)^2 f_+(k_1,k_2)f_+(k,k_3)}{\omega_{k_1-k_2}^2 - (\omega_1 - \omega_2)^2} \\
    &\quad + f_+(k_1,k_3)f_+(k,k_2) \bigg[ \frac{4(\omega_1 - \omega_3)^2}{\omega_{k_1-k_3}^2 - (\omega_1 - \omega_3)^2} - \frac{4\omega_3^2}{2\omega_1\omega_3} + \frac{8\omega_1\omega_3}{2\omega_1\omega_3} - \frac{4\omega_3^2(|k_1|+k_1\cdot\hat k_3)}{(2\omega_1\omega_3)^2} \bigg] \Bigg\}.
\end{aligned}
\end{equation}
Following this decomposition, we symmetrize each component individually by defining
\begin{equation}
    T_{k_1 k,j}^{k_2k_3}:=\frac{1}{2}(\tilde T_{k_1 k,j}^{k_2k_3}+\tilde T_{k_2k_3,j}^{k_1 k}),\qquad j=1,2,3.
\end{equation}

Next, we systematically evaluate the three components of the kernel decomposition. We begin by proving the following lemma for the leading-order term.

\begin{lem}\label{lem: T1}
On the restricted support defined by $\delta(\Sigma)\delta(\Omega)\varphi_{1<3}\varphi_{2<3}$, there exists a decomposition 
\begin{equation}\label{eq: decom T1}
    T_{k_1k,1}^{k_2k_3} = M_{k_1k,1}^{k_2k_3} + R_{k_1k,1}^{k_2k_3}
\end{equation}
where the leading-order term $M_{k_1 k,1}^{k_2k_3}$ is given exactly by 
\begin{equation}\label{eq: def M1}
   M_{k_1 k,1}^{k_2k_3} \coloneqq \frac{(\alpha-\beta)^2|k|^{3/2}(\omega_1\omega_2)^{3/2}}{16\pi^2}
\end{equation}
and the remainder term $R_{k_1 k,1}^{k_2k_3} \coloneqq T_{k_1k,1}^{k_2k_3} - M_{k_1k,1}^{k_2k_3}$ satisfies the bound
\begin{equation}
    |R_{k_1 k,1}^{k_2k_3}| \leq C_{1}(|k_1|^2+|k_2|^2)|k|
\end{equation}
for some positive constant $C_{1}>0$.
\end{lem}

\begin{proof}
We begin by introducing the intermediate quantity $\tilde T_{k_1k,1}^{k_2k}$ and writing 
\begin{equation}
    \tilde T_{k_1k,1}^{k_2k_3} = \tilde T_{k_1k,1}^{k_2k} + \big(\tilde T_{k_1k,1}^{k_2k_3} - \tilde T_{k_1k,1}^{k_2k}\big).
\end{equation}
Crucially, the leading-order terms in the kernel exhibit a quadratic growth with respect to $|k|$ or $|k_3|$. Taking the difference effectively cancels these highest-order contributions, reducing the local growth degree by one. Since the overall degree of the kernel with respect to all variables $(k_1, k_2, k_3, k)$ is $3$, we can combine the estimates from Lemma~\ref{lem: basic kernel} with the scale separations~\eqref{eq: k12 and k3} and the bounds on $|k|$ in~\eqref{eq: range k} to deduce the bound
\begin{equation}\label{eq: d t T1 1}
   \big| \tilde T_{k_1k,1}^{k_2k_3} - \tilde T_{k_1k,1}^{k_2k} \big| \leq C_{11}(|k_1|^{2}+|k_2|^{2})|k|,
\end{equation}
for some universal constant $C_{11}>0$ (see Lemma~\ref{lem: C11 est} for the rigorous proof). Next, substituting the angular variables from Eq.~\eqref{eq: ab}, we have the identities
\begin{equation}\label{eq: fpm a}
    f_\pm(k_1,k) = |k||k_1|(\alpha\pm 1) \quad \text{and} \quad f_\pm(k_2,k) = |k||k_2|(\beta\pm 1).
\end{equation}
Inserting these into Eq.~\eqref{eqL tT 1} evaluated at $k_3 = k$ yields 
\begin{equation}
\begin{aligned}
    \tilde{T}_{k_1 k,1}^{k_2 k} &= -\frac{1}{16\pi^2 (|k_1| |k|^2 |k_2| )^{1/4}} \Bigg\{ - 2\omega^2 \Big[ \omega_2 \omega |k_1||k|(\alpha-1) + \omega_1 \omega |k_2||k|(\beta-1)\Big] \\
    &\quad - 2\omega^2 \Big[ \omega \omega_2 |k_1||k|(\alpha+1) + \omega_1 \omega |k||k_2|(\beta+1)\Big] \\
    &\quad - \frac{2\omega |k_1||k_2||k|^2(\alpha-1)(\beta-1)}{\omega_1} + \frac{2\omega |k_1||k_2||k|^2(\alpha+1)(\beta+1)}{\omega_1} \Bigg\}.
\end{aligned}
\end{equation}
Simplifying the terms inside the brackets, we obtain
\begin{equation}
\begin{aligned}
    \tilde{T}_{k_1 k,1}^{k_2 k} &= \frac{1}{4\pi^2 (|k_1| |k|^2 |k_2| )^{1/4}} \Bigg\{ \omega^2 \Big[ \omega_2 \omega |k_1||k|\alpha + \omega_1 \omega |k_2||k|\beta \Big] - \frac{\omega |k_1||k_2||k|^2(\alpha+\beta)}{\omega_1} \Bigg\} \\
    &= \frac{\alpha|k|^{2}\sqrt{\omega_1\omega_2}(\omega_1-\omega_2)}{4\pi^2}.
\end{aligned}
\end{equation}
Symmetrizing this expression yields 
\begin{equation}
    T_{k_1 k,1}^{k_2 k} = \frac{1}{2} \big( \tilde{T}_{k_1 k,1}^{k_2 k} + \tilde{T}_{k_2 k,1}^{k_1 k} \big) = \frac{(\alpha-\beta)|k|^{2}\sqrt{\omega_1\omega_2}(\omega_1-\omega_2)}{8\pi^2}.
\end{equation}
On the resonant manifold supported by $\delta(\Sigma)\delta(\Omega)$, we exploit the frequency conservation $\omega_1-\omega_2 = \omega_3-\omega$ to expand
\begin{equation}\label{eq: omega1-2}
\begin{aligned}
    \omega_1-\omega_2 &= \frac{|k+k_1-k_2|^2-|k|^2}{(\omega_3+\omega)(|k_3|+|k|)} = \frac{2k\cdot (k_1-k_2)+|k_1-k_2|^2}{(\omega_3+\omega)(|k_3|+|k|)} \\
    &= \frac{(\alpha-\beta)\omega_1\omega_2}{2\omega} + \mathcal{O}_1\!\left(\frac{1}{|k|}\right),
\end{aligned}
\end{equation}
where the precise algebraic remainder is given by
\begin{equation}
\begin{aligned}
    \mathcal{O}_1\!\left(\frac{1}{|k|}\right) \coloneqq &\frac{2k\cdot (k_1-k_2)+|k_1-k_2|^2}{(\omega_3+\omega)(|k_3|+|k|)} - \frac{2k\cdot (k_1-k_2)+|k_1-k_2|^2}{4\omega^3} \\
    &+ \frac{|k_1-k_2|^2}{4\omega^3} + \frac{\alpha \omega_1(\omega_1-\omega_2)+\beta\omega_2(\omega_2-\omega_1)}{2\omega}.
\end{aligned}
\end{equation}
Substituting this expansion back into the symmetrized kernel gives 
\begin{equation}\label{eq: Tkk 12 1}
    T_{k_1 k,1}^{k_2 k} = \frac{(\alpha-\beta)^2|k|^{3/2}(\omega_1\omega_2)^{3/2}}{16\pi^2} + \tilde{\mathcal{O}}_1(|k|),
\end{equation}
where the corresponding remainder is explicitly 
\begin{equation}
\begin{aligned}
    \tilde{\mathcal{O}}_1(|k|) \coloneqq &\frac{(\alpha-\beta)|k|^{2}\sqrt{\omega_1\omega_2}}{8\pi^2} \Bigg\{ \frac{2k\cdot (k_1-k_2)+|k_1-k_2|^2}{(\omega_3+\omega)(|k_3|+|k|)} \\
    &- \frac{2k\cdot (k_1-k_2)+|k_1-k_2|^2}{4\omega^3} + \frac{|k_1-k_2|^2}{4\omega^3} + \frac{\alpha \omega_1(\omega_1-\omega_2)+\beta\omega_2(\omega_2-\omega_1)}{2\omega} \Bigg\}.
\end{aligned}
\end{equation}
Using the kinematic bounds established in Lemma~\ref{lem: basic kernel}, alongside Eqs.~\eqref{eq: k12 and k3} and \eqref{eq: range k}, we can uniformly bound this remainder by
\begin{equation}
    \big| \tilde{\mathcal{O}}_1(|k|) \big| \leq C_{12}(|k_1|^2+|k_2|^2) |k|
\end{equation}
for some positive constant $C_{12}>0$. Combining this bound with Eq.~\eqref{eq: d t T1 1} and Eq.~\eqref{eq: Tkk 12 1} immediately yields the desired decomposition \eqref{eq: decom T1} with the total constant $C_1 = C_{11} + C_{12}$.
\end{proof}

Proceeding to the next component of the kernel decomposition, we establish a similar structural bound for the sub-leading term $T_{k_1 k, 2}^{k_2 k_3}$, extracting an explicit leading-order piece that will ultimately cancel the one found in Lemma~\ref{lem: T1}.

\begin{lem}\label{lem: T2}
On the restricted support defined by $\delta(\Sigma)\delta(\Omega)\varphi_{1<3}\varphi_{2<3}$, there exists a decomposition  
\begin{equation}\label{eq: Tkk1k2k3 2}
    T_{k_1 k,2}^{k_2k_3} = M_{k_1 k,2}^{k_2k_3} + R_{k_1 k,2}^{k_2k_3},
\end{equation}
where the principal part $M_{k_1 k,2}^{k_2k_3}$ is exactly given by 
\begin{equation}
    M_{k_1 k,2}^{k_2 k_3} = -\frac{(\alpha-\beta)^2(\omega_1\omega_2)^{3/2}|k|^{3/2}}{16\pi^2}
\end{equation}
and the remainder $R_{k_1k,2}^{k_2k_3} \coloneqq T_{k_1k,2}^{k_2k_3} - M_{k_1k,2}^{k_2k_3}$ satisfies the estimate 
\begin{equation}
    |R_{k_1k,2}^{k_2k_3}| \leq C_{2}(|k_1|^{5/2}+|k_2|^{5/2})\omega
\end{equation}
for some positive constant $C_{2}>0$.
\end{lem}

\begin{proof}
Similar to the previous lemma, we introduce the evaluation at $k_3 = k$ and write 
\begin{equation}
    \tilde T_{k_1k,2}^{k_2k_3} = \tilde T_{k_1k,2}^{k_2k} + \big(\tilde T_{k_1k,2}^{k_2k_3} - \tilde T_{k_1k,2}^{k_2k}\big).
\end{equation}
Using the bounds from Lemma~\ref{lem: basic kernel}, together with Eqs.~\eqref{eq: k12 and k3} and \eqref{eq: range k}, we estimate the difference as
\begin{equation}\label{eq: d t T1 2}
   \big| \tilde T_{k_1k,2}^{k_2k_3} - \tilde T_{k_1k,2}^{k_2k} \big| \leq C_{21}(|k_1|^{5/2}+|k_2|^{5/2})\omega
\end{equation}
for some positive constant $C_{21}>0$ (see Lemma~\ref{lem: C21 est} for the rigorous proof). Substituting the angular relations \eqref{eq: ab} and \eqref{eq: fpm a} into Eq.~\eqref{eq: tT2 kk1k2k3} evaluated at $k_3=k$ yields 
\begin{equation}
\begin{aligned}
    \tilde{T}_{k_1 k,2}^{k_2 k} &= -\frac{1}{16\pi^2 (|k_1| |k|^2 |k_2|)^{1/4}} \Bigg\{ -12 |k_1| |k_2| |k|^2 \\
    &\quad - 4\omega_1\omega \Big[ \omega_2|k_1| |k|^{3/2}(\alpha-1) + \omega_1 |k_2||k|^{3/2}(\beta-1) \Big] \\
    &\quad + 4\omega_1\omega \Big[ \omega_2|k_1||k|^{3/2}(\alpha+1) + \omega_1 |k_2||k|^{3/2}(\beta+1)\Big] \\
    &\quad + |k_1||k_2||k|^2(\alpha+1)(\beta+1) + |k_1||k_2||k|^2(\alpha-1)(\beta-1) \\
    &\quad - 4|k_1||k_2||k|^2(\alpha-1)(\beta-1) + |k_1||k_2||k|^2(\alpha-1)(\beta-1)(1-\alpha) \\
    &\quad - 4|k_1||k_2||k|^2(\alpha+1)(\beta+1) + |k_1||k_2||k|^2(\alpha+1)(\beta+1)(1+\alpha) \Bigg\},
\end{aligned}
\end{equation}
which simplifies by grouping terms to 
\begin{equation}
\begin{aligned}
    \tilde{T}_{k_1 k,2}^{k_2 k} &= -\frac{1}{16\pi^2 (|k_1| |k|^2 |k_2|)^{1/4}} \Bigg\{ -12 |k_1| |k_2| |k|^2  + 8\omega_1\omega \Big[ \omega_2|k_1| |k|^{3/2} + \omega_1 |k_2||k|^{3/2}\Big] \\
    &\quad + |k_1||k_2||k|^2(\alpha-1)(\beta-1)(-2-\alpha) + |k_1||k_2||k|^2(\alpha+1)(\beta+1)(-2+\alpha) \Bigg\}.
\end{aligned}
\end{equation}
Factoring out the shared leading order $|k_1||k_2||k|^2$, we extract the exact angular dependence alongside a remainder:
\begin{equation}
\begin{aligned}
    \tilde{T}_{k_1 k,2}^{k_2 k} &= -\frac{|k_1||k_2||k|^2}{16\pi^2 (|k_1| |k|^2 |k_2|)^{1/4}} \Big\{ -12 + 16 \\
    &\quad + (\alpha-1)(\beta-1)(-2-\alpha) + (\alpha+1)(\beta+1)(-2+\alpha) \Big\} + \mathcal{O}_2(|k|) \\
    &= -\frac{\alpha (\alpha-\beta)(\omega_1\omega_2)^{3/2}|k|^{3/2}}{8\pi^2} + \mathcal{O}_2(|k|),
\end{aligned}
\end{equation}
where the algebraic remainder is given by
\begin{equation}
    \mathcal{O}_2(|k|) \coloneqq -\frac{|k|^{3/2}\omega_1^{3/2}\sqrt{\omega_2}(\omega_1-\omega_2)}{2\pi^2}.
\end{equation}
Subsequently, symmetrizing this expression gives
\begin{equation}
    T_{k_1 k,2}^{k_2 k} = \frac{1}{2} \big( \tilde{T}_{k_1 k,2}^{k_2 k} + \tilde{T}_{k_2 k,2}^{k_1 k} \big) = -\frac{(\alpha-\beta)^2(\omega_1\omega_2)^{3/2}|k|^{3/2}}{16\pi^2} + \tilde{\mathcal{O}}_2(\omega),
\end{equation}
with the symmetrized remainder defined as 
\begin{equation}
    \tilde{\mathcal{O}}_2(\omega) \coloneqq -\frac{|k|^{3/2}\sqrt{\omega_1\omega_2}(\omega_1-\omega_2)^2}{4\pi^2}.
\end{equation}
By the kinematic constraints, this remainder is bounded by 
\begin{equation}
    \big| \tilde{\mathcal{O}}_2(\omega) \big| \leq C_{22}(|k_1|^{5/2}+|k_2|^{5/2})\omega
\end{equation}
for some positive constant $C_{22}>0$. Combining this bound with Eq.~\eqref{eq: d t T1 2} completes the proof with $C_{2} = C_{21} + C_{22}$.
\end{proof}

Finally, we address the third component of the kernel decomposition. Because $T_{k_1 k, 3}^{k_2 k_3}$ consists exclusively of lower-order remainder terms, it requires no further principal part extraction and can be uniformly bounded directly; accordingly, we defer the detailed proof of Lemma~\ref{lem: T3} to Appendix~\ref{app: Cauchy}.

\begin{lem}\label{lem: T3}
On the restricted support defined by $\delta(\Sigma)\delta(\Omega)\varphi_{1<3}\varphi_{2<3}$, the remainder component $T_{k_1 k,3}^{k_2k_3}$ satisfies the bound 
\begin{equation}\label{eq: last est T3}
    |T_{k_1 k,3}^{k_2k_3}| \leq C_{3}(|k_1|^2+|k_2|^2) |k|
\end{equation}
for some positive constant $C_{3}>0$.
\end{lem}

\subsection{The \texorpdfstring{$\tilde T_{k k_1}^{k_2 k_3}+\tilde T_{k_3 k_2}^{k_1 k}$}{T\_{k,k1}\^{k2,k3} + T\_{k3,k2}\^{k1,k}} contribution}

Based on Eq.~\eqref{eq: tTk1k2k3k4}, the expression for $\tilde T_{k k_1}^{k_2 k_3}$ reads
\begin{equation}\label{eq: tTk1kk2k3}
\begin{aligned}
    \tilde{T}_{k k_1}^{k_2 k_3} &= -\frac{1}{16\pi^2 (|k| |k_1| |k_2| |k_3|)^{1/4}} \Bigg\{ -12 |k| |k_1| |k_2| |k_3| \\
    &\quad - 2(\omega + \omega_1)^2 \Big[ \omega_2 \omega_3 f_-(k,k_1) + \omega \omega_1 f_-(k_2,k_3) \Big] \\
    &\quad - 2(\omega - \omega_2)^2 \Big[ \omega_1 \omega_3 f_+(k,k_2) + \omega \omega_2 f_+(k_1,k_3) \Big] \\
    &\quad - 2(\omega - \omega_3)^2 \Big[ \omega_1 \omega_2 f_+(k,k_3) + \omega \omega_3 f_+(k_1,k_2)\Big] \\
    &\quad + f_+(k,k_1)f_+(k_2,k_3)+f_-(k,k_2)f_-(k_1,k_3)+f_-(k,k_3)f_-(k_1,k_2) \\
    &\quad + \frac{4(\omega + \omega_1)^2 f_-(k,k_1)f_-(k_2,k_3)}{\omega_{k+k_1}^2 - (\omega + \omega_1)^2} + \frac{4(\omega - \omega_2)^2 f_+(k,k_2)f_+(k_1,k_3)}{\omega_{k-k_2}^2 - (\omega - \omega_2)^2} \\
    &\quad + \frac{4(\omega - \omega_3)^2f_+(k,k_3)f_+(k_1,k_2)}{\omega_{k-k_3}^2 - (\omega - \omega_3)^2} \Bigg\}.
\end{aligned}
\end{equation}
On the restricted support of $\delta(\Sigma)\delta(\Omega)\varphi_{1<3}\varphi_{2<3}$, we decompose $\tilde{T}_{k k_1}^{k_2 k_3}$ into three components:
\begin{equation}
    \tilde{T}_{k k_1}^{k_2 k_3} = \tilde{T}_{k k_1,1}^{k_2 k_3} + \tilde{T}_{k k_1,2}^{k_2 k_3} + \tilde{T}_{k k_1,3}^{k_2 k_3},
\end{equation}
where the first piece $\tilde{T}_{k k_1,1}^{k_2 k_3}$ is defined as the sum of the leading-order terms of $O(|k_3|^{2})$ from the second, third, eighth, and ninth terms within the braces of Eq.~\eqref{eq: tTk1kk2k3}:
\begin{equation}\label{eqL tT 2}
\begin{aligned}
    \tilde{T}_{k k_1,1}^{k_2 k_3} \coloneqq & -\frac{1}{16\pi^2 (|k_1| |k| |k_2| |k_3|)^{1/4}} \Bigg\{ - 2\omega^2 \Big[ \omega_2 \omega_3 f_-(k,k_1) + \omega \omega_1 f_-(k_2,k_3) \Big] \\
    &\quad - 2\omega^2 \Big[ \omega_1 \omega_3 f_+(k,k_2) + \omega \omega_2 f_+(k_1,k_3)\Big] \\
    &\quad + \frac{4 \omega^2 f_-(k,k_1)f_-(k_2,k_3)}{-2\omega_1\omega} + \frac{4 \omega^2f_+(k,k_2)f_+(k_1,k_3)}{2\omega\omega_2} \Bigg\}.
\end{aligned}
\end{equation}
The second piece $\tilde T_{kk_1,2}^{k_2k_3}$ comprises the terms of $O(|k_3|^{3/2})$ in Eq.~\eqref{eq: tTk1kk2k3}:
\begin{equation}\label{eq: tT2 kk1k2k3 2}
\begin{aligned}
    \tilde{T}_{k k_1,2}^{k_2 k_3} \coloneqq &-\frac{1}{16\pi^2 (|k_1| |k| |k_2| |k_3|)^{1/4}} \Bigg\{ -12 |k_1| |k| |k_2| |k_3| \\
    &\quad - 4\omega_1\omega \Big[ \omega_2 \omega_3 f_-(k,k_1) + \omega \omega_1 f_-(k_2,k_3) \Big] + 4\omega\omega_2 \Big[ \omega_1 \omega_3 f_+(k,k_2) + \omega \omega_2 f_+(k_1,k_3)\Big] \\
    &\quad + f_+(k,k_1)f_+(k_2,k_3)+f_-(k,k_2)f_-(k_1,k_3) \\
    &\quad + \frac{8\omega_1\omega f_-(k,k_1)f_-(k_2,k_3)}{-2\omega_1\omega} + \frac{4\omega^2 f_-(k,k_1)f_-(k_2,k_3)(|k_1|-k_1\cdot \hat k)}{(-2\omega_1\omega)^2} \\
    &\quad + \frac{-8\omega_2\omega f_+(k,k_2)f_+(k_1,k_3)}{2\omega\omega_2} + \frac{4\omega^2 f_+(k,k_2)f_+(k_1,k_3)(|k_2|+k_2\cdot \hat k)}{(2\omega\omega_2)^2} \Bigg\}.
\end{aligned}
\end{equation}
Finally, the third piece $\tilde T_{k k_1,3}^{k_2 k_3}$ is the remainder term:
\begin{equation}\label{eq: tT3 k1kk2k3}
\begin{aligned}
    \tilde{T}_{k k_1,3}^{k_2 k_3} \coloneqq &-\frac{1}{16\pi^2 (|k_1| |k| |k_2| |k_3|)^{1/4}} \Bigg\{ - 2\omega_1^2 \Big[ \omega_2 \omega_3 f_-(k_1,k) + \omega_1 \omega f_-(k_2,k_3) \Big] \\
    &\quad - 2\omega_2^2 \Big[ \omega_1 \omega_3 f_+(k,k_2) + \omega \omega_2 f_+(k_1,k_3) \Big] \\
    &\quad - 2(\omega-\omega_3)^2 \Big[ \omega_1 \omega_2 f_+(k,k_3) + \omega \omega_3 f_+(k_1,k_2)\Big] + f_-(k_1,k_2)f_-(k,k_3) \\
    &\quad + f_-(k_1,k)f_-(k_2,k_3) \bigg\{ \frac{4(\omega_1 + \omega)^2}{\omega_{k_1+k}^2 - (\omega_1 + \omega)^2} - \frac{4\omega^2}{-2\omega_1\omega} - \frac{8\omega_1\omega}{-2\omega_1\omega} - \frac{4\omega^2(|k_1|-k_1\cdot\hat k)}{(-2\omega_1\omega)^2} \bigg\} \\
    &\quad + \frac{4(\omega - \omega_3)^2 f_+(k_1,k_2)f_+(k,k_3)}{\omega_{k-k_3}^2 - (\omega - \omega_3)^2} \\
    &\quad + f_+(k_1,k_3)f_+(k,k_2) \bigg\{ \frac{4(\omega - \omega_2)^2}{\omega_{k-k_2}^2 - (\omega - \omega_2)^2} - \frac{4\omega^2}{2\omega\omega_2} - \frac{8\omega\omega_2}{2\omega\omega_2} - \frac{4\omega^2(|k_2|+k_2\cdot\hat k)}{(2\omega\omega_2)^2} \bigg\} \Bigg\}.
\end{aligned}
\end{equation}
Following this decomposition, we define the symmetrized components
\begin{equation}
    T_{k k_1,j}^{k_2k_3} \coloneqq \frac{1}{2}\big(\tilde T_{k k_1,j}^{k_2k_3} + \tilde T_{k_3k_2,j}^{k_1 k}\big), \qquad j \in \{1,2,3\}.
\end{equation}

Next, we establish the properties of these three components via the following three lemmas.

\begin{lem}\label{lem: T4}
On the restricted support defined by $\delta(\Sigma)\delta(\Omega)\varphi_{1<3}\varphi_{2<3}$, there exists a decomposition 
\begin{equation}\label{eq: decom Tkk1k2k3 1}
    T_{kk_1,1}^{k_2k_3} = M_{kk_1,1}^{k_2k_3} + R_{kk_1,1}^{k_2k_3}
\end{equation}
with the principal part $M_{k k_1,1}^{k_2k_3}$ given exactly by 
\begin{equation}
   M_{k k_1,1}^{k_2k_3} = \frac{(\alpha-\beta)^2|k|^{3/2}(\omega_1\omega_2)^{3/2}}{16\pi^2(|k_1||k_2||k|^2)^{1/4}}
\end{equation}
and the remainder $R_{kk_1,1}^{k_2k_3} \coloneqq T_{kk_1,1}^{k_2k_3} - M_{kk_1,1}^{k_2k_3}$ satisfying the estimate
\begin{equation}
    |R_{k k_1,1}^{k_2k_3}| \leq C_{4}(|k_1|^2+|k_2|^2)|k|
\end{equation}
for some positive constant $C_4>0$.
\end{lem}

\begin{proof}
We begin by evaluating at $k_3=k$ and writing 
\begin{equation}
    \tilde T_{kk_1,1}^{k_2k_3} = \tilde T_{kk_1,1}^{k_2k} + \big(\tilde T_{kk_1,1}^{k_2k_3} - \tilde T_{kk_1,1}^{k_2k}\big).
\end{equation}
Using the kinematic bounds from Lemma~\ref{lem: basic kernel}, together with Eqs.~\eqref{eq: k12 and k3} and \eqref{eq: range k}, we obtain
\begin{equation}\label{eq: d t T2 q}
   \big| \tilde T_{kk_1,1}^{k_2k_3} - \tilde T_{kk_1,1}^{k_2k} \big| \leq C_{41}(|k_1|^2+|k_2|^2)|k|
\end{equation}
for some positive constant $C_{41}>0$ (see Lemma~\ref{lem: C41 est} for the rigorous proof). Substituting the angular variables from Eqs.~\eqref{eq: ab} and \eqref{eq: fpm a} into Eq.~\eqref{eqL tT 2} evaluated at $k_3=k$ yields 
\begin{equation}
\begin{aligned}
    \tilde{T}_{k k_1,1}^{k_2 k} &= -\frac{1}{16\pi^2 (|k_1| |k|^2 |k_2|)^{1/4}} \Bigg\{ - 2\omega^2 \Big[ \omega_2 |k_1||k|^{3/2} (\alpha-1) +  \omega_1 |k_2||k|^{3/2}(\beta-1) \Big] \\
    &\quad - 2\omega^2 \Big[ \omega_1 |k_2||k|^{3/2}(\beta+1) + \omega_2|k_1||k|^{3/2}(\alpha+1)\Big] \\
    &\quad - \frac{2 \omega |k_1||k_2||k|^2(\alpha-1)(\beta-1)}{\omega_1} + \frac{2 \omega |k_1||k_2||k|^2(\alpha+1)(\beta+1)}{\omega_2} \Bigg\}.
\end{aligned}
\end{equation}
Simplifying the terms inside the brackets gives
\begin{equation}
\begin{aligned}
    \tilde{T}_{k k_1,1}^{k_2 k} &= -\frac{1}{16\pi^2 (|k_1| |k|^2 |k_2|)^{1/4}} \Bigg\{ - 4\omega^2 \Big[ \omega_2 |k_1||k|^{3/2} \alpha +  \omega_1 |k_2||k|^{3/2}\beta \Big] \\
    &\quad - 2\left(\frac{1}{\omega_1}-\frac{1}{\omega_2}\right) \omega |k_1||k_2||k|^2(\alpha-1)(\beta-1) + \frac{4 \omega |k_1||k_2||k|^2(\alpha+\beta)}{\omega_2} \Bigg\} \\
    &= -\frac{1}{16\pi^2 (|k_1| |k|^2 |k_2|)^{1/4}} \Bigg\{ -4\beta(\omega_2-\omega_1)\omega_1\omega_2|k|^{5/2} \\
    &\quad - 2(\omega_2-\omega_1)\omega_1\omega_2|k|^{5/2}(\alpha-1)(\beta-1)\Bigg\}.
\end{aligned}
\end{equation}
Using the frequency expansion from Eq.~\eqref{eq: omega1-2}, the symmetrized kernel becomes
\begin{equation}
\begin{aligned}
     T_{k k_1,1}^{k_2 k} &= \frac{1}{2}\big( \tilde{T}_{k k_1,1}^{k_2 k} + \tilde{T}_{k k_2,1}^{k_1 k} \big) \\
     &= \frac{1}{8\pi^2(|k_1||k_2||k|^2)^{1/4}}(\alpha-\beta)\cdot\frac{(\alpha-\beta)\omega_1\omega_2}{2\omega}\cdot\omega_1\omega_2|k|^{5/2} + \tilde{\mathcal{O}}_3(|k|) \\
     &= \frac{(\alpha-\beta)^2|k|^{3/2}(\omega_1\omega_2)^{3/2}}{16\pi^2} + \tilde{\mathcal{O}}_3(|k|),
\end{aligned}
\end{equation}
where the remainder is defined as
\begin{equation}
    \tilde{\mathcal{O}}_3(|k|) \coloneqq \frac{1}{8\pi^2(|k_1||k_2||k|^2)^{1/4}}(\alpha-\beta) \cdot \mathcal{O}_1\!\left(\frac{1}{|k|}\right) \cdot \omega_1\omega_2|k|^{5/2}.
\end{equation}
Applying the bounds from Lemma~\ref{lem: basic kernel},~\eqref{eq: k12 and k3}, and~\eqref{eq: range k} yields
\begin{equation}
    \big|\tilde{\mathcal{O}}_3(|k|)\big| \leq C_{42}(|k_1|^2+|k_2|^2)|k|
\end{equation}
for some positive constant $C_{42}>0$. Taking $C_{4} = C_{41} + C_{42}$ completes the proof. 
\end{proof}

\begin{lem}\label{lem: T5}
On the restricted support defined by $\delta(\Sigma)\delta(\Omega)\varphi_{1<3}\varphi_{2<3}$, there exists a decomposition 
\begin{equation}\label{eq: Tkk1k2k3 2 2}
    T_{kk_1,2}^{k_2k_3} = M_{kk_1,2}^{k_2k_3} + R_{kk_1,2}^{k_2k_3}
\end{equation}
with the sub-leading principal part given by 
\begin{equation}\label{eq: tT2 decom M2 2}
      M_{kk_1,2}^{k_2k_3} \coloneqq -\frac{(\alpha-\beta)^2 |k|^{3/2}(\omega_1\omega_2)^{3/2}}{16\pi^2}
\end{equation}
and the remainder $R_{kk_1,2}^{k_2k_3} \coloneqq T_{kk_1,2}^{k_2k_3} - M_{kk_1,2}^{k_2k_3}$ satisfying the estimate
\begin{equation}
    |R_{kk_1,2}^{k_2k_3}| \leq C_{5}(|k_1|^{5/2}+|k_2|^{5/2})\omega
\end{equation}
for some positive constant $C_5>0$.
\end{lem}

\begin{proof}
We write the difference to the evaluated kernel as
\begin{equation}
    \tilde T_{kk_1,2}^{k_2k_3} = \tilde T_{kk_1,2}^{k_2k} + \big(\tilde T_{kk_1,2}^{k_2k_3} - \tilde T_{kk_1,2}^{k_2k}\big).
\end{equation}
Using the kinematic bounds from Lemma~\ref{lem: basic kernel},~\eqref{eq: k12 and k3}, and~\eqref{eq: range k}, we obtain 
\begin{equation}\label{eq: d t T2 2}
   \big| \tilde T_{kk_1,2}^{k_2k_3} - \tilde T_{kk_1,2}^{k_2k} \big| \leq C_{51}(|k_1|^{5/2}+|k_2|^{5/2})\omega
\end{equation}
for some positive constant $C_{51}>0$ (see Lemma~\ref{lem: C51 est} for the rigorous proof). Next, substituting Eqs.~\eqref{eq: ab} and~\eqref{eq: fpm a} into Eq.~\eqref{eq: tT2 kk1k2k3 2} evaluated at $k_3=k$ yields 
\begin{equation}
\begin{aligned}
    \tilde{T}_{k k_1,2}^{k_2 k} &= -\frac{1}{16\pi^2 (|k_1| |k|^2 |k_2| )^{1/4}} \Bigg\{ -12 |k_1| |k|^2 |k_2|  \\
    &\quad - 4\omega_1\omega \Big[ \omega_2 |k_1||k|^{3/2}(\alpha-1) +  \omega_1 |k_2||k|^{3/2}(\beta-1)\Big] \\
    &\quad + 4\omega\omega_2 \Big[ \omega_1 |k_2||k|^{3/2} (\beta+1) + \omega_2 |k_1||k|^{3/2}(\alpha+1)\Big] \\
    &\quad + |k_1||k_2||k|^2(\alpha+1)(\beta+1) + |k_1||k_2||k|^2(\alpha-1)(\beta-1) \\
    &\quad - 4|k_1||k_2||k|^2(\alpha-1)(\beta-1) + |k_1||k_2||k|^2(\alpha-1)(\beta-1)(1-\alpha) \\
    &\quad - 4|k_1||k_2||k|^2(\alpha+1)(\beta+1) + |k_1||k_2||k|^2(\alpha+1)(\beta+1)(1+\beta) \Bigg\},
\end{aligned}
\end{equation}
which simplifies by factoring the leading-order scale out:
\begin{equation}
\begin{aligned}
    \tilde{T}_{k k_1,2}^{k_2 k} &= -\frac{|k_1| |k|^2 |k_2|}{16\pi^2 (|k_1| |k|^2 |k_2| )^{1/4}} \Bigg\{ -12 + 16  \\
    &\quad + (\alpha+1)(\beta+1)(-2+\beta) + (\alpha-1)(\beta-1)(-2-\alpha) \Bigg\} + \mathcal{O}_4(|k|),
\end{aligned}
\end{equation}
where the remainder is defined algebraically as
\begin{equation}
    \mathcal{O}_4(|k|) \coloneqq \frac{1}{4\pi^2 (|k_1| |k|^2 |k_2| )^{1/4}} \Big\{ \omega_2|k_1||k|^2(\alpha-1)(\omega_1-\omega_2) - \omega_1|k_2||k|^2(\beta+1)(\omega_2-\omega_1)\Big\}.
\end{equation}
Symmetrizing this expression leads to 
\begin{equation}
\begin{aligned}
    T_{k k_1,2}^{k_2 k} &= -\frac{|k_1| |k|^2 |k_2|}{16\pi^2 (|k_1| |k|^2 |k_2| )^{1/4}}\Bigg\{ 4 + (\alpha+1)(\beta+1)\left(-2+\frac{1}{2}(\beta+\alpha)\right) \\
    &\quad + (\alpha-1)(\beta-1)\left(-2-\frac{1}{2}(\alpha+\beta)\right) \Bigg\} + \tilde{\mathcal{O}}_4(\omega),
\end{aligned}
\end{equation}
with the symmetric remainder defined as
\begin{equation}
\begin{aligned}
    \tilde{\mathcal{O}}_4(\omega) \coloneqq &\frac{\sqrt{\omega_1\omega_2}|k|^{3/2}(\omega_1-\omega_2)}{8\pi^2 } \Big\{ \big(\omega_1(\alpha-1)-\omega_2(\beta-1)\big) + \big(\omega_2(\beta+1)-\omega_1(\alpha+1)\big) \Big\} \\
    &= -\frac{\sqrt{\omega_1\omega_2}|k|^{3/2}(\omega_1-\omega_2)^2}{4\pi^2 }. 
\end{aligned}
\end{equation}
Evaluating the algebraic multiplier directly gives
\begin{equation}
\begin{aligned}
    T_{k k_1,2}^{k_2 k} &= -\frac{|k_1| |k|^2 |k_2|}{16\pi^2 (|k_1| |k|^2 |k_2| )^{1/4}}\Bigg\{ 4 - 4(\alpha\beta+1) + (\alpha+\beta)^2 \Bigg\} + \tilde{\mathcal{O}}_4(\omega) \\
    &= -\frac{(\alpha-\beta)^2 |k|^{3/2}(\omega_1\omega_2)^{3/2}}{16\pi^2 } + \tilde{\mathcal{O}}_4(\omega).
\end{aligned}
\end{equation}
Again utilizing Lemma~\ref{lem: basic kernel},~\eqref{eq: k12 and k3} and~\eqref{eq: range k}, the remainder satisfies the bound
\begin{equation}
    \big|\tilde{\mathcal{O}}_4(\omega)\big| \leq C_{52}(|k_1|^{5/2}+|k_2|^{5/2})\omega
\end{equation}
for some positive constant $C_{52}>0$. This completes the proof with $C_5=C_{51}+C_{52}$.
\end{proof}

\begin{lem}\label{lem: T6}
On the restricted support defined by $\delta(\Sigma)\delta(\Omega)\varphi_{1<3}\varphi_{2<3}$, the remainder component $T_{k k_1,3}^{k_2k_3}$ satisfies the uniform bound 
\begin{equation}
    |T_{k k_1,3}^{k_2k_3}| \leq C_{6}(|k_1|^2+|k_2|^2) |k|
\end{equation}
for some positive constant $C_{6}>0$.
\end{lem}
Since $T_{k k_1,3}^{k_2k_3}$ consists of lower-order terms, we defer the proof of Lemma~\ref{lem: T6} to Appendix~\ref{app: Cauchy}.

\begin{proof}[Proof of Proposition~\ref{prop: kernel}]
By assembling the bounds established in Lemmas~\ref{lem: T1} through~\ref{lem: T6}, and recalling the full kernel decomposition in Eq.~\eqref{eq: def T}, we can now evaluate the complete interaction kernel. For formal consistency, we identify the third components entirely as remainder terms by defining $R_{k_1k,3}^{k_2k_3} \coloneqq T_{k_1k,3}^{k_2k_3}$ and $R_{kk_1,3}^{k_2k_3} \coloneqq T_{kk_1,3}^{k_2k_3}$.

Crucially, the principal singular parts extracted in the preceding lemmas perfectly annihilate each other:
\begin{equation}
    M_{k_1k,1}^{k_2k_3} + M_{k_1k,2}^{k_2k_3} = 0 \quad \text{and} \quad M_{kk_1,1}^{k_2k_3} + M_{kk_1,2}^{k_2k_3} = 0.
\end{equation}
Exploiting this exact algebraic cancellation, the total kernel reduces entirely to a sum of bounded remainders. Applying the triangle inequality, we obtain
\begin{equation}
\begin{aligned}
    |T_{k_1k}^{k_2k_3}| &\leq \frac{1}{2}\left|\sum_{j=1}^3 \big(T_{k_1k,j}^{k_2k_3} + T_{kk_1,j}^{k_2k_3}\big)\right| \\
    &= \frac{1}{2}\left|\sum_{j=1}^3 \big(R_{k_1k,j}^{k_2k_3} + R_{kk_1,j}^{k_2k_3}\big)\right| \\
    &\leq \frac{1}{2} \Bigg(\sum_{j=1}^6 C_j\Bigg) (|k_1|^2+|k_2|^2)|k|.
\end{aligned}
\end{equation}
Squaring this expression yields Eq.~\eqref{eq: est kernel'} with the aggregate constant defined as $\tilde{C}_Q = \frac{1}{4}\big(\sum_{j=1}^6 C_j\big)^2$. Consequently, we directly obtain the final desired bound \eqref{eq: est kernel} for some strictly positive constant $C_Q$. This completes the proof of the proposition.
\end{proof}

\section{Dissipativity and Operator Bounds}\label{sec: A-Est}

For $g \in \Bc$, $p \in [1,\infty]$ and $\epsilon\in [0,1]$, let $\mathcal{Q}_{g,D,\epsilon}(t) \colon L^p_{2} \to L^p$ be the operator defined by
\begin{equation}\label{eq: QDg}
\begin{aligned} \mathcal{Q}_{g,D,\epsilon}(t) h &= 2\iiint_{\mathbb{R}^{3d}} e^{-\epsilon(|k|+|k_3|)}T_{k,k_1,k_2,k_3}\delta(\Sigma)\delta(\Omega) \chi_{2<3}\chi_{1<0}g_1(t)g_2(t) (h_3-h)\,dk_1\,dk_2\, dk_3,
\end{aligned}
\end{equation}
and let $\Qc_{g,b}(t) \colon L^p \to L^p$ be defined by
\begin{equation}\label{eq: Qbg}
\begin{aligned}
    \mathcal{Q}_{g,b}(t) h &= \iiint_{\mathbb{R}^{3d}} T_{k,k_1,k_2,k_3}\delta(\Sigma)\delta(\Omega) \Big[ 2\chi_{2<3}\chi_{1\geq 0}g_1(t)g_2(t) (h_3-h) \\
    &\qquad\qquad + g_2(t)g_3(t)h - 2\chi_{2\geq 3}g_1(t)g_2(t)h \Big] \,dk_1\,dk_2\, dk_3.
\end{aligned}
\end{equation}
\begin{lem}\label{lem: bd QgD} 
Let $g\in \Bc$. Then, for any $1\leq p\leq \infty$, the operator $\Qc_{g,D}(t)$ is bounded from $L^p_2$ to $L^p$.
\end{lem}
We defer the proof of Lemma~\ref{lem: bd QgD} until after Lemma~\ref{lem: Rep1}.
\begin{lem}\label{lem: Qf and Qg} For all $f\in \Bc$, $\Qc_f(t)f(t)=\Qc[f(t)]$ holds true.
\end{lem}
\begin{proof} Based on Eq.~\eqref{Qf}, we rewrite $\Qc[f(t)]$:
    \begin{equation}
        \begin{aligned}
            \Qc[f(t)]=& \iiint_{\mathbb{R}^{3d}} T_{k,k_1,k_2,k_3} \delta(\Sigma)\delta(\Omega) (\chi_{2<3}+\chi_{2\geq 3})\big[f_2 f_3 f_1 - f f_1 (f_2 + f_3)\big]\,dk_1\,dk_2\, dk_3\\
            &+\iiint_{\mathbb{R}^{3d}} T_{k,k_1,k_2,k_3} \delta(\Sigma)\delta(\Omega) f_2 f_3 f\,dk_1\,dk_2\, dk_3.
        \end{aligned}
    \end{equation}
Using the decomposition $\chi_{3 \geq 2} = \chi_{2 < 3} + \chi(|k_2| = |k_3|)$, we note that the resonant manifold intersecting with the diagonal set $\{|k_2| = |k_3|\}$ has measure zero. Consequently, the corresponding integral vanishes for any admissible function $f$:
\begin{equation}
    \iiint_{\mathbb{R}^{3d}} T_{k,k_1,k_2,k_3}\delta(\Sigma)\delta(\Omega)\chi(|k_2|=|k_3|)\big[f_2 f_3 f_1 - f f_1 (f_2 + f_3)\big]\,dk_1\,dk_2\,dk_3 = 0.
\end{equation}
Leveraging this vanishing property, along with the intrinsic symmetry of the collision kernel $T_{k,k_1,k_2,k_3}$ with respect to the variables $k_2$ and $k_3$, we can consolidate the collision operator into the following form:
\begin{equation}
\begin{split}
    \Qc[f(t)] &= 2\iiint_{\mathbb{R}^{3d}} T_{k,k_1,k_2,k_3} \delta(\Sigma)\delta(\Omega) \left( \chi_{2<3}\big(f_2 f_3 f_1 - f f_1 f_2 \big) - 2\chi_{2\geq 3}f_1 f_2 f \right)\,dk_1\,dk_2\,dk_3 \\
    &\quad + \iiint_{\mathbb{R}^{3d}} T_{k,k_1,k_2,k_3} \delta(\Sigma)\delta(\Omega) f_2 f_3 f \,dk_1\,dk_2\,dk_3,
\end{split}
\end{equation}
which yields~$\Qc_f(t)f(t)=\Qc[f(t)]$.
\end{proof}
\begin{lem}\label{lem: QgD diss}
$\Qc_{g,D}(t)$ is dissipative for all $1\leq p<\infty$.
\end{lem}

\begin{proof}
For each $h \in L^p_2$ with $\|h\|_{p}=1$, by Lemma~\ref{lem: bd QgD}, $(h^{p-1},\Qc_{g,D}(t)h)_{L^2}$ is well-defined. Taking the $L^2$ inner product with $h^{p-1}$, exchanging the integration variables $(k_3,k_2) \leftrightarrow (k,k_1)$, and averaging the two expressions yields 
\begin{equation}
\begin{aligned}
    \big(h^{p-1}, \Qc_{g,D}(t)h\big)_{L^2} &= 2\iiiint_{\mathbb{R}^{4d}} T_{k,k_1,k_2,k_3}\delta(\Sigma)\delta(\Omega)\chi_{2<3}\chi_{1<0}g_1(t)g_2(t) h^{p-1}(h_3-h) \,dk_1\,dk_2\,dk_3\,dk \\
    &= 2\iiiint_{\mathbb{R}^{4d}} T_{k,k_1,k_2,k_3}\delta(\Sigma)\delta(\Omega)\chi_{2<3}\chi_{1<0}g_1(t)g_2(t) h_3^{p-1}(h-h_3) \,dk_1\,dk_2\,dk_3\,dk \\
    &= -\iiiint_{\mathbb{R}^{4d}} T_{k,k_1,k_2,k_3}\delta(\Sigma)\delta(\Omega)\chi_{2<3}\chi_{1<0}g_1(t)g_2(t) (h^{p-1}-h_3^{p-1})(h-h_3) \,dk_1\,dk_2\,dk_3\,dk \\
    &\leq 0,
\end{aligned}
\end{equation}
where the non-positivity follows from the monotonicity of $x \mapsto x^{p-1}$. This yields 
\begin{equation}
    \begin{aligned}
        \big\| (\lambda-\Qc_{g,D}(t))h \big\|_{p} &= \sup_{\| g\|_{{p'}}=1} \big| \big(g, (\lambda-\Qc_{g,D}(t))h \big)_{L^2} \big| \\
        &\geq \big(h^{p-1}, (\lambda-\Qc_{g,D}(t))h \big)_{L^2} \\
        &\geq \lambda = \lambda \|h\|_{p},
    \end{aligned}
\end{equation}
where $p'$ denotes the conjugate exponent of $p$ (i.e., $\frac{1}{p'}+\frac{1}{p}=1$), and we have used the fact that 
\begin{equation}
   \| h^{p-1}\|_{{p'}} = \left(\int_{\mathbb{R}^d} |h|^{(p-1)p'}\,dk\right)^{\frac{1}{p'}} = \left(\int_{\mathbb{R}^d} |h|^p\,dk\right)^{\frac{1}{p'}} = 1.
\end{equation}  
\end{proof}

\begin{lem}\label{lem: Rep1}
Define the multidimensional collision integral as:
\begin{equation}
    \tilde{\Qc}_{F1}(k) \coloneqq \iiint_{\mathbb{R}^{3d}} \delta(\Sigma)\delta(\Omega)F(k,k_1,k_2,k_3) \,dk_1 \,dk_2 \,dk_3.
\end{equation}
Then, there exists a dimensional constant $C_d > 0$ such that the following anisotropic bounds hold:
\begin{equation}\label{eq: Qf03}
    |\tilde{\Qc}_{F1}(k)| \leq C_d\int_{\mathbb{R}^d} \frac{1}{|k-k_3|}\big\|\langle k_1\rangle^{d/2+1}\langle k_2\rangle^{d/2+1} F(k,k_1,k_2,k_3)\big\|_{L^\infty_{k_1,k_2}(\mathbb{R}^{2d})}\,dk_3,
\end{equation}
\begin{equation}\label{eq: Qf02}
    |\tilde{\Qc}_{F1}(k)| \leq C_d\int_{\mathbb{R}^d} \frac{1}{|k-k_2|}\big\|\langle k_1\rangle^{d/2+1}\langle k_3\rangle^{d/2+1} F(k,k_1,k_2,k_3)\big\|_{L^\infty_{k_1,k_3}(\mathbb{R}^{2d})}\,dk_2,
\end{equation}
\begin{equation}\label{eq: Qf01}
    |\tilde{\Qc}_{F1}(k)| \leq C_d\int_{\mathbb{R}^d} \frac{1}{|k+k_1|}\big\|\langle k_2\rangle^{d/2+1}\langle k_3\rangle^{d/2+1} F(k,k_1,k_2,k_3)\big\|_{L^\infty_{k_2,k_3}(\mathbb{R}^{2d})}\,dk_1.
\end{equation}
Furthermore, we have the isotropic uniform bound:
\begin{equation}\label{eq:general_integral_reduced}
    |\tilde{\Qc}_{F1}(k)| \leq C_d \big\|\langle k_1\rangle^{d+2}\langle k_2\rangle^{d+2}F(k,k_1,k_2,k_3)\big\|_{L^\infty(\mathbb{R}^{4d})},
\end{equation}
provided that the quantities on the right-hand sides are finite.
\end{lem}

We defer the proof of Lemma~\ref{lem: Rep1} to the end of this section.
\begin{proof}[Proof of Lemma~\ref{lem: bd QgD}] 
We begin with the case $p=1$. Let $h\in L^1_2$. By Proposition~\ref{prop: kernel}, $\Qc_{g,D}(t)h$ satisfies 
\begin{equation}\label{eq: est Qgd}
    \begin{aligned}
        |\Qc_{g,D}(t)h|\leq & \,4C_Q\iiint_{\mathbb R^{3d}} \delta(\Sigma)\delta(\Omega)\chi_{1<0}\chi_{2<3}\left((|k_1|^2+|k_2|^2)^2g_1(t)g_2(t)\right)\\
        &\qquad\qquad\times \left((|k|^2+|k_3|^2)(|h_3|+|h|)\right) \,dk_1\,dk_2\,dk_3.
    \end{aligned}
\end{equation}
Note that on the support of $\delta(\Sigma)\delta(\Omega)\chi_{1<0}\chi_{2<3}$, we have the pointwise bound
\begin{equation}
    (|k|^2+|k_3|^2)(|h_3|+|h|)\leq C\left(|k_3|^2|h_3|+|k|^2|h|\right)
\end{equation}
for some universal constant $C>0$. Integrating \eqref{eq: est Qgd} over $\mathbb{R}^d$ and exploiting the symmetry of the integral under the exchange of variables, we obtain
\begin{equation}
    \begin{aligned}
        \int_{\mathbb R^d}  |\Qc_{g,D}(t)h|\,dk \leq & \,4C_QC\iiiint_{\mathbb R^{4d}}\delta(\Sigma)\delta(\Omega)\left((|k_1|^2+|k_2|^2)^2g_1(t)g_2(t)\right) |k_3|^2|h_3| \,dk_1\,dk_2\,dk_3\,dk\\
        &+4C_QC\iiiint_{\mathbb R^{4d}}\delta(\Sigma)\delta(\Omega)\left((|k_1|^2+|k_2|^2)^2g_1(t)g_2(t)\right) |k|^2|h| \,dk_1\,dk_2\,dk_3\,dk\\
        =& \,8C_QC\iiiint_{\mathbb R^{4d}}\delta(\Sigma)\delta(\Omega)\left((|k_1|^2+|k_2|^2)^2g_1(t)g_2(t)\right) |k|^2|h| \,dk_1\,dk_2\,dk_3\,dk.
    \end{aligned}
\end{equation}
Applying Lemma~\ref{lem: Rep1} then yields
\begin{equation}
    \int_{\mathbb R^d}  |\Qc_{g,D}(t)h|\,dk \leq 32C_QCC_d \|g(t)\|_{\infty,d+6}^2\|h\|_{1,2},
\end{equation}
which establishes the boundedness for $p=1$. 

Next, for the case $p=\infty$, we take a test function $\varphi\in L^1$ with $\|\varphi\|_1=1$ and estimate the dual pairing $|(\varphi, \Qc_{g,D}(t)h)_{L^2}|$. Following a procedure similar to the $p=1$ case, we deduce the bound
\begin{equation}
     \begin{aligned}
        \int_{\mathbb R^d}  |\varphi \Qc_{g,D}(t)h|\,dk \leq & \,4C_QC\iiiint_{\mathbb R^{4d}}\delta(\Sigma)\delta(\Omega)\left((|k_1|^2+|k_2|^2)^2g_1(t)g_2(t)\right) |k_3|^2|\varphi h_3| \,dk_1\,dk_2\,dk_3\,dk\\
        &+4C_QC\iiiint_{\mathbb R^{4d}}\delta(\Sigma)\delta(\Omega)\left((|k_1|^2+|k_2|^2)^2g_1(t)g_2(t)\right) |k|^2|\varphi h| \,dk_1\,dk_2\,dk_3\,dk.
        \end{aligned}
\end{equation}
Furthermore, on the support of $\delta(\Sigma)\delta(\Omega)\chi_{1<0}\chi_{2<3}$, we can bound the first term using
\begin{equation}
    |k_3|^2|\varphi h_3|\leq \|h\|_{\infty,2} |\varphi|.
\end{equation}
Combining this pointwise bound with Lemma~\ref{lem: Rep1} yields
\begin{equation}
    \int_{\mathbb R^d}  |\varphi\Qc_{g,D}(t)h|\,dk \leq 32C_QCC_d \|g(t)\|_{\infty,d+6}^2\|h\|_{\infty,2}\|\varphi\|_1,
\end{equation}
which confirms the boundedness for $p=\infty$. 

Finally, the result for the general case $1 < p < \infty$ follows immediately by Riesz-Thorin interpolation.
\end{proof}

\begin{lem}\label{lem: QgB bd}
$\Qc_{g,b}(t)$ is a bounded operator on $L^p(\mathbb{R}^d)$ for all $1\leq p\leq \infty$ and $s\geq 0$, satisfying
\begin{equation}
    \sup_{t\in (0,T]} \|\Qc_{g,b}(t)\|_{p,s\to p,s} \leq C_b \sup_{t\in (0,T]} \|g(t)\|_{\infty,{10+2d}}^2
\end{equation}
for some universal constant $C_b>0$.
\end{lem}

\begin{proof} 
By the Riesz-Thorin interpolation theorem, it strictly suffices to establish the strong bounds for the endpoint cases $p=\infty$ and $p=1$.

We first decompose the operator into a gain term and a loss term: $\mathcal{Q}_{g,b}(t)h(k) = \mathcal{Q}_{gain}(h)(k) - \nu(k) h(k)$, where the collision frequency multiplier $\nu(k)$ is defined as
\begin{equation}\label{eq: def nuk}
    \nu(k) = \iiint_{\mathbb{R}^{3d}} T_{k,k_1,k_2,k_3} \delta(\Sigma)\delta(\Omega) \Big[ 2(1-\chi_{2<3}\chi_{1<0}) g_1(t) g_2(t) - g_2(t) g_3(t) \Big] \,dk_1 \,dk_2 \,dk_3,
\end{equation}
and the gain operator is
\begin{equation}
    \mathcal{Q}_{gain}(h)(k) = 2\iiint_{\mathbb{R}^{3d}} T_{k,k_1,k_2,k_3} \delta(\Sigma)\delta(\Omega)\chi_{2<3}\chi_{1\geq 0} g_1(t) g_2(t) h_3  \,dk_1 \,dk_2 \,dk_3.
\end{equation}

\noindent\textbf{Step 1: The $L^\infty_s$ estimate.} Assume $h \in L_s^\infty(\mathbb{R}^d)$. For the loss term, we have $\|\nu h\|_{\infty,s} \leq \|\nu\|_{\infty} \|h\|_{\infty,s}$. Let
\begin{equation}
    \tilde \nu(k):=\iiint_{\mathbb{R}^{3d}} T_{k,k_1,k_2,k_3} \delta(\Sigma)\delta(\Omega) \Big[ 2(1-\chi_{2<3}\chi_{1<0}) g_1(t) g_2(t) + g_2(t) g_3(t) \Big] \,dk_1 \,dk_2 \,dk_3.
\end{equation}
To bound $\nu(k)$, we invoke Proposition~\ref{prop: kernel} and bound $\tilde \nu(k)$. On the support of $\delta(\Sigma)\delta(\Omega)$, we have the algebraic bounds
\begin{equation}
\begin{aligned}
  &\langle k_1\rangle^{d+2} \langle k_2\rangle^{d+2} \left(\min_{0\leq j<l\leq 3}|k_j|^2+|k_l|^2\right)^2\left(\sum_{j=0}^3|k_j|^2\right)(1-\chi_{2<3}\chi_{1<0}) \\
  &\leq 100 (\langle k_1\rangle^{10+2d}+\langle k_2\rangle^{10+2d}) (1-\chi_{2<3}\chi_{1<0})
\end{aligned}
\end{equation}
and 
\begin{equation}
    \langle k_1\rangle^{d+2} \langle k_2\rangle^{d+2} \left(\min_{0\leq j<l\leq 3}|k_j|^2+|k_l|^2\right)^2\left(\sum_{j=0}^3|k_j|^2\right) \leq 100(\langle k_2\rangle^{10+2d}+\langle k_3\rangle^{10+2d}).
\end{equation}
Applying these bounds to the kernel yields
\begin{equation}
\begin{aligned}
    |\tilde\nu(k)| &\leq C_Q \iiint_{\mathbb{R}^{3d}} \left(\min_{0\leq j<l\leq 3}|k_j|^2+|k_l|^2\right)^2\left(\sum_{j=0}^3|k_j|^2\right)\delta(\Sigma)\delta(\Omega) \\
    &\qquad\qquad\qquad\times \Big[(1-\chi_{2<3}\chi_{1<0})g_1(t)g_2(t)+g_2(t)g_3(t)\Big]\,dk_1\,dk_2\,dk_3 \\
    &\leq 200C_Q \iiint_{\mathbb{R}^{3d}} \delta(\Sigma)\delta(\Omega)\langle k_1\rangle^{-d-2}\langle k_2\rangle^{-d-2}\Big[F(k_1,k_2,t)+F(k_2,k_3,t)\Big]\,dk_1\,dk_2\,dk_3,
\end{aligned}
\end{equation}
where $F(k_1,k_2,t)$ is defined as
\begin{equation}
    F(k_1,k_2,t) \coloneqq (\langle k_1\rangle^{10+2d}+\langle k_2\rangle^{10+2d})g_1(t)g_2(t).
\end{equation}
This, together with Lemma~\ref{lem: Rep1}, yields 
\begin{equation}\label{eq: nu k}
    \|\nu\|_{\infty} \leq  \|\tilde\nu\|_{\infty} \leq 800C_QC_d \|g(t)\|_{\infty,{10+2d}}^2.
\end{equation}

For the gain term, using that on the support of $\delta(\Sigma)\delta(\Omega)\chi_{2<3}\chi_{1\geq 0}$,
\begin{equation}
    \langle k\rangle^s\leq \langle k_3\rangle^s,
\end{equation}
we similarly extract $\|h\|_{\infty,s}$:
\begin{equation}
\begin{aligned}
    |\langle k\rangle^s\mathcal{Q}_{gain}(h)(k)| \leq &2\iiint_{\mathbb{R}^{3d}} T_{k,k_1,k_2,k_3} \delta(\Sigma)\delta(\Omega) \chi_{2<3}\chi_{1\geq 0} g_1(t) g_2(t) \langle k_3\rangle^s|h_3| \,dk_1 \,dk_2 \,dk_3\\
    \leq & 2\|h\|_{\infty,s} \iiint_{\mathbb{R}^{3d}} T_{k,k_1,k_2,k_3} \delta(\Sigma)\delta(\Omega) \chi_{2<3}\chi_{1\geq 0} g_1(t) g_2(t) \,dk_1 \,dk_2 \,dk_3.
    \end{aligned}
\end{equation}
Using the pointwise bound
\begin{equation}
   2 \chi_{2<3}\chi_{1\geq 0} g_1(t) g_2(t)\leq 2(1-\chi_{2<3}\chi_{1<0})g_1(t)g_2(t)+g_2(t)g_3(t)
\end{equation}
and the estimate~\eqref{eq: nu k} yields 
\begin{equation}
      \|\mathcal{Q}_{gain}(h)\|_{\infty,s} \leq \|\tilde\nu\|_{\infty} \|h\|_{\infty,s} \leq 800C_QC_d \|g(t)\|_{\infty,{10+2d}}^2\|h\|_{\infty,s}.
\end{equation}
Combining these estimates provides the $L^\infty_s$ bound:
\begin{equation}
    \|\Qc_{g,b}(t)\|_{\infty,s\to \infty,s} \leq 1600C_QC_d \|g(t)\|_{\infty,{10+2d}}^2.
\end{equation}

\noindent\textbf{Step 2: The $L^1_s$ estimate.} Assume $h \in L^1_s(\mathbb{R}^d)$. The $L^1_s$ norm of the loss term is bounded since $\|\nu h\|_{1,s} \leq \|\nu\|_{\infty} \|h\|_{1,s} \leq 800C_QC_d \|g(t)\|_{\infty,{10+2d}}^2 \|h\|_{1,s}$. For the gain term, similarly we have:
\begin{equation}
    \|\mathcal{Q}_{gain}(h)\|_{1,s} \leq 2 \iiiint_{\mathbb{R}^{4d}} T_{k,k_1,k_2,k_3} \delta(\Sigma)\delta(\Omega) \chi_{2<3}\chi_{1\geq 0} g_1(t) g_2(t)\langle k_3\rangle^s |h(k_3)| \,dk_1 \,dk_2 \,dk_3 \,dk.
\end{equation}
By Tonelli's Theorem, we interchange the order of integration. Applying the variable swap $(k,k_1) \leftrightarrow (k_3,k_2)$ and using the symmetry of the kernel $T_{k,k_1,k_2,k_3}$ yields:
\begin{equation}
\begin{aligned}
    \|\mathcal{Q}_{gain}(h)\|_{1,s} &\leq 2 \iiiint_{\mathbb{R}^{4d}} T_{k,k_1,k_2,k_3} \delta(\Sigma)\delta(\Omega) \chi_{1<0}\chi_{2\geq 3} g_1(t) g_2(t) \langle k\rangle^s|h(k)| \,dk_1 \,dk_2 \,dk_3 \,dk \\
    &\leq \|h\|_{1,s} \left\| 2 \iiint_{\mathbb{R}^{3d}} T_{k,k_1,k_2,k_3} \delta(\Sigma)\delta(\Omega) \chi_{1<0}\chi_{2\geq 3} g_1(t) g_2(t) \,dk_1 \,dk_2 \,dk_3 \right\|_{\infty}.
\end{aligned}
\end{equation}
Using the pointwise bound
\begin{equation}
   2 \chi_{1<0}\chi_{2\geq 3} g_1(t) g_2(t) \leq 2(1-\chi_{2<3}\chi_{1<0})g_1(t)g_2(t)+g_2(t)g_3(t),
\end{equation}
together with estimate~\eqref{eq: nu k}, yields 
\begin{equation}
     \|\mathcal{Q}_{gain}(h)\|_{1,s} \leq \|\tilde\nu\|_{\infty} \|h\|_{1,s} \leq 800C_QC_d \|g(t)\|_{\infty,{10+2d}}^2\|h\|_{1,s}.
\end{equation}
Combining these estimates provides the $L^1_s$ bound:
\begin{equation}
    \|\Qc_{g,b}(t)\|_{1,s\to 1,s} \leq 1600C_QC_d \|g(t)\|_{\infty,{10+2d}}^2.
\end{equation}

By the Riesz-Thorin interpolation theorem, we obtain  
\begin{equation}
    \|\Qc_{g,b}(t)\|_{p,s\to p,s} \leq 1600C_QC_d \|g(t)\|_{\infty,{10+2d}}^2 \qquad \forall\, 1\leq p\leq \infty,
\end{equation}
which completes the proof by setting $C_b = 1600C_QC_d$. 
\end{proof}

\begin{lem}\label{lem: we Dyn}
Equation~\eqref{eq: DE afn+1} is valid in the weak sense. 
\end{lem}
\begin{proof} 
Take $\rho\in C_0^\infty$. We compute 
\begin{equation}\label{eq: weak g fn+1}
    \begin{aligned}
        (\rho,\partial_t\langle k\rangle^af_{n+1}(t))_{L^2} &= (\rho,\langle k\rangle^a\Qc_n(t)f_{n+1}(t))_{L^2} \\
        &= (\rho,\Qc_n(t)\langle k\rangle^af_{n+1}(t))_{L^2} + (\rho,[\langle k\rangle^a,\Qc_n(t)]\langle k\rangle^{-a}\langle k\rangle^af_{n+1}(t))_{L^2}.
    \end{aligned}
\end{equation}
With $h = \langle k\rangle^{a}f_{n+1}(t)$, the second term on the r.h.s.\ of Eq.~\eqref{eq: weak g fn+1} reads
\begin{equation}
    \begin{aligned}
        &(\rho,[\langle k\rangle^a,\Qc_n(t)]\langle k\rangle^{-a}h)_{L^2} \\
        &= 2\iiiint_{\mathbb{R}^{4d}}T_{k,k_1,k_2,k_3}\delta(\Sigma)\delta(\Omega)\chi_{2<3}g_1(t)g_2(t) (\langle k\rangle^a-\langle k_3\rangle^a)\langle k_3\rangle^{-a}\rho h_3 \,dk_1\,dk_2\, dk_3\,dk.
    \end{aligned}
\end{equation}
This, together with the algebraic relation
\begin{equation}
   \frac{\langle k\rangle^a-\langle k_3\rangle^a}{\langle k_3\rangle^a} = \frac{\langle k\rangle^a-\langle k_3\rangle^a}{\langle k\rangle^{a/2}\langle k_3\rangle^{a/2}} + \frac{(\langle k\rangle^{a/2}-\langle k_3\rangle^{a/2})^2(\langle k\rangle^{a/2}+\langle k_3\rangle^{a/2})}{\langle k \rangle^{a/2}\langle k_3 \rangle^{a}}
\end{equation}
and Eqs.~\eqref{tQfn} and~\eqref{eq: def Rnt}, yields 
\begin{equation}
  (\rho,[\langle k\rangle^a,\Qc_n(t)]\langle k\rangle^{-a}h)_{L^2} = (\rho,(\tilde{\Qc}_n(t)+\Rc_n(t))h)_{L^2},
\end{equation}
which completes the proof.
\end{proof}

\begin{lem}\label{lem: tQn diss} 
$\tilde{\Qc}_{n}(t)$ is dissipative on $L^2$.
\end{lem}
\begin{proof}
    For each $h\in L^2$, exchanging the positions of $(k_3,k_2)$ and $(k,k_1)$ and then taking the average yields 
\begin{equation}
\begin{aligned}
    (h, \tilde{\Qc}_{n}(t)h)_{L^2} &= 2\iiiint_{\mathbb{R}^{4d}}T_{k,k_1,k_2,k_3}\delta(\Sigma)\delta(\Omega)\chi_{2<3}\chi_{1<0}g_1(t)g_2(t)\frac{\langle k\rangle^a-\langle k_3\rangle^a}{\langle k\rangle^{a/2}\langle k_3\rangle^{a/2}} hh_3 \,dk_1\,dk_2\, dk_3 \, dk \\
    &= 2\iiiint_{\mathbb{R}^{4d}}T_{k,k_1,k_2,k_3}\delta(\Sigma)\delta(\Omega)\chi_{2<3}\chi_{1<0}g_1(t)g_2(t) \frac{\langle k_3\rangle^a-\langle k\rangle^a}{\langle k\rangle^{a/2}\langle k_3\rangle^{a/2}} hh_3 \,dk_1\,dk_2\, dk_3 \, dk \\
    &= 0.
    \end{aligned}
\end{equation}
This yields 
\begin{equation}
    \begin{aligned}
        \| (\lambda-\tilde{\Qc}_{n}(t))h\|_{2} &= \sup_{\| g\|_{{2}}=1}|(g, (\lambda-\tilde{\Qc}_{n}(t))h)_{L^2}| \\
        &\geq (h/\|h\|_{2}, (\lambda-\tilde{\Qc}_{n}(t))h)_{L^2} \\
        &= \lambda \|h\|_{2}.
    \end{aligned}
\end{equation}
This completes the proof.
\end{proof}

\begin{lem}\label{lem: eq: Rn bd} 
$\Rc_n(t)$ is bounded on $L^p$ for all $1\leq p\leq \infty$ satisfying the estimate
\begin{equation}\label{eq: Rn bd est}
    \| \Rc_n(t)\|_{p\to p} \leq \tilde C_a \| f_n(t)\|_{\infty,{10+2d}}^2 \qquad \forall\, 1\leq p\leq \infty
\end{equation}
for some constant $\tilde C_a>0$. 
\end{lem}
\begin{proof}
By the mean value theorem, we have on the support of $\delta(\Sigma)\delta(\Omega)$,
\begin{equation}
   \langle k\rangle\langle k_3\rangle \frac{(\langle k\rangle^{a/2}-\langle k_3\rangle^{a/2})^2(\langle k\rangle^{a/2}+\langle k_3\rangle^{a/2})}{\langle k \rangle^{a/2}\langle k_3 \rangle^{a/2}} \chi_{2<3}\chi_{1<0} \leq C_a \frac{(|k|-|k_3|)^2}{\langle k\rangle\langle k_3\rangle} \leq C_a \frac{|k_1-k_2|^2}{\langle k\rangle\langle k_3\rangle}
\end{equation}
for some constant $C_a>0$. We also note that on the support of $\delta(\Sigma)\delta(\Omega)$,
\begin{equation}
   \Big| \frac{\langle k\rangle^a-\langle k_3\rangle^a}{\langle k_3\rangle^{a}}\Big|\chi_{2<3}\chi_{1\geq 0} \leq 2. 
\end{equation}
These bounds, combined with Eq.~\eqref{eq: def nuk}, the kernel estimates~\eqref{eq: est kernel}, and the following bound on the support of $\delta(\Sigma)\delta(\Omega)$:
\begin{equation}
\begin{aligned}
   &\left(\min_{0\leq j<l\leq 3}|k_j|^2+|k_l|^2\right)^2\left(\sum_{j=0}^3|k_j|^2\right)\frac{|k_1-k_2|^2}{\langle k\rangle\langle k_3\rangle}\chi_{2<3}\chi_{1<0} \\
   &\leq 2 \left(|k_1|^2+|k_2|^2\right)^2\left(|k|^2+|k_3|^2\right)\frac{|k_1-k_2|^2}{\langle k\rangle\langle k_3\rangle}\chi_{2<3}\chi_{1<0} \\
   &\leq C_1\left(|k_1|^2+|k_2|^2\right)^3\chi_{2<3}\chi_{1<0} \qquad \qquad \text{ for some constant }C_1>0,
\end{aligned}
\end{equation}
yield for $h\in L^\infty$, with $g(t)=f_n(t)$:
\begin{equation}\label{eq: est Rnt}
    \begin{aligned}
        |\Rc_n(t)h| &\leq 4\|h\|_{\infty}\iiint_{\mathbb{R}^{3d}}T_{k,k_1,k_2,k_3}\delta(\Sigma)\delta(\Omega) \chi_{2<3}\chi_{1\geq 0}g_1(t)g_2(t)\,dk_1\,dk_2\, dk_3 \\
        &\quad + 2C_aC_Q\|h\|_{\infty}\iiint_{\mathbb{R}^{3d}}\left(\min_{0\leq j<l\leq 3}|k_j|^2+|k_l|^2\right)^2\left(\sum_{j=0}^3|k_j|^2\right)\frac{|k_1-k_2|^2}{\langle k\rangle\langle k_3\rangle} \\
        &\qquad\qquad\times\delta(\Sigma)\delta(\Omega)\chi_{2<3}\chi_{1<0}g_1(t)g_2(t) \,dk_1\,dk_2\, dk_3 \\
        &\leq 2\|h\|_{\infty}\|\tilde\nu(k)\|_{\infty} \\
        &\quad + 2C_aC_1C_Q\|h\|_{\infty}\iiint_{\mathbb{R}^{3d}}(|k_1|^2+|k_2|^2)^3\delta(\Sigma)\delta(\Omega)\chi_{2<3}\chi_{1<0}g_1(t)g_2(t) \,dk_1\,dk_2\, dk_3.
    \end{aligned}
\end{equation}
This, together with estimate~\eqref{eq: nu k} and Lemma~\ref{lem: Rep1}, yields 
\begin{equation}
    \|\Rc_n(t)h\|_{\infty} \leq \left(1600C_QC_d \| g\|_{\infty,{10+2d}}^2 + 16C_aC_1C_QC_d \| g\|_{\infty,{8+d}}^2\right)\|h\|_{\infty}.
\end{equation}
Consequently, we have 
\begin{equation}
    \|\Rc_n(t)\|_{\infty\to \infty} \leq 1600C_QC_d \| g\|_{\infty,{10+2d}}^2 + 16C_aC_1C_QC_d \| g\|_{\infty,{8+d}}^2.
\end{equation}

Similarly, for $h\in L^1$, we interchange the order of integration to evaluate the $k_3$ integral last, swap the positions of $(k_2,k_3)$ and $(k_1,k)$, and use the symmetry of the kernel $T_{k,k_1,k_2,k_3}$:
\begin{equation}\label{eq: est Rnt L1}
    \begin{aligned}
        \|\Rc_n(t)h\|_{1} &\leq 4\|h\|_{1}\iiint_{\mathbb{R}^{3d}}T_{k,k_1,k_2,k_3}\delta(\Sigma)\delta(\Omega) \chi_{1<0}\chi_{2\geq 3}g_1(t)g_2(t)\,dk_1\,dk_2\, dk_3 \\
        &\quad + 2C_aC_Q\|h\|_{1}\iiint_{\mathbb{R}^{3d}}\left(\min_{0\leq j<l\leq 3}|k_j|^2+|k_l|^2\right)^2\left(\sum_{j=0}^3|k_j|^2\right)\frac{|k_1-k_2|^2}{\langle k\rangle\langle k_3\rangle} \\
        &\qquad\qquad\times\delta(\Sigma)\delta(\Omega)\chi_{2<3}\chi_{1<0}g_1(t)g_2(t) \,dk_1\,dk_2\, dk_3 \\
        &\leq 2\|h\|_{1}\|\tilde\nu(k)\|_{\infty} \\
        &\quad + 2C_aC_1C_Q\|h\|_{1}\iiint_{\mathbb{R}^{3d}}(|k_1|^2+|k_2|^2)^3\delta(\Sigma)\delta(\Omega)\chi_{2<3}\chi_{1<0}g_1(t)g_2(t) \,dk_1\,dk_2\, dk_3.
    \end{aligned}
\end{equation}
This, together with estimate~\eqref{eq: nu k} and Lemma~\ref{lem: Rep1}, yields 
\begin{equation}
    \|\Rc_n(t)h\|_{1} \leq \left(1600C_QC_d \| g\|_{\infty,{10+2d}}^2 + 16C_aC_1C_QC_d \| g\|_{\infty,{8+d}}^2\right)\|h\|_{1}.
\end{equation}
Consequently, we have 
\begin{equation}
    \|\Rc_n(t)\|_{1\to 1} \leq 1600C_QC_d \| g\|_{\infty,{10+2d}}^2 + 16C_aC_1C_QC_d \| g\|_{\infty,{8+d}}^2.
\end{equation}
By the interpolation inequality, we arrive at~\eqref{eq: Rn bd est}.
\end{proof}
Identifying the base functions as $g(t) := f_{n+1}(t)$ and $\tilde{g}(t) := f_n(t)$, we define the linear operator $\Fc_n(t): L^2(\mathbb{R}^d) \to L^2(\mathbb{R}^d)$. Its action on a generic function $h \in L^2(\mathbb{R}^d)$ is given by the collision integral:
\begin{equation}\label{eq: Fcnh}
\begin{aligned}
    \Fc_n(t)h := \iiint_{\mathbb{R}^{3d}} & T_{k,k_1,k_2,k_3} \, \delta(\Sigma)\delta(\Omega) \Bigg\{ \\
    & \quad 2\chi_{2<3}\langle k\rangle^{2(d+1)} \left( g_3(t)-g(t)+\frac{\langle k\rangle^{2(d+1)}-\langle k_3\rangle^{2(d+1)}}{\langle k_3\rangle^{2(d+1)}}g_3(t) \right) \\
    & \qquad \times \left( g_2(t)\langle k_1\rangle^{-2(d+1)} h_1 + \tilde{g}_1(t)\langle k_2\rangle^{-2(d+1)} h_2 \right) \\
    & \quad + \left( \langle k_2\rangle^{-2(d+1)} h_2 g_3(t) + \tilde{g}_2(t)\langle k_3\rangle^{-2(d+1)} h_3 \right) \langle k\rangle^{2(d+1)}g(t) \\
    & \quad - 2\chi_{2\geq 3} \left( \langle k_1\rangle^{-2(d+1)} h_1 g_2(t) + \tilde{g}_1(t)\langle k_2\rangle^{-2(d+1)} h_2 \right) \langle k\rangle^{2(d+1)}g(t)
    \Bigg\} \,dk_1\,dk_2\,dk_3.
\end{aligned}
\end{equation}

\begin{lem}\label{lem: est Fn}
There exists a constant $C_{F}=C_{F}(d,C_Q)>0$, depending only on the dimension $d$ and $C_Q$, such that the $L^2 \to L^2$ operator norm of $\Fc_n(t)$ satisfies
\begin{equation}\label{eq: est Fn}
    \|\Fc_n(t)\|_{2\to 2} \leq C_{F} \max\left\{ \sup_{t\in [0,T]} \| f_n(t)\|_{\infty,12+4d}^2, \sup_{t\in [0,T]} \| f_{n+1}(t)\|_{\infty,12+4d}^2 \right\}.
\end{equation}
\end{lem}

\begin{proof} 
We proceed by establishing pointwise bounds on the collision integrand. We note that by Proposition~\ref{prop: kernel}, on the support of the conservation distributions $\delta(\Sigma)\delta(\Omega)$, we have
\begin{equation} 
\begin{aligned}
    &T_{k,k_1,k_2,k_3}\chi_{2<3}\langle k\rangle^{2(d+1)}\frac{\langle k\rangle^{2(d+1)}}{\langle k_3\rangle^{2(d+1)}}\langle k_3\rangle^{-2}\langle k_2\rangle^{-4}\langle k_1\rangle^{-3-d/2}\\
    &\quad \leq \tilde{C}_{d1}C_Q\left(|k_2|^2+\min\{|k_1|^2, |k|^2\}\right)^2\langle k\rangle^{-(d+1)}\langle k_3\rangle^{3(d+1)}\langle k_2\rangle^{-4}\langle k_1\rangle^{-3-d/2}\\
    &\quad \leq 4\tilde C_{d1}C_Q\langle k\rangle^{-(d+1)}\langle k_3\rangle^{3(d+1)}
\end{aligned}
\end{equation}
and
\begin{equation} 
\begin{aligned}
    &T_{k,k_1,k_2,k_3}\chi_{2<3}\langle k\rangle^{2(d+1)}\langle k_3\rangle^{d/2+1}\langle k_2\rangle^{-4}\langle k_1\rangle^{-3-d/2}\\
    &\quad \leq \tilde C_{d2}C_Q\left(|k_2|^2+\min\{|k_1|^2, |k|^2\}\right)^2\langle k\rangle^{2(d+1)}\langle k_3\rangle^{d/2+3}\langle k_2\rangle^{-4}\langle k_1\rangle^{-3-d/2}\\
    &\quad \leq 4\tilde C_{d2}C_Q\langle k\rangle^{d/2+3+2(d+1)},
\end{aligned}
\end{equation}
for some constants $\tilde C_{d1}, \tilde C_{d2}>0$. These bounds imply that, on the support of the conservation distributions $\delta(\Sigma)\delta(\Omega)$, we obtain the first composite estimate:
\begin{equation}\label{eq: est_piece1}
\begin{split}
    & 2T_{k,k_1,k_2,k_3} \chi_{2<3}\langle k\rangle^{2(d+1)} \left|g_3(t)-g(t)+\frac{\langle k\rangle^{2(d+1)}-\langle k_3\rangle^{2(d+1)}}{\langle k_3\rangle^{2(d+1)}}g_3(t)\right| \\
    &\quad \times \left|g_2(t)\langle k_1\rangle^{-2(d+1)} h_1+\tilde g_1(t)\langle k_2\rangle^{-2(d+1)} h_2\right|\\
    &= 2T_{k,k_1,k_2,k_3} \chi_{2<3}\langle k\rangle^{2(d+1)} \left|\frac{\langle k\rangle^{2(d+1)}}{\langle k_3\rangle^{2(d+1)}}g_3(t)-g(t)\right| \\
    &\quad \times \left|g_2(t)\langle k_1\rangle^{-2(d+1)} h_1+\tilde g_1(t)\langle k_2\rangle^{-2(d+1)} h_2\right|\\
    &\leq C_1C_Q \left(\langle k\rangle^{-(d+1)}\langle k_3\rangle^{2+3(d+1)}g_3(t)+\langle k_3\rangle^{-d/2-1}\langle k\rangle^{d/2+3+2(d+1)}g(t)\right)\\
    &\quad \times \left(\langle k_2\rangle^4g_2(t)\langle k_1\rangle^{3+d/2-2(d+1)}|h_1|+\langle k_1\rangle^{3+d/2} \tilde g_1(t)\langle k_2\rangle^{4-2(d+1)}|h_2| \right),
\end{split}
\end{equation}
with $C_1 = C_1(d) = 16(\tilde C_{d1}+\tilde C_{d2})>0$. 

Similarly, utilizing the inequalities
\begin{equation}
\begin{aligned}
    & T_{k,k_1,k_2,k_3}\langle k_2\rangle^{-2-d/2-1}\langle k_3\rangle^{-2-d/2-1}\langle k\rangle^{-4}\langle k_1\rangle^{d/2+1}\\
    &\quad \leq 2C_Q \left( \min\{|k_2|,|k_3|\}+\min\{|k_1|,|k|\}\right)^2\left( \max\{|k_2|,|k_3|\}+\max\{|k_1|,|k|\}\right)\\
    &\quad \quad \times \langle k_2\rangle^{-2-d/2-1}\langle k_3\rangle^{-2-d/2-1}\langle k\rangle^{-4}\langle k_1\rangle^{d/2+1}\\
    &\quad \leq C_2C_Q
\end{aligned}
\end{equation}
and
\begin{equation}
\begin{aligned}
    &2T_{k,k_1,k_2,k_3}\langle k_2\rangle^{-4}\langle k_1\rangle^{-2-d/2-1}\langle k\rangle^{-2-d/2-1}\langle k_3\rangle^{d/2+1}\\
    &\quad \leq 4C_Q \left( \min\{|k_2|,|k_3|\}+\min\{|k_1|,|k|\}\right)^2\left( \max\{|k_2|,|k_3|\}+\max\{|k_1|,|k|\}\right)\\
    &\quad \quad \times \langle k_2\rangle^{-4}\langle k_1\rangle^{-2-d/2-1}\langle k\rangle^{-2-d/2-1}\langle k_3\rangle^{d/2+1}\\
    &\quad \leq C_3C_Q
\end{aligned}
\end{equation}
for some constants $C_2=C_2(d), C_3=C_3(d)>0$, the remaining terms are bounded as:
\begin{equation}\label{eq: est_piece2}
\begin{split}
    & T_{k,k_1,k_2,k_3}\left| \langle k_2\rangle^{-2(d+1)}h_2g_3(t)+\tilde g_2(t)\langle k_3\rangle^{-2(d+1)}h_3\right|\langle k\rangle^{2(d+1)}g(t)\\
    &\leq C_2C_Q \left(\langle k_2\rangle^{3+d/2-2(d+1)}|h_2|\langle k_3\rangle^{3+d/2} g_3(t)+\langle k_2\rangle^{3+d/2} \tilde g_2(t)\langle k_3\rangle^{3+d/2-2(d+1)}|h_3|\right)\\
    &\quad \times \langle k_1\rangle^{-(1+d/2)}\langle k\rangle^{2(d+1)+4}g(t),
\end{split}
\end{equation}
and 
\begin{equation}\label{eq: est_piece3}
\begin{split}
    & 2T_{k,k_1,k_2,k_3}\chi_{2\geq 3}\left|\langle k_1\rangle^{-2(d+1)} h_1 g_2(t)+\tilde g_1(t)\langle k_2\rangle^{-2(d+1)} h_2\right|\langle k\rangle^{2(d+1)}g(t)\\
    &\leq C_3C_Q\left(\langle k_1\rangle^{3+d/2-2(d+1)}|h_1|\langle k_2\rangle^4 g_2(t)+\langle k_1\rangle^{3+d/2}\tilde g_1(t)\langle k_2\rangle^{4-2(d+1)}|h_2| \right)\\
    &\quad \times \langle k_3\rangle^{-(1+d/2)}\langle k\rangle^{2(d+1)+3+d/2}g(t).
\end{split}
\end{equation}
By virtue of \eqref{eq: Qf03}, \eqref{eq: Qf02}, and \eqref{eq: Qf01}, substituting these pointwise bounds yields the composite estimate:
\begin{equation}\label{eq: Fn_pointwise_bound}
\begin{split}
    |\Fc_n(t)h| &\leq C_{d1} \int_{\mathbb R^d} \frac{ \left(\|g(t)\|_{\infty,6+7d/2}\|g(t)\|_{\infty,5+d/2}\langle k\rangle^{-d-1}+\|g(t)\|_{\infty,5+d/2}\langle k\rangle^{5+5d/2}g(t)\right)\langle k_1\rangle^{3+d/2-2(d+1)}|h_1| }{|k+k_1|} \,dk_1 \\
    &\quad + C_{d1} \int_{\mathbb R^d} \frac{ \left(\|g(t)\|_{\infty,6+7d/2}\|\tilde g(t)\|_{\infty,4+d}\langle k\rangle^{-d-1}+\|\tilde g(t)\|_{\infty,4+d}\langle k\rangle^{5+5d/2}g(t)\right)\langle k_2\rangle^{4-2(d+1)}|h_2| }{|k-k_2|} \,dk_2 \\
    &\quad + C_{d2} \int_{\mathbb R^d} \frac{ \|g(t)\|_{\infty,4+d}\langle k\rangle^{2d+6}g(t)\langle k_2\rangle^{3+d/2-2(d+1)}|h_2| }{|k-k_2|} \,dk_2 \\
    &\quad + C_{d2} \int_{\mathbb R^d} \frac{ \|\tilde g(t)\|_{\infty,4+d}\langle k\rangle^{6+2d}g(t)\langle k_3\rangle^{3+d/2-2(d+1)}|h_3| }{|k-k_3|} \,dk_3 \\
    &\quad + C_{d3} \int_{\mathbb R^d} \frac{ \|g(t)\|_{\infty,5+d/2}\langle k\rangle^{5d/2+5}g(t)\langle k_1\rangle^{3+d/2-2(d+1)}|h_1| }{|k+k_1|} \,dk_1 \\
    &\quad + C_{d3} \int_{\mathbb R^d} \frac{ \|\tilde g(t)\|_{\infty,4+d}\langle k\rangle^{5+5d/2}g(t)\langle k_2\rangle^{4-2(d+1)}|h_2| }{|k-k_2|} \,dk_2,
\end{split}
\end{equation}
for some constants $C_{dj} = C_d C_j C_Q > 0$, where $j=1,2,3$. 

Finally, we apply the weighted singular integral estimate for $d \geq 2$:
\begin{equation}
  \left\| \int_{\mathbb R^d}\frac{\langle k\rangle^{-(d/2+1)}\langle y\rangle^{4-2(d+1)}|h(y)|}{|k\pm y|} \,dy \right\|_2 \leq  \left\| \int_{\mathbb R^d}\frac{\langle k\rangle^{-(d/2+1)}\langle y\rangle^{3+d/2-2(d+1)}|h(y)|}{|k\pm y|} \,dy \right\|_2 \leq C_{d4} \|h\|_2,
\end{equation}
for some uniform constant $C_{d4}>0$. Taking the $L^2$-norm of \eqref{eq: Fn_pointwise_bound} and directly applying this singular integral bound yields the desired estimate \eqref{eq: est Fn}, with the overarching constant
\begin{equation}
    C_{F,d} \coloneqq 2C_{d4}\left(\sum\limits_{j=1}^3 C_{dj}\right).
\end{equation}
\end{proof}

\begin{lem}\label{lem: tUn h} 
For all $h\in L^2\cap L^\infty$ and $0\leq s<t\leq T$, we have
\begin{equation}\label{eq: est tUnh}
    \|\tilde U_{n}(t,s)h\|_{\infty} \leq \exp\left((t-s)(C_b+\tilde C_a)\sup_{t\in [0,T]} \| f_n(t)\|_{\infty,10+2d}^2\right)\|h\|_{\infty}.
\end{equation}  
\end{lem}
\begin{proof} 
It suffices to prove the $h\geq 0$ case. Let $u(t) = \tilde U_n(t,s)h$ be the solution to the evolution equation $\partial_t u(t) = \Qc_n(t)u(t) + \Rc_n(t)u(t)$ with the initial condition $u(s) = h$. Since $h \geq 0$ and $\tilde U(t,s)$ is positivity-preserving, we have $u(t) \geq 0$ for all $t \in [s, T]$.

We first establish the energy estimate in $L^p$ for any finite $p \in [2, \infty)$. Taking the time derivative of the $L^p$ norm and using the standard duality mapping (noting that $u \geq 0$, so $|u|^{p-2}u = u^{p-1}$), we obtain:
\begin{equation}
    \frac{1}{p} \frac{d}{dt} \|u(t)\|_p^p = ( \partial_t u, u^{p-1} )_{L^2} = ( \Qc_n(t)u, u^{p-1} )_{L^2} + (\Rc_n(t)u, u^{p-1})_{L^2}.
\end{equation}
By the condition that $\Qc_n(t)$ is dissipative on $L^p$, its duality pairing is non-positive, i.e., $$( \Qc_n(t)u, u^{p-1} )_{L^2} \leq 0.$$ 
Thus, applying H\"older's inequality to the remaining term yields:
\begin{equation}
    \frac{1}{p} \frac{d}{dt} \|u(t)\|_p^p \leq ( \Rc_n(t)u, u^{p-1} )_{L^2} \leq \|\Rc_n(t)u\|_p \|u^{p-1}\|_{p'} = \|\Rc_n(t)u\|_p \|u\|_p^{p-1},
\end{equation}
where $1/p + 1/p' = 1$. This simplifies to the differential inequality:
\begin{equation}
    \frac{d}{dt} \|u(t)\|_p \leq \|\Rc_n(t)u(t)\|_p \leq \|\Rc_n(t)\|_{p \to p} \|u(t)\|_p.
\end{equation}
Using Gr\"onwall's inequality, we obtain the $L^p$ bound:
\begin{equation}
    \|u(t)\|_p \leq \|h\|_p \exp\left( \int_s^t \|\Rc_n(\tau)\|_{p \to p} \, d\tau \right) \leq \|h\|_p \exp\left( (t-s)\sup_{\tau \in [0,T]} \|\Rc_n(\tau)\|_{p \to p} \right).
\end{equation}
This, together with Lemma~\ref{lem: eq: Rn bd} and $\|u\|_{\infty}=\lim_{p\to \infty}\|u\|_p$ for all $u\in L^2\cap L^\infty$, yields~\eqref{eq: est tUnh}. This completes the proof.
\end{proof}

Let $U_{n,2}(t,s)$, for $0<s<t<T$, denote the propagator generated by $\Qc_{n+1}(t)+\tilde{\Qc}_{n+1}(t)+\Rc_{n+1}(t)$. Using an argument similar to that in Lemma~\ref{lem: tUn h}, we have 
\begin{cor}\label{cor: est tUn2h}
\begin{equation}\label{eq: est tUn2h}
\|U_{n,2}(t,s)h\|_{2}
\leq
\exp\left((t-s)\tilde C_a\sup_{t\in [0,T]} \| f_n(t)\|_{\infty,10+2d}^2\right)\|h\|_{2},
\qquad \forall\, h\in L^2.
\end{equation}
\end{cor}

\begin{lem}\label{lem: infty bd} 
For all $h\in L^2_{22+5d}$, 
\begin{equation}
    \|\tilde{\Qc}_{n}(t)\langle k\rangle^{12+4d}h\|_{\infty} \leq \tilde C_b \left(\sup_{t\in (0,T]}\|f_n(t)\|_{\infty,10+2d}^2\right)\|h\|_{2,22+5d}
\end{equation}
holds true for some constant $\tilde C_b = \tilde C_b(C_Q,d)>0$.
\end{lem}
\begin{proof} 
 On the support of $\delta(\Sigma)\delta(\Omega)$, we have
\begin{equation}
\chi_{2<3}\chi_{1<0}\frac{\big| \langle k\rangle^a-\langle k_3\rangle^a \big|}{\langle k\rangle^{a/2}\langle k_3\rangle^{a/2}} \leq C_2\frac{\big| |k|-|k_3| \big|}{\langle k\rangle} \leq C_2\frac{|k_1-k_2|}{\langle k\rangle} \qquad \text{ for some }C_2>0.
\end{equation}
Together with the estimate~\eqref{eq: est kernel}, this yields, on the support of $\delta(\Sigma)\delta(\Omega)$, 
\begin{equation}
    T_{k,k_1,k_2,k_3} \chi_{2<3}\chi_{1<0}\frac{\big| \langle k\rangle^a-\langle k_3\rangle^a \big|}{\langle k\rangle^{a/2}\langle k_3\rangle^{a/2}} \leq C_3C_Q (|k_1|^2+|k_2|^2)^2|k_1-k_2||k_3| \qquad \text{ for some }C_3>0,
\end{equation}
and consequently, with $g(t)=f_n(t)$, 
\begin{equation}
    \begin{aligned}
        |\tilde{\Qc}_n(t)\langle k\rangle^{12+4d}h| &\leq 2C_3C_Q \iiint_{\mathbb{R}^{3d}}(|k_1|^2+|k_2|^2)^2|k_1-k_2||k_3| \\
        &\qquad\qquad\times\delta(\Sigma)\delta(\Omega)g_1(t)g_2(t)\langle k_3\rangle^{12+4d}|h_3| \,dk_1\,dk_2\,dk_3\\
        &= 2C_3C_Q \iiint_{\mathbb{R}^{3d}}|k-k_3|(|k_1|^2+|k_2|^2)^2|k_3| \\
        &\qquad\qquad\times\delta(\Sigma)\delta(\Omega)g_1(t)g_2(t)\langle k_3\rangle^{12+4d}|h_3| \,dk_1\,dk_2\,dk_3.
    \end{aligned}
\end{equation}
Integrating out the variable $k_1$ by enforcing the momentum conservation constraint imposed by the Dirac delta function $\delta(\Sigma)$ and using~\eqref{eq: Qf03}, yields
\begin{equation}
    \begin{aligned}
        |\tilde{\Qc}_n(t)\langle k\rangle^{12+4d}h| &\leq 2C_3C_Q C_d\int_{\mathbb{R}^{d}}\|(|k_1|^2+|k_2|^2)^2\langle k_1\rangle^{d/2+1}\langle k_2\rangle^{d/2+1} g_1(t) g_2(t)\|_{L^\infty_{k_1,k_2}(\mathbb R^{2d})}\\
        &\times\langle k_3\rangle^{13+4d}|h_3|\,dk_3 \\
        &\leq 8C_3C_Q C_d\|f_n(t)\|_{\infty,5+d/2}^2\int_{\mathbb{R}^d}\langle k_3\rangle^{13+4d}|h_3|dk_3\\
        &\leq  8C_3C_Q  C_d\tilde C_d'\|f_n(t)\|_{\infty,5+d/2}^2\|h\|_{2,22+5d}
    \end{aligned}
\end{equation}
for some constants $\tilde C_d, \tilde C_{d}'>0$. This completes the proof.\end{proof}

\begin{proof}[Proof of Lemma~\ref{lem: Rep1}] 
First, we prove \eqref{eq: Qf03}. We estimate 
\begin{equation}
    |\tilde{\Qc}_{F1}(k)| \leq \tilde{\Qc}_{F1,<}(k) + \tilde{\Qc}_{F1,>}(k),
\end{equation}
where 
\begin{equation}
    \tilde{\Qc}_{F1,<}(k) \coloneqq \iiint_{\mathbb{R}^{3d}} \delta(\Sigma)\delta(\Omega)\chi_{2<1}|F(k,k_1,k_2,k_3)| \,dk_1 \,dk_2 \,dk_3
\end{equation}
and
\begin{equation}
    \tilde{\Qc}_{F1,>}(k) \coloneqq \iiint_{\mathbb{R}^{3d}} \delta(\Sigma)\delta(\Omega)\chi_{2>1}|F(k,k_1,k_2,k_3)| \,dk_1 \,dk_2 \,dk_3.
\end{equation}
Since the diagonal set $\{ (k_1, k_2, k_3) : k = k_3 \}$ constitutes a lower-dimensional manifold (corresponding to the trivial forward scattering configuration) with Lebesgue measure zero, its contribution to the integral vanishes identically. Thus, we have
\begin{equation}
    \iiint_{\mathbb{R}^{3d}} \delta(\Sigma)\delta(\Omega)\chi(|k-k_3|=0)\chi_{1<2}F(k,k_1,k_2,k_3) \,dk_1 \,dk_2 \,dk_3 = 0,
\end{equation}
which allows us to insert the indicator function $\chi(|k-k_3|>0)$ and rewrite $\tilde{\Qc}_{F1,<}(k)$ as:
\begin{equation}
    \tilde{\Qc}_{F1,<}(k) = \iiint_{\mathbb{R}^{3d}} \delta(\Sigma)\delta(\Omega)\chi(|k-k_3|>0)\chi_{2<1}F(k,k_1,k_2,k_3) \,dk_1 \,dk_2 \,dk_3.
\end{equation}
We first integrate out the variable $k_1$ by enforcing the momentum conservation constraint imposed by the Dirac delta distribution $\delta(\Sigma)$:
\begin{equation}
    \tilde{\Qc}_{F1,<}(k) = \iint_{\mathbb{R}^{2d}} \delta\big(\Delta \omega(k, k_2,k_3)\big) \chi(|k-k_3|>0)\chi_{2<1}F(k, k_2+k_3-k, k_2, k_3) \,dk_2\,dk_3,
\end{equation}
where the energy resonance mismatch function is defined as:
\begin{equation}\label{eq: A1}
    \Delta \omega(k, k_2, k_3) = \omega + \sqrt{|k_2+k_3-k|} - \omega_2 - \omega_3.
\end{equation}
To evaluate the integration with respect to $k_2$ in the presence of $\delta(\Delta \omega)$, we compute the gradient:
\begin{equation}
    \nabla_{k_2}[\Delta \omega(k,k_2,k_3)] = \left. \left( \frac{k_1}{2\omega_1|k_1|} - \frac{k_2}{2\omega_2|k_2|} \right) \right|_{k_1=k_2+k_3-k}.
\end{equation}
We note that by Lemma~\ref{lem:group_velocity_lower_bound}, the gradient admits a strictly positive lower bound:
\begin{equation}
    |\nabla_{k_2}[\Delta \omega(k,k_2,k_3)]| \geq \frac{|k_3-k|}{2\omega_1\omega_2(\omega_1+\omega_2)} > 0.
\end{equation}
To integrate out the energy Dirac delta function, we partition the domain based on the direction of the gradient vector $V(k_2) \equiv \nabla_{k_2}[\Delta \omega(k,k_2,k_3)]$. We cover the unit sphere $\mathcal{S}^{d-1}$ with a finite number of overlapping angular cones (sectors) $\{ \mathcal{C}_j \}_{j=1}^N$ for some positive integer $N=N(d)\in \mathbb{N}^+$. Within each cone $\mathcal{C}_j$, we can choose a fixed unit vector $e_j$ (the axis of the cone) such that for any $V(k_2)$ pointing into $\mathcal{C}_j$, the directional derivative satisfies the lower bound:
\begin{equation}\label{eq: Vej}
    |e_j\cdot V(k_2)| \geq \frac{1}{2\sqrt{d}} |V(k_2)|.
\end{equation}
Let $\{ \chi_j \}_{j=1}^N$ be a smooth partition of unity subordinate to this angular covering, such that 
\begin{equation}
\sum_{j=1}^N \chi_j(V(k_2)) = 1\qquad \forall \,k_2. 
\end{equation}
The integral can then be decomposed as:
\begin{equation}\label{eq: sim1}
\begin{split}
    \tilde{\Qc}_{F1,<}(k) &= \int_{\mathbb{R}^{2d}} \delta\big(\Delta \omega(k, k_2, k_3)\big) \chi(|k-k_3|>0)\chi_{2<1}F(k, k_2+k_3-k, k_2, k_3) \,dk_2\,dk_3 \\
    &= \sum_{j=1}^N \int_{\mathbb{R}^{2d}} \delta\big(\Delta \omega(k, k_2, k_3)\big) \chi(|k-k_3|>0)\chi_{2<1}F(k, k_2+k_3-k, k_2, k_3)\chi_j \,dk_2\,dk_3.
\end{split}
\end{equation}
In each localized integral, we adopt a Cartesian coordinate system aligned with the cone axis: $k_2 = s e_j + \vt$, where $s \in \mathbb{R}$ is the coordinate along $e_j$ and $\vt \in \mathbb{R}^{d-1}$ is the transverse vector. Applying the one-dimensional property of the Dirac delta distribution along the $s$-direction, together with \eqref{eq: Vej} and Lemma~\ref{lem:group_velocity_lower_bound}, we obtain:
\begin{equation}
\begin{split}
    |\tilde{\Qc}_{F1,<}(k)| &\leq 4\sqrt{d}\sum_{j=1}^N \int_{\mathbb R^d}\int_{\mathbb{R}^{d-1}}\frac{1}{\langle \vt\rangle^{d}} \\
    &\quad \times \sum_{s \in \mathcal{Z}_{\vt}^{(3)}} \frac{\|\langle k_2\rangle^{d}\omega_1\omega_2(\omega_1+\omega_2)\chi_{2<1}F(k,k_1,k_2,k_3)\chi_j|_{k_2=se_j+\vt}\|_{L^\infty_{k_1}(\mathbb R^d)}}{|k-k_3|} \,d\vt\, dk_3,
\end{split}
\end{equation}
where $\mathcal{Z}_{\vt}^{(3)} = \{ s : \Delta \omega(k, k_2, k_3)|_{k_2=se_{j}+\vt} = 0 \}$ represents the roots of the resonance condition for a fixed transverse vector $\vt \in \mathbb R^{d-1}$. Crucially, the cardinality of the root set $|\mathcal{Z}_{\vt}^{(3)}|$ is uniformly bounded by a finite universal constant (i.e., $\mathcal{O}(1)$). Globally, the dispersion relation ensures the resonance condition is equivalent to a polynomial equation of finite degree. Locally, the lower bound \eqref{eq: Vej} guarantees that $\Delta \omega$ is strictly monotonic along the $s$-direction within each open cone $\mathcal{C}_j$. Consequently, within any connected component, the root $s$, if it exists, is strictly unique, rendering the summation over $\mathcal{Z}_{\vt}^{(3)}$ finite and well-defined. For a rigorous derivation establishing the uniform cardinality bound $|\mathcal{Z}_{\vt}^{(3)}|\leq 8$, we refer the reader to~Lemma~\ref{lem: Zvt}. 

Lemma~\ref{lem: Zvt}, together with the estimate
\begin{equation}
\begin{split}
    &\|\langle k_2\rangle^{d}\omega_1\omega_2(\omega_1+\omega_2)\chi_{2<1}F(k,k_1,k_2,k_3)\chi_j|_{k_2=se_j+\vt}\|_{L^\infty_{k_1}(\mathbb R^d)} \\
    &\quad \leq 2\|\langle k_2\rangle^{d/2+1}\langle k_1\rangle^{d/2+1} F(k,k_1,k_2,k_3)\|_{L^\infty_{k_1,k_2}(\mathbb R^{2d})},
\end{split}
\end{equation}
yields the bound:
\begin{equation}\label{eq: Qf031}
    |\tilde{\Qc}_{F1,<}(k)| \leq \frac{C_d}{2}\int_{\mathbb{R}^d} \frac{1}{|k-k_3|}\big\|\langle k_1\rangle^{d/2+1}\langle k_2\rangle^{d/2+1} F(k,k_1,k_2,k_3)\big\|_{L^\infty_{k_1,k_2}(\mathbb{R}^{2d})}\,dk_3.
\end{equation}
By a symmetric argument, we deduce:
\begin{equation}\label{eq: Qf032}
    |\tilde{\Qc}_{F1,>}(k)| \leq \frac{C_d}{2}\int_{\mathbb{R}^d} \frac{1}{|k-k_3|}\big\|\langle k_1\rangle^{d/2+1}\langle k_2\rangle^{d/2+1} F(k,k_1,k_2,k_3)\big\|_{L^\infty_{k_1,k_2}(\mathbb{R}^{2d})}\,dk_3,
\end{equation}
which immediately establishes \eqref{eq: Qf03}. Following analogous derivations, we also obtain the estimates \eqref{eq: Qf02} and \eqref{eq: Qf01}.

Next, we establish the uniform isotropic bound \eqref{eq:general_integral_reduced}. We note that on the support of $\delta(\Sigma)\delta(\Omega)$, the conservation of momentum implies:
\begin{equation}
    \langle k-k_3\rangle^{d} = \langle k_1-k_2\rangle^{d} \leq 4\max\{\langle k_1\rangle^d, \langle k_2\rangle^d\}.
\end{equation}
This allows us to bound the collision integral as:
\begin{equation}
\begin{split}
    |\tilde{\Qc}_{F1}(k)| &\leq 4\iiint_{\mathbb{R}^{3d}} \frac{1}{\langle k-k_3\rangle^{d}}\delta(\Sigma)\delta(\Omega)\chi(|k-k_3|>0) \\
    &\quad \times \max\{\langle k_1\rangle^d, \langle k_2\rangle^d\}|F(k,k_1,k_2,k_3)| \,dk_1 \,dk_2 \,dk_3.
\end{split}
\end{equation}
Following an identical geometric decomposition process to \eqref{eq: sim1}--\eqref{eq: Qf032}, we arrive at:
\begin{equation}
    |\tilde{\Qc}_{F1}(k)| \leq 4C_d\int_{\mathbb{R}^d} \frac{1}{|k-k_3|\langle k-k_3\rangle^d}\big\|\langle k_1\rangle^{d+1}\langle k_2\rangle^{d+1} F(k,k_1,k_2,k_3)\big\|_{L^\infty_{k_1,k_2}(\mathbb{R}^{2d})}\,dk_3,
\end{equation}
which implies \eqref{eq:general_integral_reduced} and concludes the proof.
\end{proof}

\section{Proof of the Main Theorem}
\begin{proof}[Proof of Theorem~\ref{thm: LWP}] 
For notational convenience, we define the constant
\begin{equation}
    E := 2\|f_0\|_{L^2_{22+5d}\cap L^\infty_{12+4d}}.
\end{equation}
Taking the weight index $a=2(d+1)$ in~\eqref{eq: DE afn+1}, we observe that the difference $(|k|^2+1)^{d+1}\bigl(f_{n+2}(t)-f_{n+1}(t)\bigr)$ satisfies the evolution equation
\begin{equation}
\begin{split}
   \partial_t\Big[(|k|^2+1)^{d+1}\bigl(f_{n+2}(t)-f_{n+1}(t)\bigr)\Big] 
   &= \left(\Qc_{n+1}(t)+\tilde{\Qc}_{n+1}(t)+\Rc_{n+1}(t)\right)(|k|^2+1)^{d+1}\bigl(f_{n+2}(t)-f_{n+1}(t)\bigr) \\
   &\quad +\Fc_n(t)(|k|^2+1)^{d+1}\bigl(f_{n+1}(t)-f_n(t)\bigr),
\end{split}
\end{equation}
where the forcing operator $\Fc_n(t)$ is given in~\eqref{eq: Fcnh}. Applying Duhamel's principle, and utilizing the uniform bound~\eqref{eq: uniform bound fn} alongside Corollary~\ref{cor: est tUn2h} and Lemma~\ref{lem: est Fn}, we estimate the $L^2$-norm of the difference:
\begin{equation}
\begin{split}
    \|f_{n+2}(t)-f_{n+1}(t)\|_{2,2d+2}
    &\leq \int_0^t \left\| U_{n,2}(t,s)\Fc_{n}(s) (|k|^2+1)^{d+1}\bigl(f_{n+1}(s)-f_n(s)\bigr)\right\|_2 ds\\
    &\leq \int_0^t \exp\left((t-s)\tilde C_aE^2\right) C_{F}E^2 \|f_{n+1}(s)-f_{n}(s)\|_{2,2d+2} ds\\
    &\leq \left( \frac{C_F}{\tilde C_a} \left(e^{t\tilde C_aE^2}-1\right) \right) \sup_{s\in [0,t]} \|f_{n+1}(s)-f_{n}(s)\|_{2,2d+2}.
\end{split}
\end{equation}
Taking the supremum over $t \in [0,T]$, we ensure a strict contraction mapping,
\begin{equation}
    \sup_{t \in [0,T]} \|f_{n+2}(t)-f_{n+1}(t)\|_{2,2d+2} \leq \frac{1}{2} \sup_{t \in [0,T]} \|f_{n+1}(t)-f_{n}(t)\|_{2,2d+2},
\end{equation}
provided that the lifespan $T$ is chosen sufficiently small such that $\frac{C_F}{\tilde C_a} (e^{T\tilde C_aE^2}-1) \leq \frac{1}{2}$. This mathematically imposes the restriction
\begin{equation}\label{eq: final_T_choice}
    T \leq \min\left\{ \frac{1}{\tilde C_a E^2} \ln\left( 1 + \frac{\tilde C_a}{2C_F} \right), T_1 \right\},
\end{equation}
where $T_1$ is defined in~\eqref{eq: def T1}. This contraction guarantees that the sequence $\{f_n(t)\}$ is convergent in the Banach space $C([0,T]; L^2_{2d+2}(\mathbb{R}^d))$. Furthermore, by H\"older's inequality, this implies the strong convergence of $\{f_n(t)\}$ in $C([0,T]; L^1(\mathbb R^d))$, which completes the proof of existence and uniqueness of the local-in-time $L^1$ solution with the lifespan $T = T(\|f_0\|_{\infty, 12+4d}, \|f_0\|_{\infty,22+5d})$ given by~\eqref{eq: final_T_choice}. Finally, combining this strong convergence with the uniform bound~\eqref{eq: uniform bound fn} ensures that the resulting limit solution rigorously satisfies the propagation of bounds formulated in~\eqref{eq: stay}.\end{proof}

\section*{Acknowledgements}
The authors thank Yu Deng, Alexander O. Korotkevich, Xiao Ma, Sergey Nazarenko, and Jalal M. I. Shatah for fruitful discussions and valuable comments. The second author is also grateful to Andrew Hassell for helpful conversations. This work was supported by the Simons Collaboration on Wave Turbulence (Award No. 651459). The first author acknowledges funding from the National Science Foundation (Grant No. AWD023422). The second author acknowledges support from the Australian Research Council through an Australian Laureate Fellowship (Grant No. FL220100072).

\medskip
\appendix

\section{Auxiliary Estimates}\label{app: Cauchy}

\begin{proof}[Proof of Lemma~\ref{lem: basic kernel}] 
The estimate~\eqref{eq: est omega1-2} follows from 
\begin{equation}
    |\omega-\omega_3| = \frac{\big| |k_3+k_2-k_1|^2-|k_3|^2 \big|}{(\omega_3+\omega)(|k_3|+|k|)} \leq \frac{|k_1-k_2|(\frac{1}{10}|k_3|+2|k_3|)}{(\sqrt{\frac{11}{10}}+1)\omega_3\cdot \frac{19}{10}|k_3|} \leq \frac{3|k_2-k_1|}{5\omega_3}.
\end{equation}
The estimate~\eqref{eq: k-k3} follows from 
\begin{equation}
    \big| |k|-|k_3| \big| \leq |k-k_3|.
\end{equation}
The estimate~\eqref{eq: kk3 1/4} follows from 
\begin{equation}
\begin{aligned}
    \left| \frac{1}{|k|^{1/4}}-\frac{1}{|k_3|^{1/4}} \right| &= \frac{|\omega-\omega_3|}{|k|^{1/4}|k_3|^{1/4}(|k|^{1/4}+|k_3|^{1/4})} \\
    &\leq \frac{3|k_2-k_1|}{5\cdot (\frac{9}{10})^{1/4}\cdot (1+(\frac{9}{10})^{1/4})|k_3|^{5/4}} \\
    &\leq \frac{2|k_1-k_2|}{5|k_3|^{5/4}}.
\end{aligned}
\end{equation}
The estimate~\eqref{eq: 1-3-1-0 omega} follows from 
\begin{equation}
    \big| (\omega_1-\omega_3)^2-(\omega_1-\omega)^2 \big| = |\omega-\omega_3|(\omega+\omega_3-2\omega_1) \leq \big| |k|-|k_3| \big| \leq |k_2-k_1|.
\end{equation}
The estimate~\eqref{eq: omega dyz} follows from 
\begin{equation}
    \begin{aligned}
        \omega_{y-z}^2-(\omega_y-\omega_z)^2 &\geq \max\{|y|,|z|\}-\min\{|y|,|z|\}-|y|-|z|+2\omega_y\omega_z \\
        &= 2\min\{|y|,|z|\} + 2\omega_y\omega_z - 2\min\{|y|,|z|\} \\
        &\geq 2\min\{|y|,|z|\} |\omega_y-\omega_z|.
    \end{aligned}
\end{equation}
\end{proof}

\begin{lem}[Explicit Algebraic Lower Bound of the Relative Group Velocity]
\label{lem:group_velocity_lower_bound}
Let the dispersion relation be $\omega(k) = |k|^{1/2}$ for $k \in \mathbb{R}^2 \setminus \{0\}$. Define the group velocity field as $v(k) = \nabla \omega(k)$. Then, for any $k_2, k_3 \in \mathbb{R}^2 \setminus \{0\}$, the following sharp algebraic lower bound holds:
\begin{equation}
    |v(k_2) - v(k_3)| \geq \frac{1}{2 \sqrt{|k_2| |k_3|} \left( \sqrt{|k_2|} + \sqrt{|k_3|} \right)} |k_2 - k_3|.
\end{equation}
\end{lem}

\begin{proof}
By direct computation, the group velocity is given by $v(k) = \frac{k}{2|k|^{3/2}}$. To establish the lower bound, we compute the squared distance $|v(k_2) - v(k_3)|^2$ algebraically. For simplicity of notation, let $r_2 = |k_2|$ and $r_3 = |k_3|$. 

Expanding the squared Euclidean norm, we have:
\begin{equation*}
    |v(k_2) - v(k_3)|^2 = \left| \frac{k_2}{2 r_2^{3/2}} - \frac{k_3}{2 r_3^{3/2}} \right|^2 = \frac{1}{4} \left( \frac{1}{r_2} + \frac{1}{r_3} - \frac{2 (k_2 \cdot k_3)}{r_2^{3/2} r_3^{3/2}} \right).
\end{equation*}
Using the fundamental geometric identity $2(k_2 \cdot k_3) = r_2^2 + r_3^2 - |k_2 - k_3|^2$, we rewrite the inner product term and find a common denominator:
\begin{equation*}
\begin{aligned}
    |v(k_2) - v(k_3)|^2 &= \frac{1}{4} \left( \frac{r_2 + r_3}{r_2 r_3} - \frac{r_2^2 + r_3^2 - |k_2 - k_3|^2}{r_2^{3/2} r_3^{3/2}} \right) \\
    &= \frac{1}{4 r_2^{3/2} r_3^{3/2}} \Big( (r_2 + r_3)\sqrt{r_2 r_3} - (r_2^2 + r_3^2) + |k_2 - k_3|^2 \Big).
\end{aligned}
\end{equation*}

We now focus on the algebraic term $A := (r_2 + r_3)\sqrt{r_2 r_3} - (r_2^2 + r_3^2)$. By setting $x = \sqrt{r_2}$ and $y = \sqrt{r_3}$, $A$ can be factorized as:
\begin{equation*}
\begin{aligned}
    A &= (x^2 + y^2)xy - x^4 - y^4 \\
    &= x^3y + xy^3 - x^4 - y^4 \\
    &= -(x - y)^2 (x^2 + xy + y^2).
\end{aligned}
\end{equation*}
Substituting $r_2$ and $r_3$ back into the expression yields:
\begin{equation*}
    A = -(\sqrt{r_2} - \sqrt{r_3})^2 (r_2 + \sqrt{r_2 r_3} + r_3).
\end{equation*}

To relate $A$ to the metric distance $|k_2 - k_3|^2$, we observe that by the triangle inequality, $(r_2 - r_3)^2 \leq |k_2 - k_3|^2$. Since $r_2 - r_3 = (\sqrt{r_2} - \sqrt{r_3})(\sqrt{r_2} + \sqrt{r_3})$, we obtain the strict inequality:
\begin{equation*}
    (\sqrt{r_2} - \sqrt{r_3})^2 \leq \frac{|k_2 - k_3|^2}{(\sqrt{r_2} + \sqrt{r_3})^2}.
\end{equation*}
Because the polynomial factor $(r_2 + \sqrt{r_2 r_3} + r_3)$ is strictly positive, we deduce the lower bound for $A$:
\begin{equation*}
    A \geq - \frac{r_2 + \sqrt{r_2 r_3} + r_3}{(\sqrt{r_2} + \sqrt{r_3})^2} |k_2 - k_3|^2.
\end{equation*}

Adding $|k_2 - k_3|^2$ to $A$ allows us to gather the coefficients of the distance term:
\begin{equation*}
\begin{aligned}
    A + |k_2 - k_3|^2 &\geq |k_2 - k_3|^2 \left( 1 - \frac{r_2 + \sqrt{r_2 r_3} + r_3}{(\sqrt{r_2} + \sqrt{r_3})^2} \right) \\
    &= |k_2 - k_3|^2 \left( 1 - \frac{r_2 + \sqrt{r_2 r_3} + r_3}{r_2 + 2\sqrt{r_2 r_3} + r_3} \right) \\
    &= |k_2 - k_3|^2 \left( \frac{\sqrt{r_2 r_3}}{(\sqrt{r_2} + \sqrt{r_3})^2} \right).
\end{aligned}
\end{equation*}

Finally, substituting this crucial lower bound back into the equation for $|v(k_2) - v(k_3)|^2$, we obtain:
\begin{equation*}
\begin{aligned}
    |v(k_2) - v(k_3)|^2 &\geq \frac{1}{4 r_2^{3/2} r_3^{3/2}} \left( \frac{\sqrt{r_2 r_3}}{(\sqrt{r_2} + \sqrt{r_3})^2} |k_2 - k_3|^2 \right) \\
    &= \frac{1}{4 r_2 r_3 (\sqrt{r_2} + \sqrt{r_3})^2} |k_2 - k_3|^2.
\end{aligned}
\end{equation*}
Taking the square root on both sides concludes the proof.
\end{proof}

\begin{lem}\label{lem: Cauchy} 
The strong operator limit 
\begin{equation}
    s\text{-}\lim_{\epsilon \downarrow 0} U_{g,D,\epsilon}(t,s)
\end{equation}
exists on $L^2(\mathbb{R}^d)$ for all $0\leq s<t\leq T$.  
\end{lem}

\begin{proof}
It suffices to consider the case $s=0$. Let $h\in L^2(\mathbb R^d)$. To show that $\{U_{g,D,\epsilon}(t,s)h\}_{\epsilon > 0}$ satisfies the Cauchy criterion, we fix an arbitrary tolerance $\eta>0$. We can choose a cutoff $M=M(\eta)\geq 1$ such that the high-frequency tail satisfies 
\begin{equation}\label{eq: tail_eta}
    \|\chi(|k|\geq M)h\|_2 < \eta. 
\end{equation}
Define the truncated solution $h_{M,\epsilon}(t) := U_{g,D,\epsilon}(t,s) \big[\chi(|k|<M)h\big]$. Similar to the derivation of~\eqref{eq: DE afn+1}, the weighted profile $\langle k\rangle^a h_{M,\epsilon}(t)$ satisfies the differential equation
\begin{equation}\label{eq: cQgDe}
    \partial_t \big( \langle k\rangle^a h_{M,\epsilon}(t) \big) = \big( \mathcal{Q}_{g,D,\epsilon}(t) + \tilde{\mathcal{Q}}_{g,D,\epsilon}(t) + \mathcal{R}_{g,\epsilon}(t) \big) \big( \langle k\rangle^a h_{M,\epsilon}(t) \big),
\end{equation}
where the modified operators are explicitly given by
\begin{equation}
\begin{split}
    \tilde{\mathcal{Q}}_{g,D,\epsilon}(t) h(k) &= 2 \iiint_{\mathbb{R}^{3d}} e^{-\epsilon(|k|+|k_3|)}T_{k,k_1,k_2,k_3} \frac{\langle k\rangle^a - \langle k_3\rangle^a}{\langle k\rangle^{a/2}\langle k_3\rangle^{a/2}} \delta(\Sigma)\delta(\Omega) \\
    &\quad \times \chi_{2<3}\chi_{1<0} g_1(t) g_2(t) h_3 \,dk_1\,dk_2\, dk_3,
\end{split}
\end{equation}
and the remainder term is 
\begin{equation}
\begin{split}
    \mathcal{R}_{g,\epsilon}(t) h(k) &= 2 \iiint_{\mathbb{R}^{3d}} e^{-\epsilon(|k|+|k_3|)}T_{k,k_1,k_2,k_3} \frac{\langle k\rangle^a - \langle k_3\rangle^a}{\langle k_3\rangle^a} \delta(\Sigma)\delta(\Omega) \\
    &\quad \times \chi_{2<3}\chi_{1\geq 0} g_1(t) g_2(t) h_3 \,dk_1\,dk_2\, dk_3 \\
    &\quad + 2 \iiint_{\mathbb{R}^{3d}} e^{-\epsilon(|k|+|k_3|)}T_{k,k_1,k_2,k_3} \frac{(\langle k\rangle^{a/2} - \langle k_3\rangle^{a/2})^2 (\langle k\rangle^{a/2} + \langle k_3\rangle^{a/2})}{\langle k\rangle^{a/2}\langle k_3\rangle^a} \\
    &\quad \times \delta(\Sigma)\delta(\Omega) \chi_{2<3}\chi_{1<0} g_1(t) g_2(t) h_3 \,dk_1\,dk_2\, dk_3.
\end{split}
\end{equation}
Notice that $\mathcal{Q}_{g,D,\epsilon}(t) + \tilde{\mathcal{Q}}_{g,D,\epsilon}(t)$ is dissipative on $L^2(\mathbb R^d)$. Following an argument identical to Lemma~\ref{lem: eq: Rn bd}, the remainder is uniformly bounded:
\begin{equation}
    \| \Rc_{g,\epsilon}(t)\|_{2\to 2} \leq \tilde C_a \sup_{t\in [0,T]}\| g(t)\|_{\infty,{10+2d}}^2,
\end{equation}
for the constant $\tilde C_a>0$ introduced in~\eqref{eq: Rn bd est}. To absorb the polynomial weight arising from the difference of the exponential factors later, we specifically take $a=3$ in~\eqref{eq: cQgDe}, which yields the uniform bound
\begin{equation}\label{eq: hM_weighted_bound}
    \|h_{M,\epsilon}(t)\|_{2,3} \leq M^3\|h\|_{2}\exp\left(t\tilde C_a \sup_{t\in [0,T]}\| g(t)\|_{\infty,{10+2d}}^2\right), \quad \forall t \in [0,T].
\end{equation}
Now, for any $\epsilon_1, \epsilon_2 > 0$, we apply Duhamel's principle to the difference $h_{M,\epsilon_1}(t)-h_{M,\epsilon_2}(t)$. Utilizing the mean value theorem bound $|e^{-\epsilon_1 x} - e^{-\epsilon_2 x}| \leq |\epsilon_1-\epsilon_2|x$ to extract the factor $(|k|+|k_3|)$, and relying on the uniform bounds of the propagators and~\eqref{eq: hM_weighted_bound}, we obtain
\begin{equation}
\begin{split}
    \|h_{M,\epsilon_1}(t)-h_{M,\epsilon_2}(t)\|_2 
    &\leq C C_Q \int_0^t \left\| \frac{e^{-\epsilon_1(|k|+|k_3|)}-e^{-\epsilon_2(|k|+|k_3|)}}{|k|+|k_3|} \right\|_\infty \left(\sup_{\tau\in [0,T]}\| g(\tau)\|_{\infty,{10+2d}}^2 \right) \|h_{M,\epsilon_1}(s)\|_{2,3} ds\\
    &\leq C C_Q T M^3 |\epsilon_1-\epsilon_2| \|h\|_2 \left(\sup_{t\in [0,T]}\| g(t)\|_{\infty,{10+2d}}^2 \right) \exp\left(T\tilde C_a \sup_{t\in [0,T]}\| g(t)\|_{\infty,{10+2d}}^2\right).
\end{split}
\end{equation}
Thus, as $\epsilon_1, \epsilon_2 \downarrow 0$, the difference $\|h_{M,\epsilon_1}(t)-h_{M,\epsilon_2}(t)\|_2 \to 0$ uniformly on $[0,T]$. Finally, using the standard triangle inequality and the uniform $L^2$-boundedness of the exact propagator $U_{g,D,\epsilon}(t,0)$, we have
\begin{equation}
\begin{split}
    \|U_{g,D,\epsilon_1}(t,0)h - U_{g,D,\epsilon_2}(t,0)h\|_2 
    &\leq \|U_{g,D,\epsilon_1}(t,0) [\chi(|k|\geq M)h]\|_2 + \|h_{M,\epsilon_1}(t) - h_{M,\epsilon_2}(t)\|_2\\
    &+ \|U_{g,D,\epsilon_2}(t,0) [\chi(|k|\geq M)h]\|_2 \\
    &\leq 2 \eta + \|h_{M,\epsilon_1}(t)-h_{M,\epsilon_2}(t)\|_2.
\end{split}
\end{equation}
Since $\eta>0$ is arbitrary, we conclude that $\{U_{g,D,\epsilon}(t,0)h\}_{\epsilon\in (0,1)}$ is uniformly Cauchy in $C([0,T]; L^2(\mathbb R^d))$. This establishes the existence of the strong limit and completes the proof.
\end{proof}
\begin{lem}\label{lem: f_pm_bound}
On the support of the conservation distributions $\delta(\Sigma)\delta(\Omega)\varphi_{1<3}\varphi_{2<3}$, the estimate
\begin{equation}
    |f_\pm(k_j,k)-f_\pm(k_j,k_3)| \leq 2 |k_j|(|k_1|+|k_2|)
\end{equation}
holds for $j=1,2$.
\end{lem}

\begin{proof}
We demonstrate the bound for the case $j=1$. By the definition of the function $f_\pm$ and applying the reverse triangle inequality, we first estimate
\begin{equation}\label{eq: f_pm_step1}
\begin{split}
    |f_\pm(k_1,k)-f_\pm(k_1,k_3)| 
    &\leq |k_1 \cdot (k-k_3)| + |k_1| \big| |k|-|k_3| \big| \\
    &\leq |k_1 \cdot (k-k_3)| + |k_1| |k-k_3|.
\end{split}
\end{equation}
On the support of $\delta(\Sigma)$, momentum conservation implies $|k-k_3| = |k_2-k_1|$. Substituting this into~\eqref{eq: f_pm_step1} and applying the Cauchy-Schwarz inequality yields
\begin{equation}\label{eq: f_pm_step2}
\begin{split}
    |f_\pm(k_1,k)-f_\pm(k_1,k_3)| 
    &\leq |k_1 \cdot (k_2-k_1)| + |k_1| |k_2-k_1| \\
    &\leq 2|k_1| |k_2-k_1| \\
    &\leq 2|k_1| (|k_1|+|k_2|) .
\end{split}
\end{equation}
By symmetry, an identical argument establishes the desired estimate for the case $j=2$. This completes the proof.
\end{proof}
\begin{lem}\label{lem: C11 est}
On the support of the conservation distributions $\delta(\Sigma)\delta(\Omega)\varphi_{1<3}\varphi_{2<3}$, the estimate~\eqref{eq: d t T1 1} holds for some constant $C_{11}>0$.
\end{lem}

\begin{proof} 
By invoking the momentum conservation relation $k_3-k=k_1-k_2$ and the scale separation bounds~\eqref{eq: range k}, we can first bound the difference of the inverse fourth-root prefactors as follows:
\begin{equation}
\begin{split}
    \left| \frac{1}{|k_3|^{1/4}}-\frac{1}{|k|^{1/4}}\right| 
    &\leq \frac{\big| |k_3|-|k| \big|}{|k_3|^{1/4}|k|^{1/4}\big(|k_3|^{1/4}+|k|^{1/4}\big)\big(|k_3|^{1/2}+|k|^{1/2}\big)}\\
    &\leq \frac{|k_3-k|}{|k_3|^{1/4}|k|^{1/4}\big(|k_3|^{1/4}+|k|^{1/4}\big)\big(|k_3|^{1/2}+|k|^{1/2}\big)}\\
    &= \frac{|k_1-k_2|}{|k_3|^{1/4}|k|^{1/4}\big(|k_3|^{1/4}+|k|^{1/4}\big)\big(|k_3|^{1/2}+|k|^{1/2}\big)}\\
    &\leq \frac{C(|k_1|+|k_2|)}{|k|^{5/4}},
\end{split}
\end{equation}
for some universal constant $C>0$. Consequently, to establish the full estimate, it suffices to prove that the core polynomial part, defined as $\bar T_{k_1k,1}^{k_2k_3}:= -16\pi^2\big(|k||k_1||k_2||k_3|\big)^{1/4}\tilde T_{k_1k,1}^{k_2k_3}$, satisfies the bound
\begin{equation}\label{eq: est bar T1}
    | \bar T_{k_1k,1}^{k_2k_3}-\bar T_{k_1k,1}^{k_2k}|\leq \bar{C}_{11}(|k_1|^2+|k_2|^2)|k|^{3/2},
\end{equation}
for some constant $\bar C_{11}>0$. To this end, we algebraically expand the difference:
\begin{equation}\label{eq: bT 1}
\begin{split}
    \bar{T}_{k_1 k,1}^{k_2 k_3}-\bar{T}_{k_1 k,1}^{k_2 k} 
    &= - 2\omega^2 \Big[ \omega_2 (\omega_3-\omega) f_-(k_1,k) + \omega_1 \omega \big(f_-(k_2,k_3)-f_-(k_2,k)\big) \Big] \\
    &\quad - 2(\omega_3^2-\omega^2) \Big[ \omega \omega_2 f_+(k_1,k_3) + \omega_1 \omega_3 f_+(k,k_2)\Big] \\
    &\quad - 2\omega^2 \Big[ \omega \omega_2 \big(f_+(k_1,k_3)-f_+(k_1,k)\big) + \omega_1 (\omega_3-\omega) f_+(k,k_2)\Big] \\
    &\quad - \frac{2 \omega f_-(k_1,k)\big(f_-(k_2,k_3)-f_-(k_2,k)\big)}{\omega_1} + \frac{2 (\omega_3-\omega)f_+(k_1,k_3)f_+(k,k_2)}{\omega_1}\\
    &\quad +\frac{2\omega\big(f_+(k_1,k_3)-f_+(k_1,k)\big)f_+(k,k_2)}{\omega_1}.
\end{split}
\end{equation}
Applying the bounds from Lemma~\ref{lem: f_pm_bound} along with the estimates~\eqref{eq: est omega1-2}, \eqref{eq: k-k3}, and \eqref{eq: range k}, we deduce
\begin{equation}
\begin{split}
    |\bar{T}_{k_1 k,1}^{k_2 k_3}-\bar{T}_{k_1 k,1}^{k_2 k}| 
    &\leq 2\omega^2 \left[ \omega_2 \frac{|k_2-k_1|}{\omega} |f_-(k_1,k)| + 2\omega_1 \omega |k_2|(|k_1|+|k_2|) \right] \\
    &\quad + 2|k_1-k_2| \Big[ \omega \omega_2 |f_+(k_1,k_3)| + \omega_1 \omega_3 |f_+(k,k_2)|\Big] \\
    &\quad + 2\omega^2 \left[ 2\omega \omega_2 |k_1|(|k_1|+|k_2|) + \omega_1 \frac{|k_1-k_2|}{\omega} |f_+(k,k_2)|\right] \\
    &\quad + \frac{4 \omega |f_-(k_1,k)| |k_2|(|k_1|+|k_2|)}{\omega_1} + \frac{2|k_2-k_1| |f_+(k_1,k_3)| |f_+(k,k_2)|}{\omega\omega_1}\\
    &\quad +\frac{4\omega |k_1|(|k_1|+|k_2|) |f_+(k,k_2)|}{\omega_1}.
\end{split}
\end{equation}
Combining this with the pointwise estimates $|f_\pm(y,z)|\leq 2|y||z|$ and the triangle inequality $|k_1-k_2|\leq |k_1|+|k_2|$, we obtain
\begin{equation}
\begin{split}
    |\bar{T}_{k_1 k,1}^{k_2 k_3}-\bar{T}_{k_1 k,1}^{k_2 k}| 
    &\leq (|k_1|+|k_2|) \Bigg\{ 4|k_1|\omega_2|k|^{3/2} + 4\omega_1|k_2||k|^{3/2} + 4|k_1|\omega_2\omega|k_3| \\
    &\quad + 4\omega_1|k_2||k|\omega_3 + 4|k_1|\omega_2|k|^{3/2} + 4\omega_1|k_2||k|^{3/2} \\
    &\quad + 8\omega_1|k_2||k|^{3/2} + 8\omega_1|k_2||k_3|\omega + 8\omega_1|k_2||k|^{3/2}\Bigg\}.
\end{split}
\end{equation}
Finally, utilizing the bounds in~\eqref{eq: range k}, the terms inside the braces are bounded by $\mathcal{O}\big((|k_1|+|k_2|)|k|^{3/2}\big)$, which directly yields the target estimate~\eqref{eq: est bar T1}. This completes the proof.
\end{proof}
\begin{lem}\label{lem: C41 est}
On the support of the conservation distributions $\delta(\Sigma)\delta(\Omega)\varphi_{1<3}\varphi_{2<3}$, the estimate~\eqref{eq: d t T2 q} holds for some constant $C_{41}>0$.
\end{lem}

\begin{proof} 
Following an argument analogous to that in Lemma~\ref{lem: C11 est}, we can factor out the inverse fractional prefactors. The proof then rigorously reduces to establishing the corresponding bound for the core polynomial part. Specifically, defining $\bar T_{kk_1,1}^{k_2k_3} := -16\pi^2\big(|k||k_1||k_2||k_3|\big)^{1/4}\tilde T_{kk_1,1}^{k_2k_3}$, it is sufficient to show that
\begin{equation}\label{eq: est bar T4}
    | \bar T_{kk_1,1}^{k_2k_3}-\bar T_{kk_1,1}^{k_2k}|\leq \bar{C}_{41}(|k_1|^2+|k_2|^2)|k|^{3/2},
\end{equation}
for some universal constant $\bar C_{41}>0$. To this end, we algebraically expand the difference:
\begin{equation}\label{eq: bT 4}
\begin{split}
    \bar{T}_{k k_1,1}^{k_2 k_3}-\bar{T}_{k k_1,1}^{k_2 k} 
    &= - 2\omega^2 \Big[ \omega_2 (\omega_3-\omega) f_-(k_1,k) + \omega_1 \omega \big(f_-(k_2,k_3)-f_-(k_2,k)\big) \Big] \\
    &\quad - 2\omega^2 \Big[ \omega_1 (\omega_3-\omega )f_+(k,k_2) + \omega_2\omega( f_+(k_1,k_3)-f_+(k_1,k))\Big] \\
    &\quad - \frac{2 \omega f_-(k_1,k)\big(f_-(k_2,k_3)-f_-(k_2,k)\big)}{\omega_1} \\
    &\quad+\frac{2\omega\big(f_+(k_1,k_3)-f_+(k_1,k)\big)f_+(k,k_2)}{\omega_2}.
\end{split}
\end{equation}
Applying the bounds from Lemma~\ref{lem: f_pm_bound} along with the estimates~\eqref{eq: est omega1-2}, \eqref{eq: k-k3}, and \eqref{eq: range k}, we deduce
\begin{equation}
\begin{split}
    |\bar{T}_{k k_1,1}^{k_2 k_3}-\bar{T}_{k k_1,1}^{k_2 k}| 
    &\leq 2\omega^2 \left[ \omega_2 \frac{|k_2-k_1|}{\omega} |f_-(k_1,k)| + 2\omega_1 \omega |k_2|(|k_1|+|k_2|) \right] \\
    &\quad + 2\omega^2 \left[   \omega_1 \frac{|k_1-k_2|}{\omega} |f_+(k,k_2)|+2\omega \omega_2 |k_1|(|k_1|+|k_2|)\right] \\
    &\quad + \frac{4 \omega |f_-(k_1,k)| |k_2|(|k_1|+|k_2|)}{\omega_1} + \frac{4\omega |k_1|(|k_1|+|k_2|) |f_+(k,k_2)|}{\omega_2}.
\end{split}
\end{equation}
Combining this with the pointwise estimates $|f_\pm(y,z)|\leq 2|y||z|$ and the triangle inequality $|k_1-k_2|\leq |k_1|+|k_2|$, we obtain
\begin{equation}
\begin{split}
    |\bar{T}_{k k_1,1}^{k_2 k_3}-\bar{T}_{k k_1,1}^{k_2 k}| 
    &\leq (|k_1|+|k_2|) \Bigg\{ 4|k_1|\omega_2|k|^{3/2} + 4\omega_1|k_2||k|^{3/2} + 4\omega_1|k_2||k|^{3/2} \\
    &\quad + 4|k_1|\omega_2|k|^{3/2}  + 8\omega_1|k_2||k|^{3/2} + 8|k_1|\omega_2|k|^{3/2} \Bigg\}.
\end{split}
\end{equation}
Finally, utilizing the bounds in~\eqref{eq: range k}, the terms inside the braces are bounded by $\mathcal{O}\big((|k_1|+|k_2|)|k|^{3/2}\big)$, which directly yields the target estimate~\eqref{eq: est bar T4}. This completes the proof.
\end{proof}
\begin{lem}\label{lem: C21 est} On the support of the conservation distributions $\delta(\Sigma)\delta(\Omega)\varphi_{1<3}\varphi_{2<3}$, the estimate~\eqref{eq: d t T1 2} holds for some constant $C_{21}>0$.
\end{lem}
\begin{proof}Following an argument analogous to that in Lemma~\ref{lem: C11 est}, we can factor out the inverse fractional prefactors. The proof then rigorously reduces to establishing the corresponding bound for the core polynomial part. Specifically, defining $\bar T_{k_1k,2}^{k_2k_3} := -16\pi^2\big(|k||k_1||k_2||k_3|\big)^{1/4}\tilde T_{k_1k,2}^{k_2k_3}$, it is sufficient to show that
\begin{equation}\label{eq: est bar T2}
    | \bar T_{k_1k,1}^{k_2k_3}-\bar T_{k_1k,1}^{k_2k}|\leq \bar{C}_{21}(|k_1|^3+|k_2|^3)|k|,
\end{equation}
for some universal constant $\bar C_{21}>0$. To this end, we algebraically expand the difference:
    \begin{equation}\label{eq: bT 2}
\begin{split}
    \bar{T}_{k_1 k,2}^{k_2 k_3}- \bar{T}_{k_1 k,2}^{k_2 k}= & -12 |k_1| |k| |k_2| (|k_3|-|k|) \\
    &- 4\omega_1\omega \Big[ \omega_2 (\omega_3-\omega) f_-(k_1,k) + \omega_1 \omega (f_-(k_2,k_3)-f_-(k_2,k)) \Big] \\
    &+ 4\omega_1(\omega_3-\omega) \Big[ \omega \omega_2 f_+(k_1,k_3) + \omega_1 \omega_3 f_+(k,k_2)\Big] \\
     &+ 4\omega_1\omega \Big[ \omega \omega_2 (f_+(k_1,k_3) -f_+(k_1,k))+ \omega_1 (\omega_3-\omega) f_+(k,k_2)\Big] \\
    &\quad + f_+(k_1,k)(f_+(k_2,k_3) -f_+(k_2,k))+ (f_-(k_1,k_3)-f_-(k_1,k))f_-(k,k_2) \\
    &\quad -4 f_-(k_1,k)(f_-(k_2,k_3)-f_-(k_2,k)) \\
    &+  f_-(k_1,k)(f_-(k_2,k_3)-f_-(k_2,k))(1-\hat{k}_1\cdot \hat k) \\
    &\quad -4 (f_+(k_1,k_3)-f_+(k_1,k))f_+(k,k_2) \\
    &+  (f_+(k_1,k_3)-f_+(k_1,k))f_+(k,k_2)(1+\hat{k}_1\cdot \hat k_3).
\end{split}
\end{equation}
Applying the bounds from Lemma~\ref{lem: f_pm_bound} along with the estimates~\eqref{eq: est omega1-2}, \eqref{eq: k-k3}, and \eqref{eq: range k}, we deduce
\begin{equation}
    \begin{split}
        | \bar{T}_{k_1 k,2}^{k_2 k_3}- \bar{T}_{k_1 k,2}^{k_2 k}|\leq & 12 |k_1| |k| |k_2| |k_1-k_2| \\
    &+4\omega_1\omega\Big[ \omega_2\cdot\frac{|k_1-k_2|}{\omega} |f_-(k_1,k)| + \omega_1 \omega \cdot 2|k_2|(|k_1|+|k_2|) \Big] \\
    &+ 4\omega_1\frac{|k_1-k_2|}{\omega} \Big[ \omega \omega_2 f_+(k_1,k_3) + \omega_1 \omega_3 f_+(k,k_2)\Big] \\
     &+ 4\omega_1\omega \Big[ \omega \omega_2 \cdot 2|k_1|(|k_1|+|k_2|)+ \omega_1 \frac{|k_1-k_2|}{\omega} f_+(k,k_2)\Big] \\
    &\quad + f_+(k_1,k) \cdot 2|k_2|(|k_1|+|k_2|)+ 2|k_1|(|k_1|+|k_2|)\cdot |f_-(k,k_2)| \\
    &\quad +4 |f_-(k_1,k)|\cdot 2|k_2|(|k_1|+|k_2|) \\
    &+  |f_-(k_1,k)|\cdot 2|k_2|(|k_1|+|k_2|)(1-\hat{k}_1\cdot \hat k) +8|k_1|(|k_1|+|k_2|)f_+(k,k_2) \\
    &+  2|k_1|(|k_1|+|k_2|) \cdot f_+(k,k_2)(1+\hat{k}_1\cdot \hat k_3).
    \end{split}
\end{equation}
Combining this with the pointwise estimates $|f_\pm(y,z)|\leq 2|y||z|$ and the triangle inequality $|k_1-k_2|\leq |k_1|+|k_2|$, we obtain
\begin{equation}
    \begin{split}
         | \bar{T}_{k_1 k,2}^{k_2 k_3}- \bar{T}_{k_1 k,2}^{k_2 k}|\leq & (|k_1|+|k_2|)\Bigg\{12|k_1||k_2||k|+8|k_1|^{3/2}\omega_2|k|+8|k_1||k_2||k|\\
&+8|k_1|^{3/2}\omega_2|k_3|+8|k_1||k_2|\omega_3\omega+8|k_1|^{3/2}\omega_2|k|+8|k_1||k_2||k|\\
&+4|k_1||k_2||k|+4|k_1||k_2||k|+16|k_1||k_2||k|+8|k_1||k_2||k|\\
&+16|k_1||k_2||k|+8|k_1||k_2||k|\Bigg\}.
    \end{split}
\end{equation}
Finally, utilizing the bounds in~\eqref{eq: range k}, the terms inside the braces are bounded by $\mathcal{O}\big((|k_1|+|k_2|)^2|k|\big)$, which directly yields the target estimate~\eqref{eq: est bar T2}. This completes the proof.
\end{proof}
\begin{lem}\label{lem: C51 est} On the support of the conservation distributions $\delta(\Sigma)\delta(\Omega)\varphi_{1<3}\varphi_{2<3}$, the estimate~\eqref{eq: d t T2 2} holds for some constant $C_{51}>0$.
\end{lem}
\begin{proof}Following an argument analogous to that in Lemma~\ref{lem: C11 est}, we can factor out the inverse fractional prefactors. The proof then rigorously reduces to establishing the corresponding bound for the core polynomial part. Specifically, defining $\bar T_{kk_1,2}^{k_2k_3} := -16\pi^2\big(|k||k_1||k_2||k_3|\big)^{1/4}\tilde T_{kk_1,2}^{k_2k_3}$, it is sufficient to show that
\begin{equation}\label{eq: est bar T5}
    | \bar T_{kk_1,1}^{k_2k_3}-\bar T_{kk_1,1}^{k_2k}|\leq \bar{C}_{51}(|k_1|^3+|k_2|^3)|k|,
\end{equation}
for some universal constant $\bar C_{51}>0$. To this end, we algebraically expand the difference:
    \begin{equation}\label{eq: bT 5}
\begin{split}
    \bar{T}_{k k_1,2}^{k_2 k_3} - \bar{T}_{k k_1,2}^{k_2 k} = & -12 |k_1| |k| |k_2| (|k_3|-|k|) \\
    &\quad - 4\omega_1\omega \Big[ \omega_2 (\omega_3-\omega) f_-(k,k_1) + \omega \omega_1 (f_-(k_2,k_3)-f_-(k_2,k)) \Big]\\
    &\quad+ 4\omega\omega_2 \Big[ \omega_1 (\omega_3-\omega) f_+(k,k_2) + \omega \omega_2 (f_+(k_1,k_3)-f_+(k_1,k))\Big] \\
    &\quad + f_+(k,k_1)(f_+(k_2,k_3)-f_+(k_2,k))+f_-(k,k_2)(f_-(k_1,k_3) -f_-(k_1,k))\\
    &\quad -4 f_-(k,k_1)(f_-(k_2,k_3)-f_-(k_2,k)) \\
    &\quad +  f_-(k,k_1)(f_-(k_2,k_3)-f_-(k_2,k))(1-\hat{k}_1\cdot \hat k) \\
    &\quad -4f_+(k,k_2)(f_+(k_1,k_3)-f_+(k_1,k))\\
    &\quad +  f_+(k,k_2)(f_+(k_1,k_3)-f_+(k_1,k))(1+\hat{k}_2\cdot \hat k).
\end{split}
\end{equation}
Applying the bounds from Lemma~\ref{lem: f_pm_bound} along with the estimates~\eqref{eq: est omega1-2}, \eqref{eq: k-k3}, and \eqref{eq: range k}, we deduce
\begin{equation}
    \begin{split}
        | \bar{T}_{k k_1,2}^{k_2 k_3}- \bar{T}_{k k_1,2}^{k_2 k}|\leq & 12 |k_1| |k| |k_2| |k_1-k_2| \\
    &+4\omega_1\omega\Big[ \omega_2\cdot\frac{|k_1-k_2|}{\omega} |f_-(k_1,k)| + \omega_1 \omega \cdot 2|k_2|(|k_1|+|k_2|) \Big] \\
     &+ 4\omega\omega_2 \Big[  \omega_1 \frac{|k_1-k_2|}{\omega} f_+(k,k_2)+\omega \omega_2 \cdot 2|k_1|(|k_1|+|k_2|)\Big] \\
    &\quad + f_+(k_1,k) \cdot 2|k_2|(|k_1|+|k_2|)+ 2|k_1|(|k_1|+|k_2|)\cdot |f_-(k,k_2)| \\
    &\quad +4 |f_-(k_1,k)|\cdot 2|k_2|(|k_1|+|k_2|) \\
    &+  |f_-(k_1,k)|\cdot 2|k_2|(|k_1|+|k_2|)(1-\hat{k}_1\cdot \hat k) +8|k_1|(|k_1|+|k_2|)f_+(k,k_2) \\
    &+  2|k_1|(|k_1|+|k_2|) \cdot f_+(k,k_2)(1+\hat{k}_1\cdot \hat k_3).
    \end{split}
\end{equation}
Combining this with the pointwise estimates $|f_\pm(y,z)|\leq 2|y||z|$ and the triangle inequality $|k_1-k_2|\leq |k_1|+|k_2|$, we obtain
\begin{equation}
    \begin{split}
         | \bar{T}_{k k_1,2}^{k_2 k_3}- \bar{T}_{k k_1,2}^{k_2 k}|\leq & (|k_1|+|k_2|)\Bigg\{12|k_1||k_2||k|+8|k_1|^{3/2}\omega_2|k|+8|k_1||k_2||k|\\
&+8|k_2|^{3/2}\omega_1|k_3|+8|k_1||k_2||k|+4|k_1||k_2||k|+4|k_1||k_2||k|\\
&+16|k_1||k_2||k|+8|k_1||k_2||k|+16|k_1||k_2||k|\\
&+8|k_1||k_2||k|\Bigg\}.
    \end{split}
\end{equation}
Finally, utilizing the bounds in~\eqref{eq: range k}, the terms inside the braces are bounded by $\mathcal{O}\big((|k_1|+|k_2|)^2|k|\big)$, which directly yields the target estimate~\eqref{eq: est bar T5}. This completes the proof.
\end{proof}
\begin{lem}\label{lem: asymptotic denominator}
On the support of $\delta(\Sigma)\delta(\Omega)\varphi_{1<3}\varphi_{2<3}$, assuming the dispersion relation $\omega^2 = |k|$, we have the asymptotic expansion
\begin{equation}\label{eq: expansion denominator}
    \frac{1}{|k_1+k|-(\omega_1+\omega)^2} = \frac{1}{-2\omega_1\omega} - \frac{|k_1|-k_1\cdot \hat k}{(-2\omega_1\omega)^2} + \mathcal{O}\left(\frac{1}{\omega^3}\right),
\end{equation}
with the error term satisfying the rigorous bound
\begin{equation}
   \left| \mathcal{O}\left(\frac{1}{\omega^3}\right)\right| \leq \frac{C_r(\omega_1+\omega_2)}{\omega^3},
\end{equation}
for some universal constant $C_r>0$.
\end{lem}

\begin{proof}
Using the dispersion relation $\omega_j^2 = |k_j|$ and $\omega^2 = |k|$, the subtracted term in the denominator expands precisely as
\begin{equation}\label{eq: expansion term1_new}
    (\omega_1+\omega)^2 = |k_1| + |k| + 2\omega_1\omega.
\end{equation}
For the term $|k_1+k|$, due to the scale separation $\varphi_{1<3}$ which enforces $|k_1| \ll |k|$, we perform a Taylor expansion of the square root. Writing $|k_1+k| = |k|\sqrt{1 + \frac{2k_1\cdot k + |k_1|^2}{|k|^2}}$, we have
\begin{equation}\label{eq: expansion term2_new}
    |k_1+k| = |k| + k_1\cdot \hat{k} + R_1,
\end{equation}
where the remainder $R_1$ is bounded by $|R_1| \leq C \frac{|k_1|^2}{|k|} = C \frac{\omega_1^4}{\omega^2}$ for some constant $C>0$. Subtracting \eqref{eq: expansion term1_new} from \eqref{eq: expansion term2_new}, the full denominator $D := |k_1+k|-(\omega_1+\omega)^2$ becomes
\begin{equation}
    D = -2\omega_1\omega - \big(|k_1| - k_1\cdot \hat{k}\big) + R_1 = -2\omega_1\omega \big(1 + x \big),
\end{equation}
where we define the dimensionless parameter
\begin{equation}
    x := \frac{|k_1| - k_1\cdot \hat{k}}{2\omega_1\omega} - \frac{R_1}{2\omega_1\omega}.
\end{equation}
We now estimate the magnitude of $x$. Since $\big||k_1| - k_1\cdot \hat{k}\big| \leq 2|k_1| = 2\omega_1^2$, the first term in $x$ is bounded by $\frac{\omega_1}{\omega}$. The second term is bounded by $\frac{C\omega_1^4/\omega^2}{2\omega_1\omega} \leq C\frac{\omega_1^3}{\omega^3}$. Thus, under the scale separation regime, $|x| \leq C' \frac{\omega_1}{\omega} \ll 1$.

Inverting $D$ using the geometric series expansion $\frac{1}{1+x} = 1 - x + R_2$, with the quadratic remainder bounded by $|R_2| \leq C x^2 \leq C \frac{\omega_1^2}{\omega^2}$, we obtain
\begin{equation}
\begin{split}
    \frac{1}{D} &= \frac{1}{-2\omega_1\omega} \big( 1 - x + R_2 \big) \\
    &= \frac{1}{-2\omega_1\omega} - \frac{|k_1| - k_1\cdot \hat{k}}{(-2\omega_1\omega)^2} + \frac{R_1}{(-2\omega_1\omega)^2} + \frac{R_2}{-2\omega_1\omega}.
\end{split}
\end{equation}
The first two terms exactly yield the leading order and the first correction in \eqref{eq: expansion denominator}. It remains to bound the overall error term $E := \frac{R_1}{4\omega_1^2\omega^2} - \frac{R_2}{2\omega_1\omega}$. Plugging in the bounds for $R_1$ and $R_2$, we compute
\begin{equation}
\begin{split}
    |E| &\leq \frac{C \omega_1^4/\omega^2}{4\omega_1^2\omega^2} + \frac{C \omega_1^2/\omega^2}{2\omega_1\omega} \\
    &\leq C \frac{\omega_1^2}{\omega^4} + C \frac{\omega_1}{\omega^3} \\
    &\leq \tilde{C} \frac{\omega_1}{\omega^3},
\end{split}
\end{equation}
where we used the fact that $\frac{\omega_1}{\omega} \ll 1$ to absorb the higher-order term. Finally, noting that trivially $\omega_1 \leq \omega_1 + \omega_2$, we conclude that the error is bounded by
\begin{equation}
    |E| \leq \frac{C_r (\omega_1+\omega_2)}{\omega^3}.
\end{equation}
This rigorously justifies the asymptotic expansion and its associated remainder bound, completing the proof.
\end{proof}

\begin{proof}[Proof of Lemma~\ref{lem: T3}]Defining the bracket
\begin{equation}
    \mathcal{B}_1 := \frac{4(\omega_1 + \omega)^2}{\omega_{k_1+k}^2 - (\omega_1 + \omega)^2} - \frac{4\omega^2}{-2\omega_1\omega} - \frac{8\omega_1\omega}{-2\omega_1\omega} - \frac{4\omega^2(|k_1|-k_1\cdot\hat k)}{(-2\omega_1\omega)^2}.
\end{equation}
Let $D := \omega_{k_1+k}^2 - (\omega_1+\omega)^2 = |k_1+k| - (\omega_1+\omega)^2$. Expanding the numerator of the first term in $\mathcal{B}_1$ as $4(\omega_1+\omega)^2 = 4\omega_1^2 + 8\omega_1\omega + 4\omega^2$, we distribute the denominator to write
\begin{equation}
    \frac{4(\omega_1+\omega)^2}{D} = \frac{4\omega_1^2}{D} + \frac{8\omega_1\omega}{D} + \frac{4\omega^2}{D}.
\end{equation}
By invoking Lemma~\ref{lem: asymptotic denominator}, we substitute the asymptotic expansion $D^{-1} = \frac{1}{-2\omega_1\omega} - \frac{|k_1|-k_1\cdot\hat k}{(-2\omega_1\omega)^2} + E$ into the $\omega^2$ and $\omega_1\omega$ terms, where the remainder is bounded by $|E| \leq C_r(\omega_1+\omega_2)/\omega^3$. This yields exact algebraic cancellations with the subtracted terms in $\mathcal{B}_1$:
\begin{equation}\label{eq: bracket B simplified}
\begin{split}
    \mathcal{B}_1 &= \frac{4\omega_1^2}{D} + 8\omega_1\omega \left( \frac{1}{-2\omega_1\omega} - \frac{|k_1|-k_1\cdot\hat k}{(-2\omega_1\omega)^2} + E \right) + 4\omega^2 \left( \frac{1}{-2\omega_1\omega} - \frac{|k_1|-k_1\cdot\hat k}{(-2\omega_1\omega)^2} + E \right) \\
    &\quad - \frac{4\omega^2}{-2\omega_1\omega} - \frac{8\omega_1\omega}{-2\omega_1\omega} - \frac{4\omega^2(|k_1|-k_1\cdot\hat k)}{(-2\omega_1\omega)^2} \\
    &= \frac{4\omega_1^2}{D} - \frac{8\omega_1\omega(|k_1|-k_1\cdot\hat k)}{(-2\omega_1\omega)^2} + 8\omega_1\omega E + 4\omega^2 E.
\end{split}
\end{equation}
We now bound each of the four remaining terms in~\eqref{eq: bracket B simplified}. Recall that $|D^{-1}| \leq C/(\omega_1\omega)$ and $\big||k_1|-k_1\cdot\hat k\big| \leq 2|k_1| = 2\omega_1^2$. 
First, for the unexpanded term, $\big| \frac{4\omega_1^2}{D} \big| \leq C \frac{\omega_1^2}{\omega_1\omega} = C\frac{\omega_1}{\omega}$.
Second, for the difference term, $\big| \frac{8\omega_1\omega(|k_1|-k_1\cdot\hat k)}{(-2\omega_1\omega)^2} \big| \leq \frac{16\omega_1^3\omega}{4\omega_1^2\omega^2} = 4\frac{\omega_1}{\omega}$.
Third, the main error term is bounded by $|4\omega^2 E| \leq 4\omega^2 \frac{C_r(\omega_1+\omega_2)}{\omega^3} \leq \tilde{C}\frac{\omega_1+\omega_2}{\omega}$.
Lastly, the higher-order error $|8\omega_1\omega E| \leq C \frac{\omega_1(\omega_1+\omega_2)}{\omega^2}$ is strictly subsumed by the previous term since $\omega_1 \ll \omega$.
Summing these bounds and recalling $\omega_j = |k_j|^{1/2}$, we obtain the critical estimate for the bracket:
\begin{equation}\label{eq: bracket B final bound}
    |\mathcal{B}_1| \leq C \frac{\omega_1+\omega_2}{\omega} = C \frac{|k_1|^{1/2}+|k_2|^{1/2}}{|k|^{1/2}}.
\end{equation}

Next, we address the prefactors. By definition, $|f_-(y,z)| \leq 2|y||z|$. Thus, $|f_-(k_1,k)| \leq 2|k_1||k|$ and $|f_-(k_2,k_3)| \leq 2|k_2||k_3|$. On the support of the conservation laws, scale separation ensures $|k_3| \sim |k|$, which implies $|f_-(k_2,k_3)| \leq C|k_2||k|$.
Multiplying the prefactors with the bound~\eqref{eq: bracket B final bound} and applying Young's inequality yields
\begin{equation}
\begin{split}
    \big| f_-(k_1,k)f_-(k_2,k_3) \mathcal{B}_1 \big| 
    &\leq C |k_1||k_2||k|^2 \frac{|k_1|^{1/2}+|k_2|^{1/2}}{|k|^{1/2}} \\
    &= C \big( |k_1|^{3/2}|k_2| + |k_1||k_2|^{3/2} \big) |k|^{3/2}\\
    &\leq C \big( |k_1|^{5/2} + |k_2|^{5/2} \big) |k|^{3/2}.
\end{split}
\end{equation}

Similarly, for the interaction involving the difference of frequencies, we define the analogous bracket
\begin{equation}
  \mathcal B_2 := \frac{4(\omega_1 - \omega_3)^2}{\omega_{k_1-k_3}^2 - (\omega_1 - \omega_3)^2} - \frac{4\omega_3^2}{2\omega_1\omega_3} + \frac{8\omega_1\omega_3}{2\omega_1\omega_3} - \frac{4\omega_3^2(|k_1|+k_1\cdot\hat k_3)}{(2\omega_1\omega_3)^2}.
\end{equation}
Following an identical asymptotic expansion, we deduce the corresponding estimate:
\begin{equation}
\begin{split}
    \big| f_+(k_1,k_3)f_+(k,k_2) \mathcal{B}_2 \big| 
    &\leq C |k_1||k_2||k|^2 \frac{|k_1|^{1/2}+|k_2|^{1/2}}{|k|^{1/2}} \\
    &= C \big( |k_1|^{3/2}|k_2| + |k_1||k_2|^{3/2} \big) |k|^{3/2}\\
    &\leq C \big( |k_1|^{5/2} + |k_2|^{5/2} \big) |k|^{3/2},
\end{split}
\end{equation}
for some uniform constant $C>0$. 

Finally, for the low-frequency difference term, we utilize the bounds~\eqref{eq: est omega1-2} and~\eqref{eq: omega dyz} to straightforwardly obtain
\begin{equation}
\begin{split}
    \left| \frac{4(\omega_1 - \omega_2)^2 f_+(k_1,k_2)f_+(k,k_3)}{\omega_{k_1-k_2}^2 - (\omega_1 - \omega_2)^2} \right| 
    &\leq C\frac{|k_1-k_2|}{\min\{\omega_1,\omega_2\}\omega_3}|k_1||k_2||k_3||k|\\
    &\leq \tilde C\big(|k_1|^{5/2}+|k_2|^{5/2}\big)|k|^{3/2},
\end{split}
\end{equation}
for some constants $C, \tilde C > 0$. Combining all these robust estimates rigorously yields the ultimate bound~\eqref{eq: last est T3}.
\end{proof}
\begin{proof}[Proof of Lemma~\ref{lem: T6}] This follows from an argument analogous to that in Lemma~\ref{lem: T3}.\end{proof}
\begin{lem}\label{lem: Zvt}
For any transverse vector $\vt\in \mathbb R^{d-1}$ and for each index $j \in \{1,2,3\}$, the cardinality of the resonance root set satisfies $|\mathcal{Z}_{\vt}^{(j)}|\leq 8$.
\end{lem}

\begin{proof}
We present the proof for $j=3$; the cases for $j=1,2$ follow by identical arguments. Let $c=\omega-\omega_3$ and $p=k-k_3$. The resonance condition $\Delta\omega(k,k_2,k_3)=0$ can then be rewritten as 
\begin{equation}
    c+\sqrt{|k_2-p|}=\omega_2.
\end{equation}
Since the dispersion relation implies $\omega_2^2 = |k_2|$, squaring both sides gives
\begin{equation}
    c^2+2c\sqrt{|k_2-p|}+|k_2-p|=|k_2|,
\end{equation}
which we rearrange to isolate the remaining radical:
\begin{equation}
   2c\sqrt{|k_2-p|}=|k_2|-c^2-|k_2-p|.
\end{equation}
Squaring both sides again yields
\begin{equation}
   4c^2|k_2-p|=|k_2|^2+c^4+|k_2-p|^2-2c^2|k_2|-2|k_2||k_2-p|+2c^2|k_2-p|,
\end{equation}
that is, 
\begin{equation}
   (2c^2+2|k_2|)|k_2-p|=|k_2|^2+c^4+|k_2-p|^2-2c^2|k_2|.
\end{equation}
To eliminate the absolute value terms appearing linearly, we square both sides a third time:
\begin{equation}
   (2c^2+2|k_2|)^2|k_2-p|^2=(|k_2|^2+c^4+|k_2-p|^2)^2+4c^4|k_2|^2-4c^2|k_2|(|k_2|^2+c^4+|k_2-p|^2).
\end{equation}
Rearranging this to isolate the single $|k_2|$ term gives
\begin{equation}
\begin{split}
   &\left( 8c^2|k_2-p|^2+4c^2(|k_2|^2+c^4+|k_2-p|^2)\right) |k_2| \\
   &\quad =(|k_2|^2+c^4+|k_2-p|^2)^2+4c^4|k_2|^2 -(4c^4+4|k_2|^2)|k_2-p|^2.
\end{split}
\end{equation}
Squaring both sides one final time eliminates all non-integer powers. Recall the local parameterization $k_2 = se_j + \vt$, which ensures that $|k_2|^2 = s^2 + |\vt|^2$ and $|k_2-p|^2$ are explicitly polynomials in the 1D variable $s$ of degree $2$. 

Tracing the algebraic degrees through the final squared equation reveals that $\Delta \omega(k,k_2,k_3)=0$ is equivalent to a polynomial equation in $s$ with a degree smaller than or equal to $8$. By the Fundamental Theorem of Algebra, there are at most $8$ distinct real solutions. This yields the uniform bound $|\mathcal{Z}_{\vt}^{(3)}|\leq 8$. 

By symmetry, analogous algebraic reductions apply for $j=1, 2$, yielding $|\mathcal{Z}_{\vt}^{(j)}|\leq 8$. This concludes the proof.
\end{proof}

\medskip

\bibliographystyle{abbrv} 

\end{document}